\documentclass[11pt,a4paper]{article}
\usepackage{amsfonts,latexsym,amsmath,amssymb,bbm,amscd,amsthm}
\usepackage{mathrsfs}
\usepackage{bm}
\usepackage{appendix}
\usepackage{color}
\usepackage[top=1.2in, bottom=1.2in, left=1.25in, right=1.25in]{geometry}
\usepackage{cite}
\usepackage{cases}
\usepackage[utf8]{inputenc}
\usepackage{graphicx,float,enumerate}
\usepackage[T1]{fontenc}
\usepackage{authblk}
\numberwithin{equation}{section}
\usepackage[colorlinks=true,citecolor=red,linkcolor=blue,urlcolor=RubineRed,pdfpagetransition=Blinds,pdftoolbar=false,pdfmenubar=false]{hyperref}

\def\varep{\varepsilon}
\def\Re{\mathop{\rm Re}}

\def\N{\mathbb{N}}
\def\Z{\mathbb{Z}}

\def\R{\mathbb{R}}
\def\epsilon{\varepsilon}

\newcommand{\be}{\begin{equation}}
\newcommand{\ee}{\end{equation}}
\newcommand{\baa}{\begin{array}}
\newcommand{\eaa}{\end{array}}

\newtheorem{theorem}{Theorem}[section]
\newtheorem{lemma}[theorem]{Lemma}
\newtheorem{proposition}[theorem]{Proposition}
\newtheorem{corollary}[theorem]{Corollary}
\newtheorem{definition}[theorem]{Definition}
\newtheorem{remark}[theorem]{Remark}

\newcommand{\blue}{\color{blue}}

\allowdisplaybreaks

\title{\bf{Propagation phenomena in periodic patchy landscapes with interface conditions}}
\author{Fran\c cois Hamel\thanks{Aix-Marseille Univ, CNRS, I2M, Marseille, France (\texttt{francois.hamel@univ-amu.fr}). This work has received funding from Excellence Initiative of Aix-Marseille Universit\'e~-~A*MIDEX, a French ``Investissements d'Avenir'' programme, and from the French ANR ReaCh (ANR-23-CE40-0023-02) project.}, \  Frithjof Lutscher\thanks{University of Ottawa, Department of Mathematics and Statistics \& Department of Biology, Ottawa, Canada (\texttt{flutsche@uottawa.ca}).} \ and Mingmin Zhang\thanks{Aix-Marseille Univ, CNRS, I2M, Marseille, France, and School of Mathematical Sciences, University of Science and Technology of China, Hefei, Anhui 230026, China (\texttt{mingmin.zhang.math@gmail.com}).}}

\date{}

\begin{document}
	
\maketitle
\vspace{-0.5cm}
\begin{center}
{\it In memory of Pavol Brunovsk\'y, with admiration for a great mathematician}
\end{center}
\vskip 0.5cm

\begin{abstract} 
\noindent{}This paper is concerned with a model for the dynamics of a single species in a one-dimensional heterogeneous environment. The environment consists of two kinds of patches, which are periodically alternately arranged along the spatial axis. We first establish the well-posedness for the Cauchy problem. Next, we give  existence and uniqueness results for the positive steady state and we analyze the long-time behavior of the solutions to the evolution problem. Afterwards, based on dynamical systems methods, we investigate the spreading properties  and the existence of pulsating traveling waves in the positive and negative directions. It is shown that the asymptotic spreading speed,  $c^*$, exists and coincides with the minimal wave speed of  pulsating traveling waves in positive and negative directions. In particular, we give a variational formula for $c^*$ by using the principal eigenvalues of certain linear periodic eigenvalue problems.
\vskip2pt\noindent{\small{\it{The manuscript has no associated data.}}}
\end{abstract}


\section{Introduction}

Reaction-diffusion equations of the type
\begin{align}
\label{RD}
u_t=\Delta u+f(u),~~ t>0,~x\in\mathbb{R}^N,
\end{align}
have been introduced in the pioneering works of Fisher \cite{Fisher1937} and Kolmogorov, Petrovsky and Piskunov \cite{KPP1937}. They are motivated by population genetics and aim at throwing light on the spatial spread of advantageous genetic features. The nonlinear reaction term considered there is that of logistic growth. Archetypes of such nonlinearities are $f(u)=u(1-u)$ or extensions like $f(u)=u(1-u^2)$, which are sometimes also called monostable since they have exactly one stable nonnegative steady state. Skellam \cite{Skellam1951} then investigated this type of models in order to study spatial propagation of species and proposed quantitative explanations for the spread of muskrats throughout Europe at the beginning of the 20th century.

Since these celebrated works, a vast mathematical literature has been devoted to the homogeneous equation \eqref{RD}. It is of particular interest to investigate the structure of  traveling front solutions and their stability, propagation or invasion and spreading properties. The former is related to the well-known result that this equation has a family of planar traveling fronts of the form $u(t,x)=U(x\cdot e-ct)$ for any given vector of unit norm $e\in\mathbb{S}^{N-1}$, which is the direction of propagation. Here, $c>0$ is the constant speed of the front and $U:\mathbb{R}\to (0,1)$ is the wave profile. It was proved in \cite{KPP1937} that, under KPP assumptions, $f(0)=f(1)=0$ and $0<f(s)\le f'(0)s$ in $s\in(0,1)$, there exists a threshold speed $c^*=2\sqrt{f'(0)}>0$ such that no fronts exist for $c<c^*$, while for each $c\ge c^*$, there is a unique (up to shift in space or time variables) planar front of the type $U(x\cdot e-ct)$. Such fronts are stable with respect to some natural classes of perturbations; see, for instance, \cite{Bramson1983,HNRR13,KPP1937,Lau1985,Sattinger1976,U1978}. Many papers were also dedicated to such planar fronts for other types of nonlinearities $f(\cdot)$, for example the bistable and combustion type; see, e.g., \cite{AW1975,BNS1985,Fife1979,FM1977,Kanel1961}. 

While traveling waves are interesting mathematical objects, they do not necessarily represent biologically realistic scenarios. For questions of biological invasions, it is more realistic to study how a locally introduced population would spread. Corresponding mathematical invasion and spreading properties for model (\ref{RD}) were established by Aronson and Weinberger \cite{AW1978}. Under the same assumptions on $f$ as in the previous paragraph, they proved that, starting with a nonnegative, compactly supported, continuous function $u_0\ge0$,  $u_0\not\equiv 0$, the solution~$u(t,x)$ of~\eqref{RD} spreads with speed $c^*$ in all directions for large times. More precisely, $\max_{|x|\le ct}|u(t,x)-1|\to 0$ as $t\to+\infty$ for each $c\in [0,c^*)$, and $\max_{|x|\ge ct}u(t,x)\to 0$ for each $c>c^*$.

Most landscapes are not homogeneous, as model (\ref{RD}) implicitly assumes. Several possible generalizations of the equation to heterogeneous landscapes exist, for example 
\begin{align}
\label{RD-periodic}
u_t=\nabla\cdot(D(x)\nabla u) u+f(x,u),~~ t>0,~x\in\mathbb{R}^N,
\end{align}
in periodic media (by periodic, we mean that $D(\cdot+k)=D$ and $f(\cdot+k,s)=f(\cdot,s)$ for all $k\in L_1\Z\times\cdots\times L_N\Z$ and $s\in[0,1]$, where $L_1,\ldots,L_N$ are given positive real numbers). For such equations, standard traveling fronts do not exist in general. Instead, the notion of traveling fronts is replaced by the more general concept of pulsating fronts \cite{SKT1986}. If a (unique) periodic positive steady state $p(x)$ of \eqref{RD-periodic} exists, a pulsating traveling front connecting 0 and $p(x)$ is a solution of the type $u(t,x)=U(x\cdot e-ct,x)$ with $c\neq 0$ and $e$ a unit vector representing the direction of propagation, if the function $U:\mathbb{R}\times\mathbb{R}^N\to\mathbb{R}$ satisfies
\begin{align*}
\begin{cases}
U(-\infty,x)=p(x),~U(+\infty,x)=0~~\text{uniformly in}~ x\in\mathbb{R}^N,\\
U(s,\cdot)~\text{is periodic in}~\mathbb{R}^N~\text{for all}~s\in\mathbb{R}.
\end{cases}
\end{align*}
Moreover,  for every $x\in\mathbb{R}^N$, the function $t\mapsto u(t,x+ct e)$ is in general quasi-periodic.

Berestycki, Hamel and Roques \cite{BHR2005-1} gave a complete and rigorous mathematical analysis of the periodic heterogeneous model (\ref{RD-periodic}) in any space dimension. They required the coefficient functions $D$ and $f$ to be sufficiently smooth, yet some of their results are valid under reduced regularity assumptions. Solutions in~\cite{BHR2005-1} are still at least of class $C^{1,\rho}$ with respect to $x$ for all $\rho\in(0,1)$, while the solutions considered in the present paper are even not continuous at all points in general, or not of class $C^1$ even when they are continuous after a rescaling. In~\cite{BHR2005-1}, existence, uniqueness and stability results were established. A criterion for species persistence and the effects of fragmentation on species survival were studied. Furthermore, the same authors  studied the question of biological invasion and existence of  pulsating fronts and they proved a variational formula of the minimal speed of such pulsating fronts and then analyzed the dependency  of this speed on the heterogeneity of the medium \cite{BHR2005-2}.  We refer the reader to~\cite{BH2002,BHN2005,H2008,HR2011,HZ1995,LZ2007,LZ2010,W2002} for more results on the existence, uniqueness and qualitative results of pulsating traveling fronts. For some results in space-time periodic media, see, e.g., \cite{Nadin2009,NRX2005,NX2005}. We also refer to \cite{BH2002,BHN2005,DGM2014,E10,FZ2015,GH2012,HS2009,Shen2010,Xin1991,Xin1992,Xin2000} for the existence and qualitative results with other types of nonlinearities or various boundary conditions in periodic domains.

While equation (\ref{RD-periodic}) and the corresponding theory is mathematically elegant, it is very difficult if not impossible to apply the model to biological invasions since the data requirements of finding diffusion coefficients and growth rates for continuously changing landscape characteristics are too costly. Alternatively, landscape ecology views natural environments as patches of homogeneous habitats such as forests, grasslands, or marshes, possibly fragmented by natural or artificial barriers like rivers or roads. Each patch is relatively homogeneous within but significantly different from adjacent patches. Shigesada, Kawasaki and Teramoto \cite{SKT1986} used this perspective of a patchy landscape and proposed a heterogeneous extension of (\ref{RD}) with piecewise constant coefficient functions; see also \cite{SK1997}. For simplicity, they considered only two types of patches, arranged in a periodically alternating way. More specifically, they introduced the following equation
\begin{align}
\label{SKT-model}
u_t=(D(x)u_x)_x+u(\mu(x)-u),~~t>0,~x\in\mathbb{R}\setminus S,
\end{align} 
where $\mu(x)$ is interpreted as the intrinsic growth rate of the population, $D(x)$ is the diffusion coefficient, and $S$ is the set of all interfaces between all adjacent patches. These functions are piecewise constant, namely, for $m\in\mathbb{Z}$,
$$\left\{\baa{llll}
D(x)=d_1(>0), & \mu(x)=\mu_1, & ml-l_1<x<ml & \text{(in patches of type 1)},\vspace{3pt}\\
D(x)=d_2(>0), & \mu(x)=\mu_2, & ml<x<ml+l_2 & \text{(in patches of type 2),}\eaa\right.$$
with $l_1,l_2>0$ and $l=l_1+l_2$. Without loss of generality, one can assume that $\mu_1\ge \mu_2$. Furthermore, if one assumes that the medium is not unfavorable everywhere, then $\mu_1>0$, i.e., type-1 patches support population growth and are ``source'' patches in biological terms. However, $\mu_2$ can be negative, so that patches of type 2 do not support population growth and are ``sink'' patches in biological terms. 

The above model is not complete. At the boundaries or interfaces between two patches, matching conditions must be imposed. Shigesada, Kawasaki and Teramoto \cite{SKT1986} required continuity of density and flux, i.e., for $t>0$:
\begin{align*}
u(t,x^-)=u(t,x^+),~~
D(x^-)u_x(t,x^-)=D(x^+)u_x(t,x^+),
\end{align*}
for all $x=ml$ and $x=ml+l_2$ ($m=0,\pm1,\pm 2,...$). Superscripts $\pm$ denote one-sided limits from the right and the left, respectively. When we take $D(x)\equiv d$, $\mu(x)\equiv \mu>0$, problem \eqref{SKT-model} is reduced to the Fisher-KPP equation with threshold speed $2\sqrt{d\mu}$. We shall refer to this kind of model, i.e., the differential equations combined with interface matching conditions as \emph{patch model} or \emph{patch problem}. 

Shigesada and coauthors  obtained the invasion conditions  in terms of the sizes of patches, diffusivities and growth rates. They proved that the population spreads successfully  if and only if the invasion condition is satisfied. Moreover, when invasion occurs, the distribution of the population initially localized in a bounded area always evolves into a pulsating front. The velocity was calculated with the aid of the dispersion relation, based on linearization at low density \cite{SKT1986}. When diffusion is constant, the rigorous analytical results in \cite{BHR2005-1} apply to this patch model, but the case of discontinuous diffusion and interface matching conditions is not covered.

Recently, Maciel and Lutscher \cite{ML2013} introduced novel interface matching conditions, based on the work by Ovaskainen and Cornell \cite{OC2003}. The population flux is still continuous at an interface, but the density is not. We will explain these conditions in detail below. These matching conditions not only allow us to include patch preference data, which are frequently collected in the field, into reaction-diffusion models, but also remove some biologically unrealistic behavior that the model with the continuous-density interface conditions above shows (see \cite{ML2013} for a thorough discussion). A number of recent studies use this new framework to study questions of persistence and spread \cite{ML2015,AL2019b} and apply it to marine reserve design \cite{AL2019a}. All these studies show that the correct choice of interface conditions has a crucial effect on basic quantities  such as population persistence conditions and spread rates in periodic environments. Later on, Maciel and Lutscher \cite{ML2018} showed how different movement strategies for competing species in patchy landscapes can lead to different outcomes of competition. Maciel and coauthors found evolutionarily stable movement strategies in a two-patch landscape \cite{MCCL2020}. A model including a higher-dimensional version of related flux matching conditions combined with Robin boundary conditions between two complementary subsets of $\R^N$ has also been studied recently by Berestycki, Rossi and Tellini~\cite{BRT2021}, with an emphasis on the propagation in the directions along the interface.

Our paper is devoted to a rigorous analytical study of the periodic patch model with two types of alternately arranged patches in a one-dimensional habitat. The population may grow or decay, depending on patch type, and diffusivity may change between patches. The setting and assumptions will be made precise below. The aim of the present work is first to rigorously prove the well-posedness of this somewhat nonstandard  patch model starting with nonnegative continuous and bounded initial data. Then we investigate the long-time behavior and spatial dynamics of this type of model in the framework of a periodic environment  with monostable dynamics. We give a criterion for the  existence and uniqueness of a positive and bounded steady state. Furthermore, under the hypothesis that the species can persist, we shall prove the existence of an asymptotic spreading speed $c^*$ of the solution to the Cauchy problem and we show that this spreading speed coincides with the minimal speed for rightward and leftward pulsating fronts. Moreover, the asymptotic spreading speed $c^*$ can be characterized using a family of eigenvalues. To the best of our knowledge, the results that had previously been discussed only formally or observed numerically in~\cite{ML2013}, are proved rigorously here for the first time. 


\section{Model presentation and statement of the main results}


\subsection{The model and some equivalent formulations}

Our model is a joint generalization of the models in \cite{SKT1986} and \cite{ML2013}. We consider a patchy periodically alternating landscape consisting of two types of patches (say, type 1 and 2); see Figure \ref{Fig:Schematic}. Each patch is homogeneous within. We denote the length of patch type $i$ ($i=1,2$) by $l_i$ so that the period is $l=l_1+l_2$. Accordingly, the real line is divided into intervals of the form
$$I_n=[nl-l_1,nl+l_2],~~n\in\mathbb{Z},$$
each consisting of two adjacent patches. Such intervals were called ``tiles'' in \cite{Cobbold}.
\begin{figure}[H]
\centering
\includegraphics[scale=0.7]{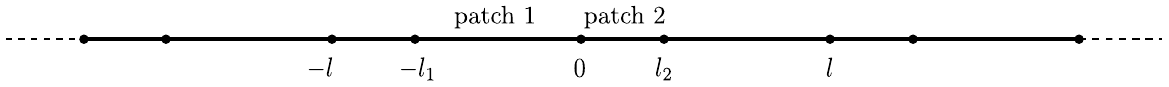}
\caption{\small Schematic figure of the one-dimensional periodic patchy environment.}
\label{Fig:Schematic}
\end{figure}
For $n\in \mathbb{Z}$, let $I_{1n}=(nl-l_1,nl)$ be the patches of type 1 and $I_{2n}=(nl,nl+l_2)$ be the patches of type 2. On each patch $I_{in}$, we denote by $v_{in}=v|_{I_{in}}$ the density of the population, by~$d_i$ the constant diffusion coefficients, and by $f_i$ the corresponding reaction nonlinearities. Our model then reads, for $n\in\mathbb{Z}$,
\begin{equation}
\label{their model-1}
\left\{\begin{aligned}
\frac{\partial v_{1n}}{\partial t}=d_1 \frac{\partial^2 v_{1n}}{\partial x^2}+f_1(v_{1n}),~~t>0,~x\in(nl-l_1,nl),\\
\frac{\partial v_{2n}}{\partial t}=d_2 \frac{\partial^2 v_{2n}}{\partial x^2}+f_2(v_{2n}),~~t>0,~x\in(nl,nl+l_2).
\end{aligned}\right.
\end{equation}
In~\eqref{their model-1}, the equations for $v_{in}=v|_{I_{in}}$ are set in the open intervals $(nl-l_1,nl)$ and $(nl,nl+l_2)$, but it will eventually turn out that the constructed solutions are such that the functions $v_{in}$ can be extended in $(0,+\infty)\times[nl-l_1,nl]$ or $(0,+\infty)\times[nl,nl+l_2]$ as $C^{1;2}_{t;x}$ functions, so that equations~\eqref{their model-1} will be satisfied in the closed intervals $[nl-l_1,nl]$ and $[nl,nl+l_2]$. The matching conditions for the population density and flux at the interfaces are given by
\begin{equation}
\label{their model-2}
\left\{\begin{aligned}
v_{1n}(t,x^-)&\!=\!kv_{2n}(t,x^+),\ \
\ ~~d_1 (v_{1n})_x(t,x^-)\!=\!d_2 (v_{2n})_x(t,x^+),~~~~~~~t>0,~x\!=\!nl,\vspace{3pt}\\
kv_{2n}(t,x^-)&\!=\!v_{1(n+1)}(t,x^+),\ \
d_2 (v_{2n})_x(t,x^-)\!=\!d_1  (v_{1(n+1)})_x(t,x^+),~
~t>0,~x\!=\!nl\!+\!l_2,
\end{aligned}\right.
\end{equation}
with parameter
\begin{align}
\label{k}
k=\frac{\alpha}{1-\alpha}\times\frac{d_2}{d_1}.
\end{align} 
Here, $\alpha\in(0,1)$ denotes the probability that an individual at the interface chooses to move to the adjacent patch of type~1, and $1-\alpha$ the probability that it moves to the patch of type~2. Individuals cannot stay at the interfaces. These interface conditions were derived in~\cite{OC2003} and studied in more detail in~\cite{ML2013}. They reflect the movement behavior of individuals when they come to the edge of a patch. With these interface conditions, the population density is discontinuous across a patch interface in the presence of patch preference and/or when the diffusion rates in these two kinds of patches are different. We point out that,  when  $k=1$ the model~\eqref{their model-1}--\eqref{their model-2} is reduced to the one in \cite{SKT1986}. Throughout the paper, we assume that the reaction terms~$f_i$~($i=1,2$) have the properties: 
\begin{equation}
\label{2.5'-existence}
f_i\in C^1(\R),\ f_i(0)=0,\hbox{ and there is $K_i>0$ such that $f_i\le0$ in $[K_i,+\infty)$}.
\end{equation}
Without loss of generality, we will consider type-1 patches as more favorable than type-2 patches, that is, $f'_1(0)\ge f'_2(0)$. In some statements, we will also assume that type-1 patches are ``source" patches, i.e., patches where the intrinsic growth rate of the population is positive ($f_1'(0)>0$), while type-2 patches may be source patches ($f'_2(0)>0$), or ``sink'' patches ($f_2'(0)<0$), or such that $f'_2(0)=0$. In order to investigate the long-time behavior and spatial dynamics, we will further assume in some statements that the functions $f_i$ satisfy the strong Fisher-KPP assumption:
\begin{equation}
\label{2.4'-uniqueness}\left\{\baa{l}
\displaystyle\hbox{the functions }s\mapsto\frac{f_i(s)}{s}~\text{are non-increasing in $s>0$ for $i=1,2$},\vspace{3pt}\\
\hbox{and decreasing in $s>0$ for at least one $i$}.\eaa\right.
\end{equation}
For instance,  $f_i$ satisfying hypotheses \eqref{2.5'-existence}--\eqref{2.4'-uniqueness} can be functions of the type  $f_i(s)=s(\mu_i-s)$.

Since the discontinuity in the densities at the interfaces makes the problem quite delicate to study, we rescale the densities in such a way that the matching conditions become continuous in the density. More precisely, we set $u_{1n}(t,x)=v_{1n}(t,x)$ for $t\ge 0$, $x\in(nl-l_1,nl)$ and $n\in\mathbb{Z}$, and $u_{2n}(t,x)=kv_{2n}(t,x)$  for $t\ge 0$, $x\in(nl,nl+l_2)$ and $n\in\mathbb{Z}$. Then
$u_{1n}$ satisfy the same equations as $v_{1n}$ with $\tilde f_1(s)=f_1(s)$, while $u_{2n}$ satisfy the equations of $v_{2n}$ with $f_2$ replaced by $\tilde f_2(s)=kf_2(s/k)$. We notice that $\tilde f_i$ ($i=1,2$) satisfy the same hypotheses as $f_i$ with $K_i$ replaced by $\tilde K_i$ where $\tilde K_1=K_1$ and $\tilde K_2 = kK_2$. Thanks to the change of variables, the interface conditions for the densities are now continuous; however, the flux interface conditions  become discontinuous, namely, 
$$\left\{\baa{lll}
u_{1n}(t,x^-)=u_{2n}(t,x^+), & \displaystyle \,d_1(u_{1n})_x(t,x^-)=\frac{d_2}{k}(u_{2n})_x(t,x^+), & t>0,~x=nl,\\
u_{2n}(t,x^-)=u_{1(n+1)}(t,x^+), & \displaystyle\frac{d_2}{k} (u_{2n})_x(t,x^-)=  d_1(u_{1(n+1)})_x(t,x^+), & t>0,~x=nl+l_2.\eaa\right.$$
We drop the tilde from hereon. Notice that the properties~\eqref{2.5'-existence} and~\eqref{2.4'-uniqueness} are invariant under this change. Putting it all together, we are led to  the following problem:
\begin{equation}
\label{m-1}
\left\{\begin{aligned}
\frac{\partial u_{1n}}{\partial t}&=d_1 \frac{\partial^2 u_{1n}}{\partial x^2}+ f_1(u_{1n}),~~t>0,~x\in(nl-l_1,nl),\\
\frac{\partial u_{2n}}{\partial t}&=d_2 \frac{\partial^2 u_{2n}}{\partial x^2}+ f_2(u_{2n}),~~t>0,~x\in(nl,nl+l_2),
\end{aligned}\right.
\end{equation}
with continuous density conditions and discontinuous flux interface conditions, 
\begin{equation}
\label{m-2}
\left\{\baa{lll}
u_{1n}(t,x^-)=u_{2n}(t,x^+), & \ \,(u_{1n})_x(t,x^-)=\sigma(u_{2n})_x(t,x^+), & t>0,~x=nl,\vspace{3pt}\\
u_{2n}(t,x^-)=u_{1(n+1)}(t,x^+), & \sigma (u_{2n})_x(t,x^-)=  (u_{1(n+1)})_x(t,x^+), & t>0,~x=nl+l_2,\eaa\right.
\end{equation}
in which, from \eqref{k}, we have
\begin{align}\label{sigma}
\sigma=\frac{d_2}{kd_1}=\frac{1-\alpha}{\alpha}>0.
\end{align}
We point out that, for the original model~\eqref{their model-1}--\eqref{their model-2} before rescaling, the total mass would be conserved if there were no growth terms, due to the continuity of the fluxes at the interfaces. For the problem~\eqref{m-1}--\eqref{m-2}, even without growth terms, the total mass is not conserved, due to the discontinuity of the fluxes at the interfaces, in general. However, one must keep in mind that the two formulations~\eqref{their model-1}--\eqref{their model-2} and~\eqref{m-1}--\eqref{m-2} are actually equivalent, and that the rescaling used in type-$2$-patches to make the density continuous at the interfaces is done only for mathematical convenience, in order to work with spatially continuous solutions.

From now on, we denote by
$$S_1=l\mathbb{Z}$$
the interface points between $(nl-l_1,nl)$ and $(nl,nl+l_2)$, and by
$$S_2=\{s+l_2:s\in l\mathbb{Z}\}$$
the interface points between $(nl,nl+l_2)$ and $(nl+l_2,(n+1)l)$. Therefore, $S=S_1\cup S_2$ represents all the interface points in $\mathbb{R}$. For convenience of our analysis, by setting $u(t,x)=u_{1n}(t,x)$ for $t>0$ and $x\in(nl-l_1,nl)$, $u(t,x)=u_{2n}(t,x)$ for $t>0$ and $x\in(nl,nl+l_2)$, $u(t,x)=u_{1n}(t,x^-)=u_{2n}(t,x^+)$ for $t>0$ and $x=nl$, and $u(t,x)=u_{2n}(t,x^-)=u_{1(n+1)}(t,x^+)$ for $t>0$ and $x=nl+l_2$, we rewrite the above model \eqref{m-1}--\eqref{m-2} in the following form:
\begin{equation}
\label{m-pb}
\left\{\baa{rclrcll}
u_t-d(x)u_{xx} & = & f(x,u), & & & & t>0,\ x\in\mathbb{R}\!\setminus\!S,\vspace{3pt}\\
u(t,x^-) & = & u(t,x^+), & u_x(t,x^-) & = & \sigma u_x(t,x^+), & t> 0,\ x\in S_1,\vspace{3pt}\\
u(t,x^-) & = & u(t,x^+), & \sigma u_x(t,x^-) & = & u_x(t,x^+), & t> 0,\ x\in S_2,\eaa\right.
\end{equation}
where the diffusivity $d$ and nonlinearity $f$ are given by
\begin{align}
\label{d-f}
d(x)=\begin{cases} d_1,~~x\in(nl-l_1,nl),\\ d_2,~~x\in(nl,nl+l_2),\end{cases}~~
f(x,s)=\begin{cases} f_1(s),~~x\in(nl-l_1,nl),\\ f_2(s),~~x\in(nl,nl+l_2),\end{cases}
\end{align}
and the parameter $\sigma>0$ is defined as in \eqref{sigma}. Conditions~\eqref{2.5'-existence} and~\eqref{2.4'-uniqueness} on $f_i$ are equivalent to the following ones: 
\begin{align}
\label{2.5}
\begin{cases}
\forall\,x\in\mathbb{R}\!\setminus\!S,\ f(x,\cdot)\in C^1(\mathbb{R}),\ f(x,0)=0,\\
\exists\,M=\max(K_1,K_2)>0,\ \forall\,x\in\mathbb{R}\!\setminus\!S,\ \forall\,s\ge M,\ f(x,s)\le0,\\
\forall\,x\in(nl-l_1,nl),\ f(x,\cdot)=f_1,\ \ \forall\,x\in(nl,nl+l_2),\ f(x,\cdot)=f_2.
\end{cases}
\end{align} 
and
\begin{equation}
\label{2.4}
\begin{aligned}
\begin{cases}
\displaystyle\hbox{the functions }s\mapsto \frac{f(x,s)}s~\text{are non-increasing in}~s>0\hbox{ in all patches},\\
\displaystyle\hbox{and decreasing in $s>0$ in at least one type of patch}.
\end{cases}
\end{aligned}
\end{equation}
From now on, we always assume that \eqref{2.5} is satisfied. Throughout this paper, unless otherwise specified,  we always write $I$ for an arbitrary patch in $\mathbb{R}$ of either type, i.e., either~$I=(nl-l_1,nl)$ or~$I=(nl,nl+l_2)$.


\subsection{Well-posedness of the Cauchy problem \eqref{m-pb}--\eqref{d-f}}

Since the patch model considered in this paper is not standard, we shall first establish the well-posedness of the Cauchy problem \eqref{m-pb}--\eqref{d-f} with hypotheses~\eqref{2.5} on $f$ and with nonnegative bounded and continuous initial conditions $u_0:\R\to\R$. Before proceeding with the analysis, we present here the definition of a classical solution to \eqref{m-pb}--\eqref{d-f}. 

\begin{definition}\label{def1}
For $T\in(0,+\infty]$, we say that a continuous function $u:[0,T)\times\R\to\R$ {$($notice that $u$ is continuous up to $t=0$$)$} is a classical solution of the Cauchy problem~\eqref{m-pb}--\eqref{d-f} in~$[0,T)\times\R$ with an initial condition~$u_0$, if $u(0,x)=u_0(x)$ for all $x\in\R$, if $u|_{(0,T)\times\bar I}$ is of class~$C^{1;2}_{t;x}\big((0,T)\times \bar I\big)$ for each patch $I=(nl-l_1,nl)$ or $(nl,nl+l_2)$, and if all identities in \eqref{m-pb} are satisfied pointwise for $0<t<T$.
\end{definition}

\begin{theorem}
\label{thm-m-wellposedness}
Under assumption~\eqref{2.5}, for any nonnegative bounded continuous initial condition $u_0$, there is a unique nonnegative bounded classical solution~$u$ in~$[0,+\infty)\times\R$ of the Cauchy problem~\eqref{m-pb}--\eqref{d-f}. Furthermore, for any $\tau>0$ and any patch $I\subset \mathbb{R}$, 
\begin{align*}
\Vert u|_{[\tau,+\infty)\times\bar I}\Vert_{C^{1,\gamma;2,\gamma}_{t;x}([\tau,+\infty)\times \bar I)}\le C,
\end{align*}
with a positive constant $C$ depending on $\tau$, $l_{1,2}$, $d_{1,2}$, $f_{1,2}$, $\sigma$ and $\Vert u_0 \Vert_{L^\infty(\mathbb{R})}$, and with a universal positive constant $\gamma\in(0,1)$. Moreover, $u(t,x)>0$ for all $(t,x)\in(0,+\infty)\times\R$ if $u_0\not\equiv0$, and $u(t,x)=u(t,x+l)$ for all $(t,x)\in[0,+\infty)\times\R$ if $u_0(x)=u_0(x+l)$ for all $x\in\R$. Lastly, the solutions depend monotonically and continuously on the initial data, in the sense that if $u_0\le v_0$ then the corresponding solutions satisfy $u\le v$ in~$[0,+\infty)\times\R$, and for any $T\in(0,+\infty)$ the map~$u_0\mapsto u$ is continuous from $C^+(\R)\cap L^\infty(\R)$ to~$C([0,T]\times\R)\cap L^\infty([0,T]\times\R)$ equipped with the sup norms, where $C^+(\R)$ denotes the set of nonnegative continuous functions in $\R$.
\end{theorem}

We remark that the existence and  uniqueness of a global bounded periodic classical solution to such a patch model was  considered in \cite{MCCL2020} for \eqref{their model-1}--\eqref{their model-2} with periodic and possibly discontinuous initial data. By contrast, our result is  established for general continuous and bounded initial data. Moreover, we also discuss the continuous dependence of solutions on intial data and give  a priori estimates, which will play a critical role in the  monotone semiflow argument used in the sequel. The well-posedness proof here can also be adapted to other non-periodic patch problems. 


\subsection{Existence, uniqueness and attractiveness of a positive periodic steady state}

To investigate the existence and uniqueness of a positive bounded steady state as well as the large-time behavior of solutions to the Cauchy problem, we first study the following eigenvalue problem. From~\cite{ML2013,SKT1986} (see also Lemma~\ref{lemma-concave} below), there exists a principal eigenvalue $\lambda_1$, defined as the unique real number such that there exists a unique continuous function $\phi:\R\to\R$ with~$\phi|_{\bar{I}}\in C^\infty(\bar I)$ for each patch~$I$, that satisfies
\begin{align}
\label{ep-0}
\begin{cases}
\mathcal{L}_0\phi:=-d(x)\phi''-f_s(x,0)\phi=\lambda_1 \phi,~~&x\in\mathbb{R}\backslash S,\\
\phi(x^-)=\phi(x^+),~\phi'(x^-)=\sigma\phi'(x^+),~~&x\in S_1,\\
\phi(x^-)=\phi(x^+),~\sigma\phi'(x^-)=\phi'(x^+),~~&x\in S_2,\\
\phi(x)~\text{is}~\text{periodic},~\phi>0,~\Vert \phi \Vert_{L^\infty(\mathbb{R})}=1.
\end{cases}
\end{align}
By periodic, we mean that $\phi(\cdot+l)=\phi$ in $\R$. In the sequel we say that $0$ is an unstable steady state of \eqref{m-pb}--\eqref{d-f} if $\lambda_1<0$, otherwise the state $0$ is said to be stable (i.e., $\lambda_1\ge 0$). These definitions will be seen to be natural in view of the results we prove here. By applying~\eqref{ep-0} at minimal and maximal points of the positive continuous periodic function $\phi$, whether these points be in patches or on the interfaces, it easily follows that
$$-f'_1(0)\le\lambda_1\le-f'_2(0)$$
(remember that $f'_1(0)\ge f'_2(0)$ without loss of generality). In particular, if $\lambda_1<0$, then $f'_1(0)>0$, that is, $f_s(x,0)$ is necessarily positive (at least) in the favorable patches. 

We first state a criterion for the existence of a continuous solution $p:\R\to\R$ (such that  $p|_{\bar{I}}\in C^2(\bar I)$ for each patch $I$) to the elliptic problem:
\begin{equation}
\label{m-pb-elliptic}
\begin{aligned}
\begin{cases}
-d(x)p''(x)-f(x,p(x))=0,~~&x\in\mathbb{R}\backslash S,\\
p(x^-)=p(x^+),~p'(x^-)=\sigma p'(x^+),~~&x\in S_1,\\
p(x^-)=p(x^+),~\sigma p'(x^-)=p'(x^+),~~&x\in S_2.
\end{cases}
\end{aligned}
\end{equation}

\begin{theorem}
\label{thm-2.1-existence}
(i) Assume that $0$ is an unstable solution of \eqref{m-pb-elliptic} $($i.e., $\lambda_1<0$$)$ and that $f$ satisfies~\eqref{2.5}. Then there exists a bounded positive and periodic  solution $p$ of \eqref{m-pb-elliptic}.\par
(ii) Assume that $0$ is a stable solution of \eqref{m-pb-elliptic} $($i.e., $\lambda_1\ge 0$$)$ and that $f$ satisfies \eqref{2.5}--\eqref{2.4}. Then $0$ is the only nonnegative bounded  solution of \eqref{m-pb-elliptic}.
\end{theorem}

For reaction-diffusion equations that describe population dynamics in general periodically fragmented landscapes but do not include movement behavior at interfaces, the criteria for existence (and uniqueness) of the stationary problem in arbitrary dimension can be found in \cite{BHR2005-1}. It turns out that the approach there can be adapted to our periodic patch model with the additional nonstandard interface conditions. 

Let us now provide an insight into the stability of the trivial solution of \eqref{m-pb-elliptic}.  Under certain reasonable hypotheses on the diffusitivies, the sizes of favorable and unfavorable patches, as well as the nonlinearities, the principal eigenvalue $\lambda_1$ of \eqref{ep-0} can indeed be negative. For example,  when all patches support population growth, namely $f'_1(0)>0$ and $f'_2(0)>0$, then the zero state is unstable. When the landscape consists of source and sink patches, i.e., when $f_1'(0)>0>f_2'(0)$, the stability of the zero state depends on the relationships between patch size, patch preference, diffusivity and growth rates. In the case {\blue $k=1$},  Shigesada and coworkers derived such a stability criterion \cite{SKT1986}; the case for general $\sigma>0$ can be found in \cite{ML2013}. We here derive an even more general formula, when we only assume that $f'_2(0)\le f'_1(0)$. To do so, we first observe that the continuous functions $x\mapsto\phi(-l_1-x)$ and $x\mapsto\phi(l_2-x)$ still solve~\eqref{ep-0} as $\phi$ does, and by uniqueness we get that $\phi(-l_1-x)=\phi(l_2-x)=\phi(x)$ for all $x\in\R$, hence~$\phi'(-l_1/2)=\phi'(l_2/2)=0$. Then, as in~\cite{ML2013,SKT1986}, by solving~\eqref{ep-0} in $[-l_1/2,0]$ and in $[0,l_2/2]$ with zero derivatives at $-l_1/2$ and $l_2/2$, and by matching the interface conditions at $0$, we find that~$\lambda_1$ is the smallest root in~$[-f'_1(0),-f'_2(0)]$ of the equation:
\begin{equation}\label{formula}
\sqrt{\frac{f_1'(0)+\lambda_1}{d_1}}\,\tan\Big(\sqrt{\frac{f_1'(0)+\lambda_1}{d_1}}\times\frac{l_1}{2}\Big)=\sigma\sqrt{-\frac{\lambda_1+f_2'(0)}{d_2}}\,\tanh\Big(\sqrt{-\frac{\lambda_1+f_2'(0)}{d_2}}\times\frac{l_2}{2}\Big).
\end{equation}
When $0<f'_2(0)\le f'_1(0)$ or when $0=f'_2(0)<f'_1(0)$ (irrespective of the other parameters), then the trivial solution of \eqref{m-pb-elliptic} is unstable (i.e., $\lambda_1<0$). When $f'_2(0)\le f'_1(0)\le0$, then $\lambda_1\ge0$. When $f'_2(0)<0<f'_1(0)$, we then derive that the trivial solution of \eqref{m-pb-elliptic} is stable ($\lambda_1\ge 0$) if 
\begin{equation}
l_1\le l_1^c\colon =2\sqrt{\frac{d_1}{f'_1(0)} }\tan^{-1}\left(\sigma\sqrt{\frac{-d_1f'_2(0)}{d_2f'_1(0)}}\tanh\Big(\sqrt{\frac{-f_2'(0)}{d_2}}\times\frac{l_2}{2}\Big)\right)
\label{l1c}
\end{equation}
(notice that $l_1^c>0$), and unstable ($\lambda_1<0$) if $l_1>l_1^c$. The persistence threshold $l_1^c$ is decreasing with $f_1'(0)>0$ and increasing with $d_1$ and $l_2$. Passing to the limit   $l_2\to+\infty$, we find that 
\begin{equation*}
l_1^c\to L^c_1\colon=2\sqrt{\frac{d_1}{f'_1(0)} }\tan^{-1}\left(\sigma\sqrt{\frac{-d_1f'_2(0)}{d_2f'_1(0)}}\right).
\end{equation*}
Therefore, as long as $l_1>L^c_1$, the trivial solution of~\eqref{m-pb-elliptic} is unstable (i.e., $\lambda_1<0$), no matter how large the size of the unfavorable patches is. Similarly, there is a critical rate
$$(f'_2(0))^c=-\frac{d_2 f'_1(0)}{\sigma^2 d_1}\left[\tan\Big(\sqrt{\frac{f_1'(0)}{d_1}}\times\frac{l_1}{2}\Big)\right]^2$$
such that, if $0>f_2'(0)>(f'_2(0))^c$, then  the trivial solution of \eqref{m-pb-elliptic} is unstable (i.e., $\lambda_1<0$), no matter how large the size of the unfavorable patch is.  

It also follows from~\eqref{formula} that, provided $f'_2(0)\neq f'_1(0)$, the principal eigenvalue $\lambda_1$ is increasing with respect to $\sigma>0$, that is, $\lambda_1$ is decreasing with respect to $\alpha\in(0,1)$. When $\alpha\in(0,1)$ increases, then the individuals at the interfaces have more propensity to go to patches of type~$1$ rather than to patches of type~$2$. This means that the relative advantage of the more favorable patches becomes more prominent:~$\lambda_1$ decreases and the~$0$ solution has more chances to become unstable. It is also easy to see that $\lambda_1\to-f'_1(0)$ as $\sigma\to0^+$ (that is, as $\alpha\to1^-$), hence $0$ is unstable if $\alpha\approx1$, provided the patches of type $1$ support population growth. On the other hand, $\lambda_1\to\min(d_1\pi^2/l_1^2-f'_1(0),-f'_2(0))$ as $\sigma\to+\infty$ (that is, as $\alpha\to0^+$). Therefore, if $f'_1(0)\ge d_1\pi^2/l_1^2$, and even if $f'_2(0)<0$, then~$0$ is still unstable when $\alpha$ is small (and actually whatever the value of $\alpha\in(0,1)$ and the other parameters may be).

Next, we state a Liouville type result for problem \eqref{m-pb-elliptic}.

\begin{theorem}
\label{thm-2.4-uniqueness}
Assume that $f$ satisfies~\eqref{2.5}--\eqref{2.4} and that the zero solution of \eqref{m-pb-elliptic} is unstable $($i.e., $\lambda_1<0$$)$. Then there exists at most one positive and bounded solution $p$ of \eqref{m-pb-elliptic}. Furthermore, such a solution $p$, if any, is periodic and $\inf_{\mathbb{R}}p=\min_{\R}p>0$.
\end{theorem}

Under the assumptions of Theorem~\ref{thm-2.1-existence}~(i) and Theorem~\ref{thm-2.4-uniqueness}, we now look at the global attractiveness of the unique positive and bounded stationary solution $p$ of~\eqref{m-pb-elliptic} for the solutions of the Cauchy problem~\eqref{m-pb}--\eqref{d-f}.

\begin{theorem}
\label{thm-long time behavior}
Assume that  $f$ satisfies \eqref{2.5}--\eqref{2.4}. Let $u$ be the solution of the Cauchy problem \eqref{m-pb}--\eqref{d-f} with a nonnegative bounded and continuous initial datum $u_0\not\equiv0$.  
\begin{enumerate}[(i)]
\item If $0$ is an unstable solution of \eqref{m-pb-elliptic} $($i.e., $\lambda_1<0$$)$, then $u(t,\cdot)|_{\bar{I}}\to p|_{\bar I}$ in $C^2(\bar I)$ as $t\to+\infty$ for each patch $I$, where $p$ is the unique positive bounded and periodic solution of \eqref{m-pb-elliptic} given by Theorem~$\ref{thm-2.1-existence}$~{\rm{(i)}} and Theorem~$\ref{thm-2.4-uniqueness}$.\footnote{This statement shows that the solution $u$ converges as $t\to+\infty$ locally uniformly in space to the space-periodic function $p$. For a convergence result to time-periodic solutions for time-periodic quasilinear parabolic equations, we refer to~\cite{BPS92}.}
\item If $0$ is a stable solution of \eqref{m-pb-elliptic} $($i.e., $\lambda_1\ge 0$$)$, then $u(t,\cdot)\to0$ uniformly in $\mathbb{R}$ as $t\to+\infty$.
\end{enumerate}
\end{theorem}


\subsection{Spreading speeds and pulsating traveling waves}

In this subsection, we assume that the zero solution of \eqref{m-pb-elliptic} is unstable (i.e., $\lambda_1<0$) and that~$f$ satisfies~\eqref{2.5}--\eqref{2.4}. Let $p$ be the unique positive bounded and periodic solution of~\eqref{m-pb-elliptic} obtained from Theorem~\ref{thm-2.1-existence}~(i) and Theorem~\ref{thm-2.4-uniqueness}. After showing in Theorem~\ref{thm-long time behavior}~(i) the attractiveness of $p$, we now want to describe the way the positive steady state $p$ invades the whole domain.
   
Let $\mathcal{C}$ be the space of all bounded and uniformly continuous functions from $\mathbb{R}$ to $\mathbb{R}$ equipped with the compact open topology, i.e., we say that $u_n\to u$ as $n\to+\infty$ in $\mathcal{C}$ when $u_n\to u$ locally uniformly in~$\R$. For $u,v\in\mathcal{C}$, we write $u\ge v$ when $u(x)\ge v(x)$ for all $x\in\mathbb{R}$, $u>v$ when $u\ge v$ and $u\not\equiv v$, and~$u\gg v$ when $u(x)> v(x)$ for all $x\in\mathbb{R}$. Notice that $p\in\mathcal{C}$ is periodic and satisfies~$p\gg 0$. We define
\be\label{defCp}
\mathcal{C}_p=\{v\in \mathcal{C}: 0\le v\le p\}.
\ee
Let $\mathbb{P}$ be the set of all continuous and periodic functions from~$\mathbb{R}$ to~$\mathbb{R}$ equipped with the $L^\infty$-norm,  and $\mathbb{P}^+=\{u\in\mathbb{P} : u\ge 0\}$.

The first result of this section states the existence of a speed of invasion by the state $p$.
 
\begin{theorem}
\label{thm-spreading result}
Assume that $f$ satisfies \eqref{2.5}--\eqref{2.4} and that the zero solution of \eqref{m-pb-elliptic} is unstable $($i.e., $\lambda_1<0$$)$. Then there is an asymptotic spreading speed, $c^*>0$, given explicitly by 
\be\label{formulac*}
c^*=\inf_{\mu>0}\frac{-\lambda(\mu)}{\mu},
\ee
where $\lambda(\mu)$ is the principal eigenvalue of the operator
\begin{align*}
\mathcal{L}_\mu\psi(x) :=-d(x)\psi''(x)+2\mu d(x)\psi'(x)-(d(x)\mu^2+f_s(x,0))\psi(x)~~\text{for }x\in \mathbb{R}\!\setminus\!S,
\end{align*}
acting on the set 
\begin{align*}
E_\mu=\big\{\psi\in C(\mathbb{R}): & ~~\psi|_{\bar I}\in C^2(\bar I)\hbox{ for each patch }I,\ \psi\text{ is periodic in}~\mathbb{R},\\
&~~[-\mu \psi+\psi'](x^-)=\sigma[-\mu \psi+\psi'](x^+)~\text{for}~x\in S_1,\\
&~~\sigma[-\mu \psi+ \psi'](x^-)=[-\mu \psi+\psi'](x^+)~\text{for}~ x\in S_2\big\},
\end{align*}
such that the following statements are valid:
\begin{enumerate}[(i)]
\item if $u$ is the solution to problem~\eqref{m-pb}--\eqref{d-f} with a compactly supported initial condition $u_0\in\mathcal{C}_p$, then $\lim_{t\to+\infty}\,\sup_{|x|\ge ct}u(t,x)=0$ for every $c>c^*$;
\item if $u_0\in\mathcal{C}_p$ with $u_0\not\equiv 0$,  then $\lim_{t\to+\infty}\,\max_{|x|\le ct}|u(t,x)-p(x)|=0$ for every $0\le c<c^*$.
\end{enumerate}
\end{theorem}

It finally turns out that the asymptotic spreading speed $c^*$ is also related to some speeds of rightward or leftward periodic (also called pulsating) traveling waves, whose definition is recalled:

\begin{definition}\label{def4}
A bounded continuous solution $u:\R\times\R\to\R$ of problem \eqref{m-pb}--\eqref{d-f} is called a periodic rightward traveling wave connecting $p(x)$ to $0$ if it has the form $u(t,x)=W(x-ct,x)$, where $c\in\R$ and the function $W:\R\times\R\to\R$ has the properties: for each $s\in\R$ the map $x\mapsto W(x+s,x)$ is continuous\footnote{Notice that the continuity of $x\mapsto W(x+s,x)$ is automatic if $c\neq0$, since $u$ is assumed to be continuous itself in $\R\times\R$.} and the map $x\mapsto W(s,x)$ is periodic, and for each $x\in\mathbb{R}$ the map $s\mapsto W(s,x)$ is decreasing with $W(-\infty,x)=p(x)$ and $W(+\infty,x)=0$.\par
Similarly, a bounded continuous solution $u:\R\times\R\to\R$ of problem \eqref{m-pb}--\eqref{d-f} is called a periodic leftward traveling wave connecting $0$ to $p(x)$ if it has the form $u(t,x)=W(x+ct,x)$, where $c\in\R$ and the function $W:\R\times\R\to\R$ has the properties: for each $s\in\R$ the map $x\mapsto W(x+s,x)$ is continuous and the map $x\mapsto W(s,x)$ is periodic, and for each $x\in\mathbb{R}$ the map $s\mapsto W(s,x)$ is increasing with $W(-\infty,x)=0$ and $W(+\infty,x)=p(x)$.
\end{definition}
 
The following result shows that   the asymptotic spreading speed $c^*$ given in Theorem \ref{thm-spreading result} coincides with  minimal speeds of periodic traveling waves in the positive and negative directions. 

\begin{theorem}\label{thm-existence of PTW}
 Assume that the zero solution of \eqref{m-pb-elliptic} is unstable $($i.e., $\lambda_1<0$$)$ and that $f$ satisfies \eqref{2.5}--\eqref{2.4}.  Let $c^*$  be the  asymptotic spreading speed given in Theorem~$\ref{thm-spreading result}$. Then the following statements are valid:
\begin{enumerate}[(i)]
\item problem \eqref{m-pb}--\eqref{d-f} has a periodic rightward traveling wave $W(x-ct,x)$ connecting $p(x)$ to $0$, in the sense of Definition~$\ref{def4}$, if and only if $c\ge c^*$;
\item problem \eqref{m-pb}--\eqref{d-f} has a periodic leftward traveling wave $W(x+ct,x)$ connecting $0$ to~$p(x)$, in the sense of Definition~$\ref{def4}$, if and only if $c\ge c^*$.
\end{enumerate}
\end{theorem}

\begin{remark}{\rm It is known that for the standard spatially periodic Fisher-KPP problem~\eqref{RD-periodic} with $N=1$, the variational characterization of minimal speeds in terms of a family of principal eigenvalues implies that the minimal wave speeds of rightward and leftward pulsating waves are the same. Theorem~$\ref{thm-existence of PTW}$ shows that this property still holds true for our one-dimensional  patchy periodic habitat, with nonstandard movement behavior at interfaces. }
\end{remark}

\begin{remark}{\rm After the completion of this work, we learned  about the work~\cite{SKW2015} by Shigesada, Kawasaki and Weinberger, on reaction-diffusion-advection models in periodic environments, with advection given by a gradient-based taxis. In the case of patchy environments with logistic growth rates, the piecewise constant discontinuous coefficients are approximated by periodic continuous coefficients converging locally uniformly outside of the set of discontinuities. For the approximated problems, the existence of leftward and rightward traveling waves connecting $0$ and the unique positive periodic steady state follows from~\cite{W2002}. Then the minimal speeds of traveling waves of the approximated problems are shown in~\cite{SKW2015} to converge to a quantity, which we here call~$c^*_{SKW}$. Interface conditions of type~\eqref{their model-2} are also derived for the limits of exponential tails $e^{-s(\pm x-ct)}g(x)$ (with periodic $g$) solving linearizations of the approximated equations. A detailed analysis of the dependence of~$c^*_{SKW}$ on the various parameters, in particular the diffusions and the gradient-based taxis, is also carried out in~\cite{SKW2015}. Approximating our equivalent patch models \eqref{their model-1}--\eqref{their model-2} and~\eqref{m-pb}--\eqref{d-f} by problems with continuously interpolated diffusion and growth rates, with additional advection terms supported in smaller and smaller neighborhoods of the interfaces as in~\cite{SKW2015}, would certainly be an interesting problem. But proving rigorously the well-posedness of~\eqref{m-pb}--\eqref{d-f} through this method, and obtaining the continuity and Schauder estimates, as in Theorem~\ref{thm-m-wellposedness}, would require a serious analysis, though we expect it would work. We have here chosen an alternate method: though we still consider approximated problems, this time by truncation in bounded intervals, we solve directly the patch problem with discontinuous coefficients, by developing the semigroup theory for this problem. We also point out that Theorem~\ref{thm-m-wellposedness} plays a crucial role in the derivation of all further results, such as the attractiveness of the positive steady state~$p$, whose existence and uniqueness is also proved for~\eqref{m-pb}--\eqref{d-f}. By using the theory of dynamical systems as in~\cite{LZ2007,W2002}, Theorem~\ref{thm-m-wellposedness} also allows us to prove the existence of a spreading speed~$c^*$ for the solutions of~\eqref{m-pb}--\eqref{d-f}, and to show that $c^*$ is the minimal speed of leftward and rightward traveling waves connecting $p$ and $0$ for~\eqref{m-pb}--\eqref{d-f}. As expected, it turns out from~\eqref{formulac*} and from~\cite[Formulas (21)--(22)]{SKW2015} that~$c^*=c^*_{SKW}$. In addition to the previous results, our analysis thus confirms that the limit $c^*_{SKW}$ of minimal speeds of traveling waves of approximated smoother problems is really the minimal speed~$c^*$ of traveling waves of the original patch model, as well as the spreading speed for the solutions of the Cauchy problem~\eqref{m-pb}--\eqref{d-f}. On the other hand, the analysis in~\cite{SKW2015} of the dependence of $c^*_{SKW}$ on the parameters $d_{1,2}$, $f_{1,2}$, $l_{1,2}$ and the preference parameter between the two types of patches provides additional relevant informations on such models.}
\end{remark}

\noindent{\bf Outline of the paper.} The rest of the paper is organized as follows. In the next section, we give the proof of Theorem~\ref{thm-m-wellposedness} on the well-posedness of the Cauchy probldem~\eqref{m-pb}--\eqref{d-f}. Section~\ref{Sec-thm-Liouville} is devoted to the study of the stationary problem \eqref{m-pb-elliptic} and we give the proofs of Theorems \ref{thm-2.1-existence} and \ref{thm-2.4-uniqueness}. In Section \ref{Sec-thm-long time behavior},  we prove Theorem \ref{thm-long time behavior} on the large-time behavior of the evolution problem. Finally, Section \ref{Sec-thm-spreading properities} is devoted to the proofs of Theorems \ref{thm-spreading result} and \ref{thm-existence of PTW}, based on the abstract monotone semiflow method developed in \cite{LZ2007,LZ2010,W2002}. Lastly, the appendix is devoted to giving supplementary comparison results concerning finitely many patches, which play an essential role in the well-posedness argument in Section \ref{Sec-thm-wellposedness}.


\section{Well-posedness of the Cauchy problem \eqref{m-pb}--\eqref{d-f}: proof of Theorem \ref{thm-m-wellposedness}}\label{Sec-thm-wellposedness}

In this section, we establish the well-posedness of the Cauchy problem \eqref{m-pb}--\eqref{d-f} with nonnegative, bounded and continuous initial data. We first show the existence of classical solutions based on a semigroup argument and an approximation approach. Then we prove that the solutions are unique and depend monotonically and continuously on the initial data.


\subsection{Truncated problem}\label{sec-2.1}

Fix $n\in \mathbb{N}$. We consider the following truncated problem of \eqref{m-pb}--\eqref{d-f}  in the finite interval $[-nl,nl]$, which consists of $4n$ disjoint patches (see Figure~2):
\begin{equation}
\label{tp-1}
\left\{\begin{aligned}
\frac{\partial u_{1m}}{\partial t}&=d_1 \frac{\partial^2 u_{1m}}{\partial x^2}\!+\!f_1(u_{1m}),~ t>0,~x\!\in\!(ml-l_1,ml),~m=0, \pm 1,\cdots, \pm (n-1),n,\cr
\frac{\partial u_{2m}}{\partial t}&=d_2 \frac{\partial^2 u_{2m}}{\partial x^2}\!+\!f_2(u_{2m}),~t>0,~x\!\in\!(ml,ml+l_2),~m=0, \pm 1,\cdots, \pm (n-1),-n,\cr
\end{aligned}\right.
\end{equation}
together with interface conditions 
\begin{equation}
\left\{\begin{aligned}
\label{tp-2}
u_{1m}(t,x^-)&=u_{2m}(t,x^+),\qquad~~~~(u_{1m})_x(t,x^-)=\sigma(u_{2m})_x(t,x^+),\\
&\qquad\qquad\qquad\qquad\qquad\ \ t> 0,~x=ml,~m=0, \pm 1,\cdots, \pm (n-1),\\
u_{2m}(t,x^-)&=u_{1(m+1)}(t,x^+),~~~\sigma (u_{2m})_x(t,x^-)=  (u_{1(m+1)})_x(t,x^+),\\
&\qquad\qquad\qquad\qquad\qquad\ \ t>0,~x=ml+l_2,~m=0, \pm 1,\cdots, \pm (n-1),-n,
\end{aligned}\right.
\end{equation}
and boundary conditions at $x=\pm nl$:
\begin{equation}
\begin{aligned}
\label{tp-3}
u_{1n}(t,(-nl)^+)=u_{2(-n)}(t,(nl)^-)=0,~~t> 0.
\end{aligned}
\end{equation}
For consistency of notations, we set
$$\left\{\baa{ll}
I_{1m}=(ml-l_1,ml) & \hbox{for $m\in J_1=\{0,\pm 1,...,\pm (n-1),n\}$},\vspace{3pt}\\
I_{2m}=(ml,ml+l_2) &\hbox{for $m\in J_2=\{0,\pm 1,...,\pm (n-1),-n\}$}.\eaa\right.$$
We number these $4n$ patches from left to right by $I_{2(-n)}$, $I_{1(-n+1)}$,..., $I_{10}$, $I_{20}$,..., $I_{2(n-1)}$, $I_{1n}$, so that
$$[-nl,nl]=\Big(\bigcup_{j\in J_1}\overline{I_{1j}}\Big)\cup \Big(\bigcup_{j\in J_2}\overline{I_{2j}}\Big).$$
 
\begin{figure}[H]
\centering
\includegraphics[scale=0.6]{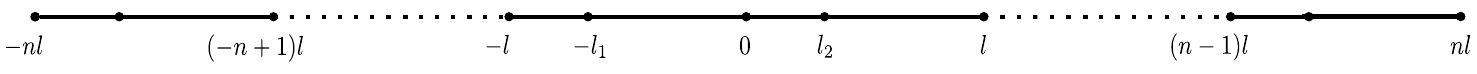}\caption{Truncated interval $[-nl, nl]$.}
\end{figure}

For a solution $(u_{2(-n)},\ldots,u_{1n})$ of~\eqref{tp-1}--\eqref{tp-3}, we define $u:(0,+\infty)\times[-nl,nl]\to\R$ such that, for $t>0$,
\be\label{tp-4}\left\{\baa{l}
u(t,x)=u_{1j}(t,x)\hbox{ if $x\in I_{1j}$ with $j\in J_1$},\vspace{3pt}\\
u(t,x)=u_{2j}(t,x)\hbox{ if $x\in I_{2j}$ with $j\in J_2$},\vspace{3pt}\\
u(t,\cdot)\hbox{ is extended by continuity at the interior interfaces and by $0$ at $\pm nl$}.\eaa\right.
\ee
We finally set 
\be\label{defC0}
\mathcal{C}_0=\big\{\varphi\in C([-nl,nl]):\varphi(-nl)=\varphi(nl)=0\big\},
\ee
equipped with the sup norm. 

\begin{definition}\label{def2}
For $T\in(0,+\infty]$, we say that a continuous function $u:[0,T)\times[-nl,nl]\to\R$ $($notice that $u$ is continuous up to $t=0$$)$ is a classical solution  to the truncated problem~\eqref{tp-1}--\eqref{tp-4} in $[0,T)\times[-nl,nl]$ with an initial condition $u_0\in\mathcal{C}_0$, if $u(0,x)=u_0(x)$ for all $x\in[-nl,nl]$, if~$u|_{(0,T)\times\bar I}$ is of class $C^{1;2}_{t;x}\big((0,T)\times \bar I\big)$ for each patch $I\subset[-nl,nl]$, and if all identities in~\eqref{tp-1}--\eqref{tp-4} are satisfied pointwise for $0<t<T$.
\end{definition}

\begin{theorem}
\label{tp-wellposedness}
Under the assumption~\eqref{2.5}, the Cauchy problem \eqref{tp-1}--\eqref{tp-4} with a nonnegative initial condition $u_0\in\mathcal{C}_0$ admits a unique bounded classical solution $u:[0,+\infty)\times[-nl,nl]\to\R$. Furthermore,
\be\label{umax}
0\le u\le\max\big(K_1,K_2,\|u_0\|_{L^\infty(-nl,nl)}\big)\ \hbox{ in $[0,+\infty)\times[-nl,nl]$}.
\ee
If $0\le u_0\le v_0$ in $[-nl,nl]$ with $u_0,v_0$ in~$\mathcal{C}_0$, then the solutions $u$ and $v$ with respective initial conditions $u_0$ and $v_0$ satisfy $u\le v$ in~$[0,+\infty)\times[-nl,nl]$.
\end{theorem}
 
The uniqueness and comparison properties immediately follow from Proposition \ref{bdd-cp} in the appendix. In what follows, we prove the existence of a bounded classical solution to \eqref{tp-1}--\eqref{tp-4}, relying on semigroup theory. To do so, we first introduce some auxiliary spaces and recast the truncated Cauchy problem into the abstract form:
\begin{equation}
\label{abstract-nonlinear}
\begin{cases}
\displaystyle\frac{dU}{dt}+AU=f(U),\ \ t>0,\\
U(0)=U_0,
\end{cases}
\end{equation}
where $U=(u_{2(-n)},\cdots,u_{1n})^T$, $U_0=(u_0|_{(-nl,-nl+l_2)},\cdots,u_0|_{(nl-l_1,nl)})^T$, and $A$ and $f(U)$ are defined by
\begin{align}
\label{defAf}
A=\begin{pmatrix}
-d_2\partial_{xx} & & & &\\
&-d_1\partial_{xx}& & & \\
& &\ddots& &\\
& & & -d_2\partial_{xx}& \\
& & & & -d_1\partial_{xx}
\end{pmatrix}_{4n\times 4n}\!\!\hbox{and }~~f(U)=\begin{pmatrix}
f_2(u_{2(-n)})\\
f_1(u_{1(-n+1)})\\
\vdots\\
f_2(u_{2(n-1)})\\
f_1(u_{1n})
\end{pmatrix}\!.
\end{align}
Set
\begin{align*}
X&=L^2(I_{2(-n)})\times\cdots\times L^2(I_{1n}),
\end{align*}
with elements viewed as column vectors. With a slight abuse of notation, $X$ can be identified with~$L^2(-nl,nl)$. We then define an inner product in $X$ as follows:
\begin{align}
\label{inner product on X}
\langle U,V\rangle_X=\sum_{j\in J_1} \int_{I_{1j}} u_{1j}v_{1j}+\frac{1}{k}\sum_{j\in J_2} \int_{I_{2j}} u_{2j}v_{2j},
\end{align}
which induces the norm $U\mapsto\|U\|_X=\sqrt{\langle U,U\rangle_X}$ and makes $X$ a Hilbert space.\footnote{We recall that $k>0$ is given in~\eqref{k}. In all integrals, we integrate with respect to the one-dimensional Lebesgue measure.} We also introduce other Hilbert spaces
$$\left\{\baa{l}
\mathcal{H}^1=\big\{(u_{2(-n)},\ldots,u_{1n})^T\in H^1(I_{2(-n)})\times\cdots\times H^1(I_{1n}),\vspace{2pt}\\
\qquad\qquad u_{1m}(x)=u_{2m}(x),\ \ \ \ \ \ x=ml,\ \ \ \ \ \ \, m=0,\pm 1,\ldots, \pm(n-1),\vspace{2pt}\\
\qquad\qquad u_{2m}(x)=u_{1(m+1)}(x),\ x=ml+l_2,\ m=0,\pm 1,\ldots, \pm(n-1),-n\big\},\vspace{4pt}\\
\mathcal{H}^1_0=\big\{(u_{2(-n)},\ldots, u_{10}, u_{20},\ldots, u_{1n})^T\in \mathcal{H}^1:u_{2(-n)}(-nl)=u_{1n}(nl)=0\big\},\eaa\right.$$
with elements viewed as column vectors, equipped with the norm
\begin{align*}
\Vert U \Vert_{\mathcal{H}^1}=\Vert U \Vert_{\mathcal{H}^1_0}=\sqrt{\sum_{j\in J_1} \Vert u_{1j}\Vert^2_{H^1(I_{1j})}+\frac{1}{k}\sum_{j\in J_2} \Vert u_{2j}\Vert^2_{H^1(I_{2j})}}.
\end{align*}
From the Sobolev embeddings and with a slight abuse of notation, $\mathcal{H}^1$ and $\mathcal{H}^1_0$ can be identified with $H^1(-nl,nl)$ and $H^1_0(-nl,nl)$, respectively, and viewed as subsets of~$C([-nl,nl])$ and~$\mathcal{C}_0$, respec\-tively, with definition~\eqref{defC0}. Furthermore, in $\mathcal{H}^1_0$ the norms $\|U\|_{\mathcal{H}^1_0}$ and $\|U'\|_X$ are equivalent, from Poincar\'e's inequality. We finally define the Hilbert space
$$\mathcal{H}^2=H^2(I_{2(-n)})\times\cdots\times H^2(I_{1n}),$$
with elements viewed as column vectors, equipped with the norm
\begin{align*}
\Vert U \Vert_{\mathcal{H}^2}=\sqrt{\sum_{j\in J_1} \Vert u_{1j}\Vert^2_{H^2(I_{1j})}+\frac{1}{k}\sum_{j\in J_2} \Vert u_{2j}\Vert^2_{H^2(I_{2j})}}\,,
\end{align*}
and the subspace
\be\label{defDA}\baa{l}
\mathcal{D}(A)=\bigg\{(u_{2(-n)},\ldots, u_{1n})^T\in \mathcal{H}^2\cap\mathcal{H}^1_0~:\\   
\qquad\qquad\ \ (u_{1m})'(x)=\sigma(u_{2m})'(x),\ \ \ \ \ \,x=ml,\ \ \ \ \ \ \ m=0,\pm1,\ldots,\pm(n-1),\vspace{3pt}\\
\qquad\qquad\ \ \sigma (u_{2m})'(x)=  (u_{1(m+1)})'(x),\ x=ml+l_2,\ m=0,\pm1,\ldots,\pm(n-1),-n\bigg\}.\eaa
\ee
From the Sobolev embeddings, the set $\mathcal{H}^2$ can be viewed as a subset of $C^1(\overline{I_{2(-n)}})\times\cdots\times C^1(\overline{I_{1n}})$ and, with a slight abuse of notation, it can also be identified with the set of $\varphi$ in $L^2(-nl,nl)$ such that $\varphi|_I\in H^2(I)$ for each patch $I\subset[-nl,nl]$. As for $\mathcal{D}(A)$, one has $\mathcal{D}(A)\subset\mathcal{H}^2\cap\mathcal{H}^1_0\subset X$ and, from the Sobolev embeddings, $\mathcal{D}(A)$ is a Banach space when endowed with the norm $\|\ \|_{\mathcal{H}^2}$. With a slight abuse of notation, $\mathcal{D}(A)$ can also be identified with the set of $\varphi$ in $H^1_0(-nl,nl)$ such that~$\varphi|_I\in H^2(I)$ for each patch $I\subset[-nl,nl]$ and $\varphi$ satisfies the above flux conditions at the interior interfaces.

The proof of the well-posedness of the Cauchy problem~\eqref{tp-1}--\eqref{tp-4} is based on the following auxiliary lemma.

\begin{lemma}\label{lemma-A maximal monotone}
The linear operator $A:\mathcal{D}(A)\subset X \to X$ is symmetric maximal monotone, and~$-A$ is the infinitesimal generator of an analytic semigroup on $X$. 
\end{lemma}

\begin{proof}
For any $U=(u_{2(-n)},\ldots, u_{1n})^T\in\mathcal{D}(A)$, by using~\eqref{sigma} and the interface and boundary conditions given in the definition of $\mathcal{D}(A)$, we have
\begin{equation}
\label{A is monotone}
\begin{aligned}
\langle A U,U\rangle _{X}=&\sum_{j\in J_1} \int_{I_{1j}} (-d_1 u''_{1j}) u_{1j}+\frac{1}{k}\sum_{j\in J_2} \int_{I_{2j}} (-d_2 u''_{2j})  u_{2j}\\
= & \sum_{j\in J_1} \int_{I_{1j}} d_1 u'_{1j} u'_{1j}+\frac{1}{k}\sum_{j\in J_2} \int_{I_{2j}} d_2u'_{2j} u'_{2j}\ge\beta\|U\|_{\mathcal{H}^1_0}^2\ge0, 
\end{aligned}
\end{equation}
where $\beta>0$ is a positive constant independent of $U\in\mathcal{D}(A)\subset\mathcal{H}^1_0$, given by Poincar\'e's inequality. Therefore, $A$ is monotone. The symmetry of $A$ is also obvious from a similar calculation.
    
Next, we shall prove that,  for every $\lambda\ge0$, the range $\mathcal{R}(\lambda I_X+A)$ of $\lambda I_X+A$ is equal to $X$ ($I_X$ denotes the identity operator in $X$), that is, for any $F\in X$, there exists $U\in\mathcal{D}(A)$ such that $\lambda U+AU=F$ (such a $U$ is then unique from~\eqref{A is monotone}). For any $F=(f_{2(-n)},\ldots,f_{1n})^T\in X$, we consider the following boundary value problem:
\begin{equation}
\label{resolvent pb-1}
\begin{aligned} 
\begin{cases}
-d_2u''_{2(-n)} +\lambda u_{2(-n)}=  f_{2(-n)},~~ &\hbox{in }(-nl,-nl+l_2), \\
-d_1u''_{1m}+\lambda u_{1m} =  f_{1m},~~ &\hbox{in }(ml-l_1,ml),\ m=0,\pm 1,...,\pm (n-1),\\ 
-d_2u''_{2m}+\lambda u_{2m}=  f_{2m},~~ &\hbox{in }(ml,ml+l_2),\ m=0,\pm 1,...,\pm (n-1),\\
-d_1u''_{1n}+\lambda u_{1n}=  f_{1n},~~ &\hbox{in }(nl-l_1,nl),
\end{cases}
\end{aligned}
\end{equation}
with  interface conditions 
\begin{equation}
\label{resolvent pb-2}\left\{
\begin{aligned}
u_{1m}( x)&=u_{2m}(x),\qquad\ ~(u_{1m})'(x^-)=\sigma(u_{2m})'(x^+),\\
& \qquad\qquad\qquad\qquad\qquad\qquad x=ml,\ m=0,\pm1,...,\pm(n-1),\\
u_{2m}(x)&=u_{1(m+1)}(x),~\ \sigma (u_{2m})'(x^-)=  (u_{1(m+1)})'(x^+),\\
& \qquad\qquad\qquad\qquad\qquad\qquad x=ml+l_2,\ m=0,\pm1,...,\pm(n-1),-n,
\end{aligned}\right.
\end{equation}
and boundary conditions
\begin{equation}
\label{resolvent pb-3}
u_{2(-n)}(-nl) =  u_{1n}(nl)=0.
\end{equation}
Problem \eqref{resolvent pb-1}--\eqref{resolvent pb-3} is first converted into a weak problem, which consists in finding~$U\in\mathcal{H}^1_0$ such that 
\begin{equation}
\label{a(U,V)}
a(U,V)=\langle F, V\rangle_X\ \hbox{ for all }V\in\mathcal{H}^1_0,
\end{equation}
where the bilinear form $a$ is defined in $\mathcal{H}^1_0\times\mathcal{H}^1_0$ by
\begin{equation*}
a(U,V)=\sum_{j\in J_1} \int_{I_{1j}}\big(d_1u'_{1j}v'_{1j}+\lambda u_{1j}v_{1j}\big)+\frac{1}{k}\sum_{j\in J_2} \int_{I_{2j}}\big(d_2u'_{2j}v'_{2j}+\lambda u_{2j}v_{2j}\big).
\end{equation*}
It is clear that the bilinear form defined in $\mathcal{H}^1_0\times \mathcal{H}^1_0$ is continuous and coercive (from Poincar\'e's inequality again). Then, by the Lax-Milgram theorem, problem \eqref{a(U,V)} admits a unique solution~$U\in\mathcal{H}^1_0$, and we have
\begin{equation} \label{estimate-1}
\|U\|_{\mathcal{H}^1_0}\le C \Vert  F\Vert _X, 
\end{equation}
for some constant $C>0$ only depending on $d_{1,2}$, $k$, $n$ and $l_{1,2}$. Furthermore, owing to the definition of~$a$, the solution $U$ belongs to~$\mathcal{D}(A)$ and satisfies~\eqref{resolvent pb-1}--\eqref{resolvent pb-3}. By rewriting the equations as $u''_{ij} = (\lambda u_{ij}- f_{ij})/d_i$ for $j\in J_i$ and $i\in\{1,2\}$, and taking  $L^2$-norms on both sides, we get $\Vert u''_{ij}\Vert _{L^2(I_{ij})}\le(1/d_i)\times(\lambda\Vert  u_{ij} \Vert _{L^2(I_{ij})}+\Vert f_{ij}\Vert _{L^2(I_{ij})})$. By~\eqref{estimate-1}, we finally obtain the overall~$H^{2}$ estimate $\Vert   U\Vert _{\mathcal{H}^2}\le C' \Vert F\Vert _X$ for some constant $C'>0$ only depending on $d_{1,2}$, $k$, $n$, $l_{1,2}$ and $\lambda$. We then conclude that~$\mathcal{R}(\lambda I_X+A)=X$, that $\lambda I_X+A$ is invertible from $\mathcal{D}(A)$ onto~$X$ and that $(\lambda I_X+A)^{-1}$ is bounded from $X$ onto $\mathcal{D}(A)$. In particular, $\mathcal{R}(I_X+A)=X$ and the operator $A$ is maximal monotone. It is then densely defined and closed, and $\mathcal{D}(A)$ is also a Banach space if endowed with the graph norm~$\|U\|_{\mathcal{D}(A)}$ of $A$.

Lastly, let us show that $-A$ is the infinitesimal generator of an analytic semigroup on $X$. First of all, since $A$ is monotone, one has $\lambda\Vert u\Vert_X\le \Vert (\lambda I_X+A) u\Vert_X$
for every $\lambda\ge0$ and $u\in\mathcal{D}(A)$, hence
\begin{equation*}
\Vert(\lambda I_X+A)^{-1}\Vert_{\mathcal{L}(X)}\le \frac{1}{\lambda}~~\text{for every}~\lambda>0.
\end{equation*}
Therefore, the Hille-Yosida theorem implies that the operator $-A$ is the infinitesimal generator of a contraction semigroup on $X$. On the other hand, by viewing $A$ in the complexified Hilbert spaces associated with $X$ and $\mathcal{H}^2$, one sees from~\eqref{A is monotone} and the symmetry of $A$ that the numerical range ${\mathcal{S}(-A)}$ of $-A$ is included in $\R$ and more precisely in an interval $(-\infty,-\delta]$, for some $\delta>0$. Fix any $\theta\in(0,\pi/2)$ and denote
\begin{equation*}
\Sigma_\theta:=\{\lambda\in\mathbb{C}^*: |\arg\lambda|<\pi-\theta\}.
\end{equation*}
Then there is $C_\theta>0$ such that $\text{dist}(\lambda,\overline{\mathcal{S}(-A)})\ge C_\theta |\lambda|$ for all $\lambda\in \Sigma_\theta$, where $\text{dist}(\lambda,\overline{\mathcal{S}(-A)})$ represents the distance in $\mathbb{C}$ between $\lambda$ and $\overline{\mathcal{S}(-A)}$. We observe that $\Sigma_\theta\cap\rho(-A)\neq\emptyset$, since any $\lambda>0$ is in the resolvent set $\rho(-A)$ of $-A$ from the above analysis. Therefore, \cite[Theorem~1.3.9]{P1983} then states that $\Sigma_\theta\subset\rho(-A)$ and 
$$\Vert (\lambda I_X+A)^{-1}\Vert_{\mathcal{L}(X)} \le \frac{1}{C_\theta|\lambda|}~~\text{for all}~\lambda\in \Sigma_\theta.$$
Since $0$ belongs to $\rho(-A)$ as well, we conclude by \cite[Theorem~2.5.2]{P1983} that $-A$ is  the infinitesimal generator of an analytic semigroup on $X$. The proof of Lemma~\ref{lemma-A maximal monotone} is thereby complete.
\end{proof}

With Lemma~\ref{lemma-A maximal monotone} in hand, we are now ready to prove Theorem~\ref{tp-wellposedness} on the well-posedness of the Cauchy problem~\eqref{tp-1}--\eqref{tp-4}.

\noindent\begin{proof}[Proof of Theorem~$\ref{tp-wellposedness}$] The proof is divided into two main steps. The first one assumes an additional hypothesis on the functions $f_i$ in~\eqref{2.5}, and the second one deals with the general case of~$f_i$ satisfying~\eqref{2.5}.
\vskip 0.3cm
\noindent{\it Step 1: in addition to~\eqref{2.5}, assume that $f_1$ and $f_2$ are globally Lipschitz continuous from~$\R$ to~$\R$}. The function $f$ given in~\eqref{defAf} is then Lipschitz continuous from $X$ to $X$. Therefore, it follows from Lemma~\ref{lemma-A maximal monotone} and~\cite[Theorem~2.5.1]{Zheng2004} that, for each $U_0\in X$, problem~\eqref{abstract-nonlinear} has a unique global mild solution $U\in C([0,+\infty),X)$, satisfying
\begin{equation}
\label{7.1.30}
U(t)=e^{-tA}U_0+\int_0^t e^{-(t-s)A}f(U(s))~\mathrm{d}s
\end{equation}
for all $t\ge0$. Note that the function $t\mapsto f(U(t))$ belongs to $C([0,+\infty),X)$ as well. By~\cite[Lemma~7.1.1]{L1996}, the integral on the right-hand side of~\eqref{7.1.30} belongs to $C^{0,\gamma}_{loc}([0,+\infty),X)$ for any $\gamma\in(0,1)$. Since~$t\mapsto e^{-tA}U_0$ is of class $C^\infty((0,+\infty),\mathcal{D}(A))$ by \cite[Theorem 2.3.2]{Zheng2004}, we see that $U\in C^{0,\gamma}_{loc}((0,+\infty),X)$ for any $\gamma\in(0,1)$ and that the function $t\mapsto f(U(t))$ belongs to $C^{0,\gamma}_{loc}((0,+\infty),X)$ too for any $\gamma\in(0,1)$. It then follows from~\cite[Theorem~4.3.1]{L1996} that~$U\in C^{0,\gamma}_{loc}((0,+\infty),\mathcal{D}(A))\cap C^{1,\gamma}_{loc}((0,+\infty),X)$ for any~$\gamma\in(0,1)$. As a consequence, $U$ is a classical solution of~\eqref{abstract-nonlinear}, with equalities in $X$.

Furthermore, by Lemma \ref{lemma-A maximal monotone} and the fact that $0\in\rho(A)$, we can define fractional powers~$A^\beta$ of~$A$. For $0 <\beta\le1$, $A^\beta$ is a closed operator whose domain $\mathcal{D}(A^\beta)$ is dense in $X$ and $\mathcal{D}(A)\hookrightarrow\mathcal{D}(A^\beta)\hookrightarrow X$ continuously. Endowed with the graph norm $\|U\|_{\mathcal{D}(A^\beta)}$ of $A^\beta$, $\mathcal{D}(A^\beta)$ is a Banach space. Since $A$ is sectorial and $\inf\{\Re(\lambda):\lambda\not\in\rho(A)\}>0$, it follows that $A^\beta$ is invertible with bounded inverse $(A^\beta)^{-1}\in\mathcal{L}(X)$ and that the norm $\Vert U\Vert_{\mathcal{D}(A^\beta)}$ is equivalent to $\Vert A^\beta U\Vert_X$ in~$\mathcal{D}(A^\beta)$. From~\cite[Lemma 37.8]{SY2002}, one has, for each $1/4<\beta\le 1$ and $\delta\in(0,\min(2\beta-1/2,1))$, a continuous embedding~$\mathcal{D}(A^\beta)\hookrightarrow C^{0,\delta}([-nl,nl])$.\footnote{With a slight abuse of notation, the embedding $\mathcal{D}(A^\beta)\hookrightarrow C^{0,\delta}([-nl,nl])$ means that the elements $U=(u_{2(-n)},\ldots,u_{1n})^T$ of $\mathcal{D}(A^\beta)$ have continuous components $u_{ij}$ in each corresponding closed patch $\overline{I_{ij}}$, and that the function equal to each $u_{ij}$ on each closed patch $\overline{I_{ij}}$ is well defined, continuous in $[-nl,nl]$, vanishes at~$\pm nl$ and is H\"older continuous of exponent $\delta$ in $[-nl,nl]$, with a sup norm and a H\"older norm controlled by $\|U\|_{\mathcal{D}(A^\beta)}$.} From now on, we fix $\beta\in(1/4,1)$. We also observe that $f:\mathcal{D}(A^\beta)\to X$ is globally Lipschitz continuous: indeed, for any $U,V\in\mathcal{D}(A^\beta)$, there holds
$$\baa{rcl}
\|f(U)-f(V)\|_X\le L\|U-V\|_X & = & L\Vert (A^\beta)^{-1} A^\beta U-(A^\beta)^{-1} A^\beta V\Vert_{X}\vspace{3pt}\\
& \le & L\|(A^\beta)^{-1}\|_{\mathcal{L}(X)}\Vert A^\beta U- A^\beta V\Vert_X\vspace{3pt}\\
& \le & L\|(A^\beta)^{-1}\|_{\mathcal{L}(X)}\Vert U-V\Vert_{\mathcal{D}(A^\beta)},\eaa$$
for some constant $L\in[0,+\infty)$ independent of $U,V\in\mathcal{D}(A^\beta)$. Now, for any $U_0\in\mathcal{D}(A^\beta)\ (\subset X)$, the unique global solution $U\in C([0,+\infty),X)\cap  C^{0,\gamma}_{loc}((0,+\infty),\mathcal{D}(A))\cap C^{1,\gamma}_{loc}((0,+\infty),X)$ (for any $\gamma\in(0,1)$) of~\eqref{abstract-nonlinear}, given in the previous paragraph, also belongs to $C([0,+\infty),\mathcal{D}(A^\beta))$ and then to $C([0,+\infty), C^{0,\delta}([-nl,nl]))$ for any $\delta\in(0,\min(2\beta-1/2,1))$. Since $U$ satisfies~\eqref{7.1.30} for all $t\ge0$, we then get by~\cite[Theorem~3.5.2]{H81} and~\cite[Lemma 37.8]{SY2002} the existence of some $\eta\in(0,1)$, $\theta\in(1/4,1)$ and $\omega\in(0,\min(2\theta-1/2,1))$ such that $U\in C^{1,\eta}_{loc}((0,+\infty),\mathcal{D}(A^\theta))$ and
$$U\in C^{1,\eta}_{loc}((0,+\infty), C^{0,\omega}([-nl,nl])).$$

Since $\mathcal{D}(A)\hookrightarrow\mathcal{D}(A^\beta)$, it follows from the previous two paragraphs that, for any $U_0\in X$, the solution $U\in C([0,+\infty),X)\cap C((0,+\infty),\mathcal{D}(A))\cap C^1((0,+\infty),X)$ of~\eqref{abstract-nonlinear} belongs to
$$C((0,+\infty),\mathcal{C}_0\cap C^{0,\delta}([-nl,nl]))\cap C^{1,\eta}_{loc}((0,+\infty), C^{0,\omega}([-nl,nl])).$$
Moreover, if $U_0\in\mathcal{D}(A^\beta)$, then $U\in C([0,+\infty), C^{0,\delta}([-nl,nl]))$. One infers that, for any~$U_0\in X$, the function $u$ defined as in~\eqref{tp-4} (with similar definition at $t=0$) is continuous in~$(0,+\infty)\times[-nl,nl]$, vanishes on $(0,+\infty)\times\{\pm nl\}$, is of class~$C^1$ with respect to $t$ in $(0,+\infty)\times[-nl,nl]$, with~$u$ and~$\frac{\partial u}{\partial t}$ H\"older continuous in $[\tau,\tau']\times[-nl,nl]$ for every $0<\tau<\tau'<+\infty$. Therefore, for each patch~$I\subset[-nl,nl]$ of type $i\in\{1,2\}$ and for each $0<\tau<\tau'<+\infty$, the function~$f_i(u)$ is H\"older continuous in $[\tau,\tau']\times\overline{I}$, hence equation~\eqref{abstract-nonlinear} implies that~$u|_{[\tau,\tau']\times\overline{I}}$ is of class~$C^2$ with respect to~$x$ and~$\frac{\partial^2u|_{[\tau,\tau']\times\overline{I}}}{\partial x^2}$ is H\"older continuous in $[\tau,\tau']\times\overline{I}$. In particular, $u$ is a classical solution of~\eqref{tp-1}--\eqref{tp-4} for $t>0$.  Furthermore,  if $U_0\in\mathcal{D}(A^\beta)$, then~$u$ is also continuous in~$[0,+\infty)\times[-nl,nl]$.

It remains to show, still in this step~1, that $u$ is bounded and continuous up to $t=0$ when~$u_0\in\mathcal{C}_0$. To do so, we first prove a comparison principle for the solutions when the initial conditions are in $X$. Take any $V_0,W_0\in X$ such that $v_0\le w_0$ almost everywhere in~$[-nl,nl]$, with obvious notations for $v_0$ and $w_0$. There exist then two sequences $(V_{0j})_{j\in\N}$ and $(W_{0j})_{j\in\N}$ in~$\mathcal{D}(A)\,(\subset\mathcal{D}(A^\beta)\subset\mathcal{C}_0)$ such that $v_{0j}\le w_{0j}$ everywhere in $[-nl,nl]$ (with obvious notations) for all~$j\in\N$, and~$V_{0j}\to V_0$, $W_{0j}\to W_0$ in $X$ as $j\to+\infty$. For each $j\in\N$, with obvious notations, let~$v_j$ and~$w_j$ be the classical solutions of~\eqref{tp-1}--\eqref{tp-4} with initial conditions $v_{0j}$ and~$w_{0j}$. The functions~$v_j$ and~$w_j$ are continuous in $[0,+\infty)\times[-nl,nl]$, from the previous paragraph. Therefore, the maximum principle of Proposition~\ref{bdd-cp} implies that
$$v_j\le w_j\hbox{ in $[0,+\infty)\times[-nl,nl]$},$$
for all $j\in\N$. Since, for each $t\ge0$, the map $U_0\mapsto U(t)$ given by~\eqref{7.1.30} is continuous (and even Lipschitz continuous) from~$X$ to~$X$ by~\cite[Theorem~2.5.1]{Zheng2004}, one infers that, for each $t>0$, $v(t,\cdot)\le w(t,\cdot)$ almost everywhere in~$(-nl,nl)$ and then everywhere in $[-nl,nl]$ by continuity. To sum up,
\be\label{inequvn}
v\le w\hbox{ in }(0,+\infty)\times[-nl,nl].
\ee

If $u_0\in\mathcal{C}_0$ with $u_0\ge0$ in $[-nl,nl]$, 	without loss of generality, one can choose a sequence $(u_{0k})_{k\in\mathbb{N}}$ in  $\mathcal{D}(A)$ such that $u_{0k}\to u_0$ as $k\to +\infty$  and  $0\le u_{0k}\le\|u_0\|_{L^\infty(-nl,nl)}$ in $[-nl,nl]$ for all $k\in\N$.
  Remembering~\eqref{2.5}, the constant functions $0$ and~$\max\big(K_1,K_2,\|u_0\|_{L^\infty(-nl,nl)}\big)$ are, respectively, a subsolution and a supersolution, in the sense of Definition~\ref{defsubsuper}, of the problem~\eqref{tp-1}--\eqref{tp-4} satisfied by the continuous and classical solution~$u_k$ in~$[0,+\infty)\times[-nl,nl]$. The maximum principle of Proposition~\ref{bdd-cp} then yields $0\le u_k\le\max\big(K_1,K_2,\|u_0\|_{L^\infty(-nl,nl)}\big)$ in~$[0,+\infty)\times[-nl,nl]$ for all $k\in\N$, hence
\be\label{inequ2}
0\le u\le\max\big(K_1,K_2,\|u_0\|_{L^\infty(-nl,nl)}\big)\ \hbox{ in }(0,+\infty)\times[-nl,nl],
\ee
by passing to the limit as $k\to+\infty$ for each $t>0$, as in the previous paragraph. Notice that~\eqref{inequ2} holds as well on $\{0\}\times[-nl,nl]$ by assumption on $u_0$.

Lastly, consider again any $u_0\in\mathcal{C}_0\,(\subset X)$ in $[-nl,nl]$ and let us show that $u$ is continuous up to time $t=0$. Let $\epsilon>0$ be arbitrary. Let $\underline{U}_0$ and $\overline{U}_0$ be two functions in $\mathcal{D}(A)\,(\subset\mathcal{D}(A^\beta)\subset X)$ such that
$$u_0-\epsilon\le\underline{u}_0\le u_0\le\overline{u}_0\le u_0+\epsilon\ \hbox{ in }[-nl,nl]$$
(with obvious notations for the functions $\underline{u}_0$ and $\overline{u}_0$, which can be chosen in $C^2([-nl,nl])\cap\mathcal{C}_0$ with zero derivatives at the interior interfaces) and let $\underline{u}$ and $\overline{u}$ be the two classical solutions of~\eqref{tp-1}--\eqref{tp-4} with initial conditions $\underline{u}_0$ and $\overline{u}_0$. From the above arguments, the functions~$\underline{u}$ and~$\overline{u}$ are continuous in $[0,+\infty)\times[-nl,nl]$, and $\underline{u}\le u\le\overline{u}$ in $[0,+\infty)\times[-nl,nl]$ from~\eqref{inequvn} and the choice of the initial conditions. Finally, there is $t_0>0$ such that
$$u_0-2\epsilon\le\underline{u}_0-\epsilon\le\underline{u}\le u\le\overline{u}\le\overline{u}_0+\epsilon\le u_0+2\epsilon\ \hbox{ in }[0,t_0]\times[-nl,nl],$$
from which it follows that the $C((0,+\infty)\times[-nl,nl])$ function $u$ is also continuous up to time $t=0$. It is therefore a bounded classical solution of~\eqref{tp-1}--\eqref{tp-4} in $[0,+\infty)\times[-nl,nl]$ with initial condition $u_0$, in the sense of Definition~\ref{def2}.
\vskip 0.3cm
\noindent{\it Step 2: the general case of assumption~\eqref{2.5}.} Consider a nonnegative initial condition~$u_0$ in~$\mathcal{C}_0$. Denote $K=\max\big(K_1,K_2,\|u_0\|_{L^\infty(-nl,nl)}\big)$ and, for $i=1,2$, let $\tilde{f}_i:\R\to\R$ be a globally Lipschitz continuous function of class $C^1(\R)$ such that $\tilde{f}_i|_{{\blue [0,K]}}=f_i|_{{\blue[0,K]}}$ and $\tilde{f}_i\le0$ in $[K_i,+\infty)$. From Step~1, there is a unique bounded classical solution $u$ of~\eqref{tp-1}--\eqref{tp-4} in $[0,+\infty)\times[-nl,nl]$ with initial condition $u_0$, but with the nonlinearities $\tilde{f}_i$ instead of $f_i$, and $u$ satisfies~\eqref{inequ2} in~$[0,+\infty)\times[-nl,nl]$. From~\eqref{inequ2} and the choice of $\tilde{f}_i$, the function $u$ is then a bounded classical solution of the problem~\eqref{tp-1}--\eqref{tp-4} in $[0,+\infty)\times[-nl,nl]$ with initial condition $u_0$ and with the original nonlinearities~$f_i$.

Since the uniqueness and comparison properties in Theorem~\ref{tp-wellposedness} directly follow from Proposition~\ref{bdd-cp}, the proof of Theorem~\ref{tp-wellposedness} is thereby complete.
\end{proof}


\subsection{Proof of Theorem \ref{thm-m-wellposedness}}

We first prove the existence of a solution to the problem \eqref{m-pb}--\eqref{d-f} through a truncation and  approximation argument. Set
$$\varep:=\min\Big(\frac{l_1}{4},\frac{l_2}{4}\Big)>0.$$
We first fix a sequence of cut-off functions $(\delta^n)_{n\in\N}$ in $C(\mathbb{R})$ such that, for each $n\in\N$,
\begin{equation}
\label{cut-off func}
0\le\delta^n\le1\hbox{ in }\R,\ \hbox{ and }\ \delta^n(x):=\begin{cases}
1 & \text{if}~x\in[-nl-\varep+l_2,nl-l_1+\varep],\\
0 & \text{if}~x\notin (-nl,nl).
\end{cases}
\end{equation}
We can for instance define uniquely $\delta^n$ by also assuming that $\delta^n$ is affine in $[-nl,-nl-\epsilon+l_2]$ and in $[nl-l_1+\epsilon,nl]$.

Let us now take any nonnegative bounded continuous function $u_0:\R\to\R$. For each $n\in\N$, we consider the truncated problem \eqref{tp-1}--\eqref{tp-4} on $[-nl,nl]$, with initial condition $\delta^nu_0|_{[-nl,nl]}$. This problem involves $4n$ patches still identified  by $I_{2(-n)},I_{1(-n+1)},\ldots,I_{2(n-1)},I_{1n}$ as before, and the nonnegative initial condition $\delta^nu_0|_{[-nl,nl]}$ belongs to the space $\mathcal{C}_0$ defined in~\eqref{defC0}. Therefore, by Theorem~\ref{tp-wellposedness}, there is a unique bounded classical solution $u^n$ of~\eqref{tp-1}--\eqref{tp-4} in $[0,+\infty)\times[-nl,nl]$ with initial condition $\delta^nu_0|_{[-nl,nl]}$, and $u^n$ also satisfies~\eqref{umax} in $[0,+\infty)\times[-nl,nl]$. Moreover, for every $m<n\in\N$, one has
$$0\le u^m(0,\cdot)=\delta^mu_0|_{[-ml,ml]}\le u_0|_{[-ml,ml]}=\delta^nu_0|_{[-ml,ml]}=u^n(0,\cdot)\ \hbox{ in } [-ml,ml]$$
and $u^n(t,\pm ml)\ge0$ for all $t\ge0$ by~\eqref{umax}, hence
$$u^m\le u^n|_{[0,+\infty)\times[-ml,ml]}\ \hbox{ in $[0,+\infty)\times[-ml,ml]$}$$
by Proposition~\ref{bdd-cp}. As a consequence, for each $(t,x)\in[0,+\infty)\times\R$, the sequence $(u^n(t,x))_{n\ge|x|/l}$ is non-decreasing and ranges in $[0,K]$ with
\be\label{defK2}
K:=\max\big(K_1,K_2,\|u_0\|_{L^\infty(\R)}\big),
\ee
hence the sequence $(u^n(t,x))_{n\ge|x|/l}$ converges to a quantity $u(t,x)\in[0,K]$, that is,
\be\label{defu2}
u^n(t,x)\to u(t,x)\in[0,K]\ \hbox{ as }n\to+\infty.
\ee

Notice also that, if $u_0\not\equiv 0$, then $u^n(0,\cdot)\ge\not\equiv0$ in $[-nl,nl]$ for all $n$ large enough, hence $u^n>0$ in $(0,+\infty)\times(-nl,nl)$ for all $n$ large enough by Proposition~\ref{bdd-cp}, and finally $u(t,x)>0$ for all $(t,x)\in(0,+\infty)\times\R$ because the sequence $(u^n(t,x))_{n\ge|x|/l}$ is non-decreasing for each $(t,x)\in[0,+\infty)\times\R$.

In order to show that $u$ is a classical solution of~\eqref{m-pb}--\eqref{d-f}, we need further differential estimates on the sequence $(u^n)_{n\in\N}$. Consider any $0<\tau\le\tau'<+\infty$ and any patch $I\subset\R$. Let us assume that $I$ is of type~1, that is, $I=I_{1m}=(ml-l_1,ml)$ for some $m\in\Z$ (the case of a patch~$I$ of type~2 can be dealt with similarly). Let us fix an arbitrary $\nu\in(0,1/2)$, say for instance $\nu=1/4$. Since the solutions $u^n$ (for $n\ge|m|+1$) of~\eqref{tp-1}--\eqref{tp-4} are uniformly bounded in $[0,+\infty)\times[-(|m|+1)l,(|m|+1)l]$, it follows from standard interior parabolic estimates that
$$\sup_{n\ge|m|+1}\|u^n(\cdot,ml-l_1-\epsilon)\|_{C^{1,\nu}([\tau/2,+\infty))}+\sup_{n\ge|m|+1}\|u^n(\cdot,ml+\epsilon)\|_{C^{1,\nu}([\tau/2,+\infty))}\le C_0,$$
for some positive constant $C_0$ only depending on $\tau$, $l_{1,2}$, $d_2$, $f_2$ and $K$ given in~\eqref{defK2}, hence on $\tau$, $l_{1,2}$, $d_2$, $f_{1,2}$ and~$\|u_0\|_{L^\infty(\R)}$. Consider then two $C^3([ml-l_1-\epsilon,ml+\epsilon])$ functions $g:[ml-l_1-\epsilon,ml+\epsilon]\to[0,1]$ and $h:[ml-l_1-\epsilon,ml+\epsilon]\to[0,1]$ such that
$$\left\{\baa{l}
g(ml-l_1-\epsilon)=h(ml+\epsilon)=0,\vspace{3pt}\\
g(ml+\epsilon)=h(ml-l_1-\epsilon)=1,\vspace{3pt}\\
g'(ml-l_1)=h'(ml-l_1)=g'(ml)=h'(ml)=0.\eaa\right.$$
They can be chosen so that their $C^3([ml-l_1-\epsilon,ml+\epsilon])$ norms only depend on $l_{1,2}$. Consider now, for each $n\ge|m|+1$, the function $\tilde{u}^n$ defined in $[\tau/2,+\infty)\times[ml-l_1-\epsilon,ml+\epsilon]$ by
\be\label{deftildeun}
\tilde{u}^n(t,x)=u^n(t,x)-h(x)u^n(t,ml-l_1-\epsilon)-g(x)u^n(t,ml+\epsilon).
\ee
Each such function $\tilde{u}^n$ is continuous in $[\tau/2,+\infty)\times[ml-l_1-\epsilon,ml+\epsilon]$ and has restrictions of class $C^{1;2}_{t;x}$ in $[\tau/2,+\infty)\times[ml-l_1-\epsilon,ml-l_1]$, in $[\tau/2,+\infty)\times[ml-l_1,ml]=[\tau/2,+\infty)\times\bar I$ and in $[\tau/2,+\infty)\times[ml,ml+\epsilon]$. Furthermore, from~\eqref{tp-1}--\eqref{tp-4} and~\eqref{deftildeun}, one has
$$\left\{\baa{ll}
\displaystyle\frac{\partial\tilde{u}^n}{\partial t}=d_2\frac{\partial^2\tilde{u}^n}{\partial x^2}+\tilde{f}_2^n(t,x,\tilde{u}^n(t,x)), & t\ge\tau/2,\ x\in(ml-l_1-\epsilon,ml-l_1),\vspace{3pt}\\
\displaystyle\frac{\partial\tilde{u}^n}{\partial t}=d_1\frac{\partial^2\tilde{u}^n}{\partial x^2}+\tilde{f}_1^n(t,x,\tilde{u}^n(t,x)), & t\ge\tau/2,\ x\in(ml-l_1,ml),\vspace{3pt}\\
\displaystyle\frac{\partial\tilde{u}^n}{\partial t}=d_2\frac{\partial^2\tilde{u}^n}{\partial x^2}+\tilde{f}_2^n(t,x,\tilde{u}^n(t,x)), & t\ge\tau/2,\ x\in(ml,ml+\epsilon),\vspace{3pt}\\
\tilde{u}^n(t,ml-l_1-\epsilon)=\tilde{u}^n(t,ml+\epsilon)=0, & t\ge\tau/2,\vspace{3pt}\\
\tilde{u}^n(t,(ml-l_1)^-)=\tilde{u}^n(t,(ml-l_1)^+),  & t\ge\tau/2,\vspace{3pt}\\
\sigma\tilde{u}^n_x(t,(ml-l_1)^-)=\tilde{u}^n_x(t,(ml-l_1)^+),  & t\ge\tau/2,\vspace{3pt}\\
\tilde{u}^n(t,(ml)^-)=\tilde{u}^n(t,(ml)^+),  & t\ge\tau/2,\vspace{3pt}\\
\tilde{u}^n_x(t,(ml)^-)=\sigma\tilde{u}^n_x(t,(ml)^+),  & t\ge\tau/2,\eaa\right.$$
with
$$\baa{rcl}
\tilde{f}_i^n(t,x,s) & = & f_i\big(s+h(x)u^n(t,ml-l_1-\epsilon)+g(x)u^n(t,ml+\epsilon)\big)\vspace{3pt}\\
& & -\,h(x)u^n_t(t,ml-l_1-\epsilon)-g(x)u^n_t(t,ml+\epsilon)\vspace{3pt}\\
& & +\,d_ih''(x)u^n(t,ml-l_1-\epsilon)+d_ig''(x)u^n(t,ml+\epsilon).\eaa$$
In other words, each function $\tilde{u}^n$ solves a truncated problem similar to~\eqref{tp-1}--\eqref{tp-4}, but this time on the interval $[ml-l_1-\epsilon,ml+\epsilon]$ (with only three patches) and with nonlinearities $\tilde{f}_i^n(t,x,s)$ which are still of class $C^1$ with respect to $s$, with partial derivatives equal to $f'_i\big(s+h(x)u^n(t,ml-l_1-\epsilon)+g(x)u^n(t,ml+\epsilon)\big)$, and are now H\"older continuous of any exponent $\nu$ with respect to~$(t,x)\in[\tau/2,+\infty)\times[ml-l_1-\epsilon,ml+\epsilon]$ uniformly with respect to $s$ and $n$. Remember that $\tau'\ge\tau$, hence $\tau'-\tau/2\ge\tau/2$. Since the sequence $(\tilde{u}^n(\tau'-\tau/2,\cdot))_{n\ge|m|+1}$ is bounded in particular in $L^2(ml-l_1-\epsilon,ml+\epsilon)$, it then follows with similar notations and arguments as in the proof of Theorem~\ref{tp-wellposedness} that there is a universal constant $\gamma\in(0,1)$ such that the sequence
$$\Big(\tilde{u}^n(\tau'-\tau/4,\cdot)|_{(ml-l_1-\epsilon,ml-l_1)},\tilde{u}^n(\tau'-\tau/4,\cdot)|_{(ml-l_1,ml)},\tilde{u}^n(\tau'-\tau/4,\cdot)|_{(ml,ml+\epsilon)}\Big)_{n\ge|m|+1}$$
is bounded in the set $\mathcal{D}(A)$ (defined as in~\eqref{defDA}, but with now only three patches) and the sequences $(\tilde{u}^n)_{n\ge|m|+1}$ and $(\tilde{u}^n_t)_{n\ge|m|+1}$ are bounded in $C^\gamma([\tau',\tau'+1]\times[ml-l_1-\epsilon,ml+\epsilon])$, with bounds depending only on $\sup_{n\ge|m|+1}\|\tilde{u}^n(\tau'-\tau/2,\cdot)\|_{L^2(ml-l_1-\epsilon,ml+\epsilon)}$, $\tau$, $l_{1,2}$, $d_{1,2}$, $f_{1,2}$ and $\sigma$, hence only on $\tau$, $l_{1,2}$, $d_{1,2}$, $f_{1,2}$, $\sigma$ and $\|u_0\|_{L^\infty(\R)}$ (notice that these bounds are independent of $\tau'\in[\tau,+\infty)$). Owing to the definitions of $\tilde{f}_i^n$ and $\tilde{u}^n$, one infers that the sequence $(\tilde{f}_i^n(\cdot,\cdot,\tilde{u}^n(\cdot,\cdot)))_{n\ge|m|+1}$ is bounded in $C^\gamma([\tau',\tau'+1]\times[ml-l_1-\epsilon,ml+\epsilon])$, hence so is the sequence
$$\Big(\frac{\partial^2\tilde{u}^n|_{[\tau',\tau'+1]\times[ml-l_1-\epsilon,ml-l_1]}}{\partial x^2},\frac{\partial^2\tilde{u}^n|_{[\tau',\tau'+1]\times[ml-l_1,ml]}}{\partial x^2},\frac{\partial^2\tilde{u}^n|_{[\tau',\tau'+1]\times[ml,ml+\epsilon]}}{\partial x^2}\Big)_{n\ge|m|+1}$$
in $C^\gamma([\tau',\tau'+1]\times[ml-l_1-\epsilon,ml-l_1])\times C^\gamma([\tau',\tau'+1]\times[ml-l_1,ml])\times C^\gamma([\tau',\tau'+1]\times[ml,ml+\epsilon])$, with bounds depending only on $\tau$, $l_{1,2}$, $d_{1,2}$, $f_{1,2}$, $\sigma$ and $\|u_0\|_{L^\infty(\R)}$. Finally, using~\eqref{deftildeun} again, the sequence $(u^n|_{[\tau',\tau'+1]\times\bar I})_{n\ge m+1}$ is bounded in $C^{1,\gamma;2,\gamma}_{t;x}([\tau',\tau'+1]\times\bar I)$, and, since the bound does not depend on $\tau'\in[\tau,+\infty)$, the sequence $(u^n|_{[\tau,+\infty)\times\bar I})_{n\ge m+1}$ is bounded in $C^{1,\gamma;2,\gamma}_{t;x}([\tau,+\infty)\times\bar I)$ by a constant depending only on $\tau$, $l_{1,2}$, $d_{1,2}$, $f_{1,2}$, $\sigma$ and $\|u_0\|_{L^\infty(\R)}$.

From the Arzel\`a-Ascoli theorem and the uniqueness of the limit $u$ in~\eqref{defu2}, it follows that $u^n\to u$ as $n\to+\infty$ in $C^{1;2}_{t;x}([\tau_1,\tau_2]\times\bar I)$ for every $0<\tau_1\le\tau_2$ and every patch $I\subset\R$, hence $u$ is a bounded classical solution of~\eqref{m-pb}--\eqref{d-f} in $(0,+\infty)\times\R$. Furthermore, for every $\tau>0$ and every patch~$I\subset\R$, there holds
$$\Vert u|_{[\tau,+\infty)\times\bar I}\Vert_{C^{1,\gamma;2,\gamma}_{t;x}([\tau,+\infty)\times \bar I)}\le C,$$
for some constant $C$ only depending on $\tau$, $l_{1,2}$, $d_{1,2}$, $f_{1,2}$, $\sigma$ and $\|u_0\|_{L^\infty(\R)}$.

Next, we shall prove the continuity of the function $u$ up to time $t=0$. Fix any $x_0\in\mathbb{R}$, $R>0$, and $\eta>0$. With $K=\max\big(K_1,K_2,\Vert u_0\Vert_{L^\infty(\mathbb{R})}\big)$ as in~\eqref{defK2}, one can choose two nonnegative functions~$\underline u_0$ and~$\overline u_0$ in $C^2(\R)\cap L^\infty(\R)$ such that both $\underline u_0$ and $K-\overline u_0$ are supported in $[x_0-2R,x_0+2R]$, and such that
\begin{align}
\label{initial-1}
0\le\underline u_0(x)\le u_0(x)\le\overline u_0(x)\le K~\text{ for all }x\in[x_0-2R,x_0+2R],
\end{align}
and 
\begin{align}
\label{initial-2}
u_0(x)-\eta\le\underline u_0(x)\le u_0(x)\le\overline u_0(x)\le u_0(x)+\eta\ \text{ for all }x\in[x_0-R,x_0+R].
\end{align}
These functions $\underline u_0$ and $\overline u_0$ can also be chosen so that their derivatives vanish at all interface points in $[x_0-2R,x_0+2R]$. There are then $B>0$ large enough and $t_0>0$ small enough such that $Bt_0\le1$ and the $C^{1;2}_{t;x}([0,+\infty)\times\R)$ functions $(t,x)\mapsto\underline u_0(x)-Bt$ and $(t,x)\mapsto\overline u_0(x)+Bt$ are, respectively, a sub- and a supersolution of truncated problem~\eqref{tp-1}--\eqref{tp-4} for $(t,x)\in[0,t_0]\times[-nl,nl]$ and for any $n\in\N$ large enough so that~$[x_0-2R,x_0+2R]\subset[-nl,nl]$ and $\delta^n  u_0= u_0$ in $[x_0-2R,x_0+2R]$, where $\delta^n$ is the cut-off function defined in~\eqref{cut-off func}. Remembering~\eqref{initial-1}, the inequality~\eqref{umax} satisfied by~$u^n$ and the fact that $\underline{u}_0$ and $K-\overline{u}_0$ are supported in $[x_0-2R,x_0+2R]$ (hence, $\underline{u}_0(\pm nl)-Bt\le0\le K\le\overline{u}_0(\pm nl)+Bt$ for all $t\ge0$), it follows from Proposition~\ref{bdd-cp} that
\begin{align*}
\underline u_0(x)-Bt\le u^n(t,x)\le\overline u_0(x)+Bt\ \hbox{ for all $(t,x)\in[0,t_0]\times[-nl,nl]$}
\end{align*}
and for all $n$ large enough. By passing to the limit $n\to+\infty$ for any $(t,x)\in(0,t_0]\times\R$, one gets that
\begin{align*}
\underline u_0(x)-Bt\le u(t,x)\le\overline u_0(x)+Bt\ \hbox{ for all $(t,x)\in(0,t_0]\times\R$}.
\end{align*}
Together with~\eqref{initial-2}, there is then $t_1>0$ such that $|u(t,x)-u_0(x)|\le2\eta$ for all $(t,x)\in(0,t_1]\times[x_0-R,x_0+R]$. Finally, since $\eta>0$ was arbitrary, this shows that $u$ is continuous up to time $t=0$, and that $u(t,\cdot)\to u_0$ locally uniformly as $t\to0^+$. To sum up, $u$ is a nonnegative bounded classical solution of~\eqref{m-pb}--\eqref{d-f} in $[0,+\infty)\times\R$ with initial condition $u_0$, in the sense of Definition~\ref{def1}.

It now immediately follows from Proposition~\ref{cp} that, if $u$ and $v$ are bounded classical solutions of~\eqref{m-pb}--\eqref{d-f} in $[0,+\infty)\times\R$ with respective initial conditions $u_0$ and $v_0$ such that $0\le u_0(x)\le v_0(x)$ for all $x\in\R$, then $0\le u(t,x)\le v(t,x)$ for all $(t,x)\in[0,+\infty)\times\R$. As a consequence, the nonnegative bounded classical solution $u$ of~\eqref{m-pb}--\eqref{d-f} in $[0,+\infty)\times\R$ with initial condition $u_0$ is necessarily unique.

Let us consider in this paragraph the special case of a periodic initial condition $u_0$, that is, $u_0(x)=u_0(x+l)$ for all $x\in\R$. Since the function $(t,x)\mapsto \tilde{u}(t,x):=u(t,x+l)$ is still a nonnegative bounded classical solution of~\eqref{m-pb}--\eqref{d-f} in $[0,+\infty)\times\R$ (because the coefficients and the set $S$ of~\eqref{m-pb}--\eqref{d-f} are themselves periodic), and since $\tilde{u}(0,\cdot)=u(0,\cdot+l)=u(0,\cdot)$, the uniqueness of $u$ implies that $\tilde{u}\equiv u$ in $[0,+\infty)\times\R$, that is,
$$u(t,x)=u(t,x+l)\ \hbox{ for all $(t,x)\in[0,+\infty)\times\R$}.$$

Finally, let us show the local-in-time continuous dependence of the solutions $u$ with respect to the initial condition. We actually show more, that is, for each $T>0$, the map $u_0\mapsto u$ is Lipschitz continuous from $C^+(\R)\cap L^\infty(\R)$ to $C([0,T]\times\R)\cap L^\infty([0,T]\times\R)$ equipped with the sup norms. Consider two functions $u_0$, $v_0$ in $C^+(\R)\cap L^\infty(\R)$, and denote
$$\underline{u}_0=\max\big(0,u_0-\|u_0-v_0\|_{L^\infty(\R)}\big)\ \hbox{ and }\overline{u}_0=u_0+\|u_0-v_0\|_{L^\infty(\R)}.$$
Then, $\underline{u}_0$ and $\overline{u}_0$ are in $C^+(\R)\cap L^\infty(\R)$, and $0\le\underline{u}_0\le \min(u_0,v_0)\le\max(u_0,v_0)\le\overline{u}_0$ in~$\R$, with $\|\overline{u}_0-\underline{u}_0\|_{L^\infty(\R)}\le 2\|u_0-v_0\|_{L^\infty(\R)}$. Let $u$, $v$, $\underline{u}$ and $\overline{u}$ be the nonnegative bounded classical solutions of~\eqref{m-pb}--\eqref{d-f} in $[0,+\infty)\times\R$, with respective initial conditions $u_0$, $v_0$, $\underline{u}_0$ and $\overline{u}_0$. It follows from the previous monotonicity properties and from~\eqref{defK2}--\eqref{defu2} (applied to $\overline{u}$) that
\be\label{unvn}
0\le\underline{u}\le\min(u,v)\le\max(u,v)\le\overline{u}\le\overline{K}:=\max\big(K_1,K_2,\|\overline{u}_0\|_{L^\infty(\R)}\big)
\ee
in $[0,+\infty)\times\R$. The function $w:=\overline{u}-\underline{u}$ is bounded, continuous and nonnegative in $[0,+\infty)\times\R$ and its restriction to $(0,+\infty)\times\bar I$, for each patch $I$ of type $i\in\{1,2\}$, is of class $C^{1;2}_{t;x}((0,+\infty)\times\bar I)$ and satisfies
$$\frac{\partial w}{\partial t}=d_i\frac{\partial^2w}{\partial x^2}+f_i(\overline{u}(t,x))-f_i(\underline{u}(t,x))\le d_i\frac{\partial^2w}{\partial x^2}+Lw\ \hbox{ in }(0,+\infty)\times I,$$
with $L:=\max\big(\|f'_1\|_{L^\infty([0,\overline{K}])},\|f'_2\|_{L^\infty([0,\overline{K}])}\big)\in[0,+\infty)$. Furthermore, the function $w$ satisfies the interface conditions in~\eqref{m-pb}, since $\underline{u}$ and $\overline{u}$ do so. On the other hand, the nonnegative $C^{\infty}([0,+\infty)\times\R)$ function $(t,x)\mapsto\overline{w}(t,x):=2\|u_0-v_0\|_{L^\infty(\R)}e^{Lt}$ satisfies the interface conditions (since it is independent of $x$) in~\eqref{m-pb} and $\frac{\partial\overline{w}}{\partial t}=L\overline{w}=d_i\frac{\partial^2\overline{w}}{\partial x^2}+L\overline{w}$ in~$[0,+\infty)\times\bar{I}$, for each patch $I$ of type $i\in\{1,2\}$. Notice also that $\overline{w}$ is locally bounded with respect to $t\in[0,+\infty)$. Lastly,
$$0\le w(0,\cdot)=\overline{u}(0,\cdot)-\underline{u}(0,\cdot)=\overline{u}_0-\underline{u}_0\le2\|u_0-v_0\|_{L^\infty(\R)}=\overline{w}(0,\cdot)\ \hbox{ in $\R$}.$$
Proposition~\ref{cp} then implies that $0\le\overline{u}(t,x)-\underline{u}(t,x)=w(t,x)\le\overline{w}(t,x)$ for all $(t,x)\in[0,+\infty)\times\R$. Since $0\le\underline{u}\le\min(u,v)\le\max(u,v)\le\overline{u}$ in~$[0,+\infty)\times\R$ by~\eqref{unvn}, one finally infers that
\be\label{Lipuv}
|u(t,x)-v(t,x)|\le2\|u_0-v_0\|_{L^\infty(\R)}e^{Lt}\ \hbox{ for all }(t,x)\in[0,+\infty)\times\R.
\ee
This yields the Lipschitz continuity of the map $u_0\mapsto u$ from $C^+(\R)\cap L^\infty(\R)$ to $C([0,T]\times\R)\cap L^\infty([0,T]\times\R)$ equipped with the sup norms, for each $T>0$. As a conclusion, the proof of Theorem~\ref{thm-m-wellposedness} is complete.\hfill$\Box$


\section{Existence and uniqueness of a stationary solution}\label{Sec-thm-Liouville}

In this section, we focus on the stationary problem~\eqref{m-pb-elliptic}. In Section~\ref{sec41}, we show Theorem~\ref{thm-2.1-existence} on the existence and non-existence of a positive periodic bounded solution of~\eqref{m-pb-elliptic}. Section~\ref{sec42} is devoted to the proof of Theorem~\ref{thm-2.4-uniqueness} on the uniqueness of such a solution.


\subsection{Existence of solutions: proof of Theorem~\ref{thm-2.1-existence}}\label{sec41}

(i) Assume that \eqref{2.5} is fulfilled and that $0$ is an unstable solution of~\eqref{m-pb-elliptic}, that is, $\lambda_1< 0$, where $\lambda_1$ is the principal eigenvalue of the eigenvalue problem~\eqref{ep-0}, associated with the principal eigenfunction $\phi$. Since $f(x,\cdot)|_{I}=f_i$ is of class $C^1(\mathbb{R})$ for each $x\in\R\!\setminus\!S$ belonging to a patch $I$ of type $i\in\{1,2\}$, there exists $\kappa_0>0$ small enough such that, for all $0<\kappa\le \kappa_0$,
\begin{equation*}
f(x,\kappa\phi(x))\ge \kappa\phi(x) f_s(x,0)+\frac{\lambda_1}{2}\kappa\phi(x)~~\text{for all }x\in\mathbb{R}\!\setminus\!S,
\end{equation*}
hence $\kappa\phi$ then satisfies
\be\label{defkappa2}
-d(x)\kappa\phi''(x)-f(x,\kappa\phi(x))\le -d(x)\kappa\phi''(x)-\kappa\phi(x) f_s(x,0)-\frac{\lambda_1}{2}\kappa\phi(x)=\frac{\lambda_1}{2}\kappa\phi(x)<0
\ee
for all $x\in\mathbb{R}\setminus S$, as well as the interface conditions in~\eqref{m-pb-elliptic}. With $M>0$ as in~\eqref{2.5}, we can then fix $\kappa\in(0,\kappa_0]$ so that $\kappa\phi\le M$ in $\mathbb{R}$. Now, for each $n\in\N$, let $u^n$ be the unique bounded classical solution of~\eqref{tp-1}--\eqref{tp-4} with initial condition $u^n(0,\cdot)=M\delta^n|_{[-nl,nl]}$, with the cut-off function $\delta^n$ given in~\eqref{cut-off func}. From the proof of Theorem~\ref{thm-m-wellposedness}, the sequence $(u^n)_{n\in\N}$ converges monotonically pointwise in~$[0,+\infty)\times\R$ to a nonnegative bounded classical solution $u$ of~\eqref{m-pb}--\eqref{d-f} in $[0,+\infty)\times\R$, with initial condition $M$, and $u(t,x)=u(t,x+l)$ for all $(t,x)\in[0,+\infty)\times\R$. Furthermore, by~\eqref{2.5}, the constant $M$ is a supersolution of~\eqref{tp-1}--\eqref{tp-4} in $[0,+\infty)\times[-nl,nl]$ for each $n\in\N$, in the sense of Definition~\ref{defsubsuper}. Proposition~\ref{bdd-cp} implies that
$$u^n(t,x)\le M\ \hbox{ for all $t\ge0$ and $x\in[-nl,nl]$}.$$
In particular, for each $h\ge0$ and $n\in\N$, one has $u^n(h,x)\le M=u^{n+1}(0,x)$ for all $x\in[-nl,nl]$, together with $u^n(t+h,\pm nl)=0\le u^{n+1}(t,\pm nl)$ for all $t\ge0$. Hence $u^n(t+h,x)\le u^{n+1}(t,x)$ for all $t\ge0$ and $x\in[-nl,nl]$, by Proposition~\ref{bdd-cp} again. Therefore,
$$u(t+h,x)\le u(t,x)\ \hbox{ for all $(t,x)\in[0,+\infty)\times\R$},$$
by passing to the limit as $n\to+\infty$. In other words, the nonnegative continuous function $u$ is non-increasing in $t$, and, together with the periodicity in space and the Schauder estimates of Theorem~\ref{thm-m-wellposedness}, there is a continuous periodic solution $p:\R\to[0,M]$ of~\eqref{m-pb-elliptic} such that $u(t,\cdot)\to p$ uniformly in $\R$ as $t\to+\infty$, and $u(t,\cdot)|_{\bar I}\to p|_{\bar I}$ in $C^2(\bar I)$ for each patch $I\subset\R$. Finally, since the periodic continuous function $\kappa\phi$ has restrictions of class $C^2(\bar I)$ for each patch $I\subset\R$ and satisfies~\eqref{defkappa2} and the interface conditions in~\eqref{m-pb-elliptic} (it is a subsolution of this problem), and since $\kappa\phi\le M=u(0,\cdot)$ in $\R$, Proposition~\ref{cp} implies that $\kappa\phi(x)\le u(t,x)$ for all $t\ge0$ and $x\in\R$, hence $\kappa\phi(x)\le p(x)$ for all $x\in\R$ at the limit $t\to+\infty$. As a conclusion, there exists a positive and periodic continuous solution $p$ of \eqref{m-pb-elliptic} satisfying $\kappa\phi\le p\le M$ in $\mathbb{R}$, that is, Theorem~\ref{thm-2.1-existence}~(i) is proved.

(ii) Next, in addition to~\eqref{2.5}, we assume that~\eqref{2.4} holds, that $p$ is a nonnegative bounded continuous solution to the elliptic problem~\eqref{m-pb-elliptic}, and that $0$ is a stable solution of \eqref{m-pb-elliptic}, that is,~$\lambda_1\ge 0$. Let $\phi$ be the unique positive solution of~\eqref{ep-0}. In~\eqref{2.4}, let us assume that $s\mapsto f_1(s)/s$ is decreasing with respect to $s>0$ (the case when $s\mapsto f_2(s)/s$ is decreasing with respect to~$s>0$ can be handled similarly). We infer that, for every $\gamma>0$,
$$f(x,\gamma\phi(x))=f_1(\gamma\phi(x))<f'_1(0)\gamma\phi(x)=f_s(x,0)\gamma\phi(x)\hbox{ for all $x\in(nl-l_1,nl)$ and $n\in\Z$},$$
while
$$f(x,\gamma\phi(x))=f_2(\gamma\phi(x))\le f'_2(0)\gamma\phi(x)=f_s(x,0)\gamma\phi(x)\hbox{ for all $x\in(nl,nl+l_2)$ and $n\in\Z$}.$$
Hence, for all $\gamma>0$,
\be\label{ineqphi}\left\{\baa{l}
\!\!\!-d_1\gamma\phi''(x)\!-\!f_1(\gamma\phi(x))>-d_1\gamma\phi''(x)\!-\!f_1'(0)\gamma\phi(x)=\lambda_1 \gamma\phi(x)\ge0,~x\!\in\!(nl\!-\!l_1,nl),\vspace{3pt}\\
\!\!\!-d_2\gamma\phi''(x)\!-\!f_2(\gamma\phi(x))\ge-d_2\gamma\phi''(x)\!-\!f_2'(0)\gamma\phi(x)=\lambda_1 \gamma\phi(x)\ge0,~x\!\in\!(nl,nl\!+\!l_2).\eaa\right.
\ee
Since $\phi$ is bounded from below by a positive constant (because it is positive, periodic and continuous), and since $p$ is bounded, one can define
\begin{equation*}
\gamma^*=\inf\big\{\gamma>0, \gamma\phi >p~~\text{in}~\mathbb{R}\big\}\ \in[0,+\infty).
\end{equation*}
Our goal is to show that $\gamma^*=0$. Assume by way of contradiction that $\gamma^*>0$, and set $z:=\gamma^*\phi-p$. Then $z\ge 0$ in $\R$ and there exists a sequence $(x_m)_{m\in\mathbb{N}}$ in $\mathbb{R}$ such that $z(x_m)\to 0$ as $m\to +\infty$. Moreover, $z$ satisfies
\begin{equation}
\label{3.5}
\begin{aligned}
\begin{cases}
-d_1z''(x)-b(x)z(x)>0,~~&x\in(nl-l_1,nl),\cr
-d_2z''(x)-b(x)z(x)\ge 0,~~&x\in(nl,nl+l_2),\cr
z(x^-)=z(x^+),~z'(x^-)=\sigma z'(x^+),~~&x=nl,\cr
z(x^-)=z(x^+),~\sigma z'(x^-)= z'(x^+),~~&x=nl+l_2,\cr
\end{cases}
\end{aligned}
\end{equation}
for some  bounded function $b$ defined in $\mathbb{R}\!\setminus\!S$.

Assume at first that up to a subsequence, $x_m\to \bar x\in  \mathbb{R}$ as $m\to+\infty$. By continuity of $\phi$ and $p$, one has $z(\bar x)=0$. We distinguish two cases. Assume first that $\bar x\in  \mathbb{R}\backslash S$. It is easily seen from the strong elliptic maximum principle and the Hopf lemma, applied by induction from one patch to an adjacent one, that $z\equiv 0$ in $\mathbb{R}$. This is a contradiction with the strict inequality in the first line of~\eqref{3.5}. Thus, $z>0$ in $\mathbb{R}\backslash S$ and $\bar x\in S$, hence the Hopf lemma yields $z'(\bar x^-)<0$ and $z'(\bar x^+)>0$, contradicting the interface condition in~\eqref{3.5}.

In the general case, let $\bar x_m\in (-l_1,l_2]$ be such that $x_m-\bar x_m\in l\mathbb{Z}$. Then up to some subsequence, one can assume that there is $\bar x_\infty\in [-l_1,l_2]$ such that $\bar x_m\to\bar x_\infty$ as $m\to+\infty$. Set $z_m=\gamma^*\phi_m-p_m=\gamma^*\phi-p_m$, where $\phi_m(x):=\phi(x+x_m-\bar x_m)=\phi(x)$ and $p_m(x):=p(x+x_m-\bar x_m)$. Since $d(x)$ and $f(x,\cdot)$ are periodic in $x$, one then infers from~\eqref{m-pb-elliptic} and~\eqref{ineqphi} that each function~$z_m$ satisfies
\begin{equation*}
\begin{aligned}
\begin{cases}
-d_1z''_m(x)-f_1(\gamma^*\phi(x))+f_1(p_m(x))>0,~~&x\in(nl-l_1,nl),\cr
-d_2z''_m(x)-f_2(\gamma^*\phi(x))+f_2(p_m(x))\ge 0,~~&x\in(nl,nl+l_2),\cr
z_m(x^-)=z_m(x^+),~z_m'(x^-)=\sigma z_m'(x^+), &x=nl,\cr
z_m(x^-)=z_m(x^+),~\sigma z_m'(x^-)= z_m'(x^+), &x=nl+l_2.
\end{cases}
\end{aligned}
\end{equation*}
From standard elliptic estimates, it follows that up to some subsequence, the sequences $(p_m)_{m\in\N}$ and $(z_m)_{m\in\N}$ converge as $m\to+\infty$ to some functions $p_\infty$ and $z_\infty$ locally uniformly in $\R$, and in~$C^2(\bar I)$ for each patch $I\subset\mathbb{R}$, with $z_\infty=\gamma^*\phi-p_\infty$ and
\begin{equation}
\label{3.6}
\begin{aligned}
\begin{cases}
-d_1z''_\infty(x)-f_1(\gamma^*\phi(x))+f_1(p_\infty(x))\ge0,~~&x\in(nl-l_1,nl),\cr
-d_2z''_\infty(x)-f_2(\gamma^*\phi(x))+f_2(p_\infty(x))\ge 0,~~&x\in(nl,nl+l_2),\cr
z_\infty(x^-)=z_\infty(x^+),~z_\infty'(x^-)=\sigma z_\infty'(x^+), &x=nl,\cr
z_\infty(x^-)=z_\infty(x^+),~\sigma z_\infty'(x^-)= z_\infty'(x^+), &x=nl+l_2.
\end{cases}
\end{aligned}
\end{equation}
Moreover, $z_\infty\ge 0$ in $\R$, $z_\infty(\bar x_\infty)=0$, and the first inequality in~\eqref{3.6} is actually strict from the strict sign in the first line of~\eqref{ineqphi} applied with $\gamma^*>0$, and from the periodicity of $\phi$. From similar lines as above, one reaches a contradiction by using the strong elliptic maximum principle and the Hopf lemma together with the interface conditions in~\eqref{3.6}. 

Consequently, $\gamma^*=0$, whence $p\equiv 0$. This completes the proof of Theorem~\ref{thm-2.1-existence}.\hfill$\Box$


\subsection{Uniqueness of solutions: proof of Theorem~\ref{thm-2.4-uniqueness}}\label{sec42}

In order to prove the uniqueness of the positive stationary solution, we show the following crucial property.

\begin{proposition}\label{prop-3.2}
Assume~\eqref{2.5} and that $0$ is an unstable solution of \eqref{m-pb-elliptic} $($i.e., $\lambda_1<0$$)$. Let~$p$ be a bounded nonnegative continuous solution of the stationary problem \eqref{m-pb-elliptic}. Then, either~$p\equiv 0$ in $\R$, or $\inf_{\mathbb{R}}p>0$.
\end{proposition}

Note that Proposition~\ref{prop-3.2} also holds in particular in the class of periodic solutions. However, we look in Theorem~\ref{thm-2.4-uniqueness} at the uniqueness within a more general class of functions which are not assumed to be \textit{a priori} periodic. Proposition~\ref{prop-3.2}, which implies that any positive solution of~\eqref{m-pb-elliptic} is in fact bounded from below by a positive constant, will be the essence in proving  uniqueness under the additional assumption~\eqref{2.4}.

We prove Proposition \ref{prop-3.2} via a series of lemmas.  First of all, for any $R>0$ and $y\in\R$, we claim that there exist a unique real number (principal eigenvalue) $\lambda^y_R$ and a unique nonnegative continuous and piecewise smooth function (principal eigenfunction) $\varphi^y_R$ in $[-R,R]$ satisfying 
\begin{equation}
\label{3.10}
\begin{aligned}
\begin{cases}
-d(x+y)(\varphi^y_R)''(x)-f_s(x+y,0)\varphi^y_R(x)=\lambda^y_R\varphi^y_R(x),~~~&x\in(-R,R)\backslash(S-y),\\
\varphi^y_R(x^-)=\varphi^y_R(x^+),~~~(\varphi^y_R)'(x^-)=\sigma(\varphi^y_R)'(x^+),~~&x=nl-y\in(-R,R),\\
\varphi^y_R(x^-)=\varphi^y_R(x^+),~\sigma(\varphi^y_R)'(x^-)=(\varphi^y_R)'(x^+),~~~~&x=nl+l_2-y\in(-R,R),\\
\varphi^y_R>0~\text{in}~(-R,R),~
\varphi^y_R(\pm R)=0,~
\Vert\varphi^y_R\Vert_{L^\infty(-R,R)}=1.
\end{cases}
\end{aligned}
\end{equation}

We sketch the proof below. For convenience, we denote by $J_s$ and $K_r$ the finitely many shifted (by $-y$) patches of type 1 and of type 2 in $(-R,R)$ so that
$$(-R,R)\backslash(S-y)=\Big(\bigcup_s J_s\Big)\cup\Big(\bigcup_r K_r\Big).$$
The functions $d(\cdot+y)$ and $f_s(\cdot+y,0)$ are now constant in each patch $J_s$ or $K_r$. Consider the Hilbert space $H=H^1_0(-R,R)$ and the Banach space
$$G=\big\{u\in C([-R,R]):u|_{\overline{J_s}}\in C^1(\overline{J_s}),u|_{\overline{K_r}}\in C^1(\overline{K_r}), u(\pm R)=0\big\},$$
equipped with the norms
$$\left\{\baa{rcl}
\Vert u\Vert_H & = & \displaystyle\sqrt{\sum\limits_s\Vert u\Vert^2_{H^1(J_s)}+\sum\limits_r\frac{1}{k}\Vert u\Vert^2_{H^1(K_r)}}\vspace{3pt}\\
& = & \displaystyle\sqrt{\sum\limits_s\int_{J_s}\big(|u'|^2+u^2\big)+\sum\limits_r\frac{1}{k}\int_{K_r}\big(|u'|^2+u^2\big)}\ge\min\Big(1,\frac1{\sqrt{k}}\Big)\|u\|_{H^1(-R,R)},\vspace{3pt}\\
\Vert u \Vert_G & = & \displaystyle\sum\limits_s\Vert u \Vert_{C^1(\overline{J_s})}+\sum\limits_r\frac{1}{k}\Vert u \Vert_{C^1(\overline{K_r})}.\eaa\right.$$
Set
\be\label{defLambda}
\Lambda:=\max\big(f'_1(0),f'_2(0)\big)+1.
\ee
For $g\in G$, let us solve the following  problem
\begin{equation}
\label{3.10'}
\begin{aligned}
\begin{cases}
-d(\cdot+y) u''-f_s(\cdot+y,0) u+\Lambda u=g,~~~&\text{in}~(-R,R)\backslash(S-y),\\
 u(x^-)= u(x^+),~~~ u'(x^-)=\sigma u'(x^+),~~&x=nl-y\in(-R,R),\\
 u(x^-)= u(x^+),~\sigma u'(x^-)= u'(x^+),~~~~&x=nl+l_2-y\in(-R,R),\\
 u(\pm R)=0,
\end{cases}
\end{aligned}
\end{equation}
first in the weak sense: that is, we look for a solution $u\in H$ such that 
\be\label{weakB}
B(u,z)=\langle g, z \rangle~~\text{for all}~z\in H,
\ee
where the bilinear form $B$ is defined by
\begin{align}\label{defB}
B(u,z)=\sum\limits_s\int_{J_s}\big(d_1 u'z'+(\Lambda-f'_1(0))uz\big)+\sum\limits_r \frac{1}{k}\int_{K_r} \big(d_2 u'z'+(\Lambda-f'_2(0))uz\big),
\end{align}
and 
$$\langle g, z \rangle=\sum\limits_s\int_{J_s} gz+\sum\limits_r\frac{1}{k}\int_{K_r} gz.$$
Clearly, the map $z\mapsto\langle g,z\rangle$ is continuous in $H$, and $B$ is continuous in $H\times H$. Moreover, it is easily seen that, for any $u\in H$,
\begin{align*}
B(u,u)&=\sum\limits_s\int_{J_s}\big(d_1 (u')^2+(\Lambda-f'_1(0))u^2\big)+\sum\limits_r\frac{1}{k}\int_{K_r}\big(d_2 (u')^2+(\Lambda-f'_2(0))u^2\big)\\
&\ge \min\big(d_1,d_2,1\big)\Vert u \Vert^2_H,
\end{align*}
whence $B$ is coercive. The Lax-Milgram theorem yields the existence of a unique $u \in H$ (hence,~$u$ is continuous in $[-R,R]$, by identifying $u$ with its unique continuous representative, and $u(\pm R)\!=\!0$) satisfying~\eqref{weakB}, and $\|u\|_{H^1(-R,R)}\le C_1\|g\|_{L^2(-R,R)}$ for a positive constant $C_1$ only depending on $d_{1,2}$ and~$k$. Thus,
$$\max_{[-R,R]}|u|\le C_2\|g\|_{L^2(-R,R)}\le C_3\|g\|_G$$
for some positive constants $C_2$ and $C_3$ only depending on $d_{1,2}$, $k$, and $R$. Furthermore, owing to the definitions of $B$ in~\eqref{defB} and of $\sigma$ in~\eqref{sigma}, the function $u$ has restrictions in $\mathcal{I}$ belonging to~$H^2(\mathcal{I})$ for each patch $\mathcal I$ of the type $J_s$ or $K_r$ in $(-R,R)\backslash(S-y)$ and $u$ satisfies the equations and the interface conditions in~\eqref{3.10'}. In particular, $-d_iu''+(\Lambda-f'_i(0))u=g$ in $L^2(\mathcal{I})$ for each patch~$\mathcal{I}$ of the type $J_s$ or $K_r$ in $(-R,R)\!\setminus\!(S-y)$, hence $\max_{\mathcal{I}\in(\cup_sJ_s)\cup(\cup_rK_r)}\|{u|_{\mathcal{I}}}''\|_{L^2(\mathcal{I})}\le C_4\|g\|_{L^2(-R,R)}$ for a positive constant $C_4$ only depending on $d_{1,2}$, $k$ and $f'_{1,2}(0)$, while the equations satisfied by~$u$ in each patch $\mathcal{I}$ and the previous estimates also imply that $u|_{\overline{\mathcal{I}}}$ in $C^2(\overline{\mathcal{I}})$ and
$$\max_{\mathcal{I}\in(\cup_sJ_s)\cup(\cup_rK_r)}\big(\max_{\overline{\mathcal{I}}}|{u|_{\overline{\mathcal{I}}}}''|\big)\le C_5\|g\|_G$$
for a positive constant $C_5$ only depending on $d_{1,2}$, $k$, $R$ and $f'_{1,2}(0)$. Notice in particular that~$u$ then belongs to $G$. Using again the equations satisfied by $u$ and the fact that $u$ necessarily has an interior critical (with vanishing derivative) point in $(-R,R)$ thanks to~\eqref{3.10'}, it follows that $\max_{\mathcal{I}\in(\cup_sJ_s)\cup(\cup_rK_r)}\big(\max_{\overline{\mathcal{I}}}|{u|_{\overline{\mathcal{I}}}}'|\big)\le C_6\|g\|_{L^2(-R,R)}\le C_7\|g\|_G$ for some positive constants~$C_6$ and~$C_7$ only depending on $d_{1,2}$, $k$, $R$ and $f'_{1,2}(0)$, and finally that $u|_{\overline{\mathcal{I}}}\in C^3(\overline{\mathcal{I}})$ for each patch $\mathcal I$ of the type $J_s$ or $K_r$ in $(-R,R)\backslash(S-y)$ and
$$\max_{\mathcal{I}\in(\cup_sJ_s)\cup(\cup_rK_r)}\|u|_{\overline{\mathcal{I}}}\|_{C^3(\overline{\mathcal{I}})}\le C_8\|g\|_G$$
for a positive constant $C_8$ only depending on $d_{1,2}$, $k$, $R$ and $f'_{1,2}(0)$.

The mapping $T$: $g\in G \mapsto Tg:=u\in G$ is obviously linear. The previous estimates and the Arzel\`a-Ascoli theorem yield the compactness of~$T$. Let now $\mathcal{K}$ be the cone $\mathcal{K}=\{u\in G:u\ge 0~\text{in}~[-R,R]\}$. Its interior $\mathring{\mathcal{K}}$ is not empty, and $\mathcal{K}\cap(-\mathcal{K})=\{0\}$.  We claim that, if $g\in\mathcal{K}\!\setminus\!\{0\}$, then $u\in\mathring{\mathcal{K}}$. Indeed, by using $z:=u^-=\max(-u,0)\in H$ in~\eqref{weakB} one has
$$-\sum\limits_s\!\int_{J_s}\!\!\!\big(d_1((u^-)')^2+(\Lambda-f'_1(0))(u^-)^2\big)-\sum\limits_r\frac{1}{k}\!\int_{K_r}\!\!\!\big(d_2((u^-)')^2+(\Lambda-f'_2(0))(u^-)^2\big)\!=\!\langle g,u^-\rangle\!\ge\!0,$$
hence $u^-\equiv 0$, that is, $u\ge 0$ in $[-R,R]$. By finitely many applications of the strong elliptic maximum principle and the Hopf lemma, together with the fact that $g\ge\not\equiv0$, one concludes that~$u>0$ in~$(-R,R)$ and that $u'(-R)>0$ and $u'(R)<0$. Therefore, $T(\mathcal{K}\!\setminus\!\{0\})\subset\mathring{\mathcal{K}}$. From the Krein-Rutman theory, there exist a unique positive real number~$\tilde\lambda^y_R$ and a unique function~$\varphi^y_R\in\mathring{K}$ such that $\tilde\lambda^y_R T\varphi^y_R=\varphi^y_R$ and, say, $\|\varphi_R^y\|_{L^\infty(-R,R)}=1$. Hence, $\varphi^y_R$ is continuous in $[-R,R]$, of class $C^3(\overline{\mathcal{I}})$ (and then $C^\infty(\overline{\mathcal{I}})$ by induction) for each patch $\mathcal{I}$ of the type $J_s$ or $K_s$ in $(-R,R)\backslash(S-y)$, and
\begin{equation*}
\begin{aligned}
\begin{cases}
-d(x+y)(\varphi^y_R)''(x)-f_s(x+y,0)\varphi^y_R(x)+\Lambda\varphi^y_R(x)=\tilde\lambda^y_R\varphi^y_R(x), &x\in(-R,R)\backslash(S-y),\\
\varphi^y_R(x^-)=\varphi^y_R(x^+),~~~(\varphi^y_R)'(x^-)=\sigma(\varphi^y_R)'(x^+), &x=nl-y\in(-R,R),\\
\varphi^y_R(x^-)=\varphi^y_R(x^+),~\sigma(\varphi^y_R)'(x^-)=(\varphi^y_R)'(x^+), &x=nl+l_2-y\in(-R,R),\\
\varphi^y_R>0~\text{in}~(-R,R),~\varphi^y_R(\pm R)=0,~
\Vert\varphi^y_R\Vert_{L^\infty(-R,R)}=1.
\end{cases}
\end{aligned}
\end{equation*}

Therefore, $\lambda^y_R:=\tilde\lambda^y_R-\Lambda>-\Lambda=-\max\big(f'_1(0),f'_2(0)\big)-1$ is the first eigenvalue of~\eqref{3.10} associated with a unique continuous function $\varphi^y_R$ in $[-R,R]$ that is positive in $(-R,R)$ and of class $C^\infty(\overline{\mathcal{I}})$ for each patch $\mathcal{I}$ of the type $J_s$ or $K_s$ in $(-R,R)\backslash(S-y)$. Furthermore, for each~$R>0$, the interval $(-R,R)$ contains at least a patch $\mathcal{I}$ of the type $J_s$ or $K_s$ of length larger than or equal to $\ell_R:=\min(l_1,l_2,R)$, hence~$\lambda^y_R\le\max(-f'_1(0)+d_1\pi^2/\ell_R^2,-f'_2(0)+d_2\pi^2/\ell_R^2)$ from the positivity of $\varphi^y_R$ in $\mathcal{I}$. As a consequence, for each $R>0$, there is a constant $N_R$ such that
$$|\tilde\lambda^y_{R'}|+|\lambda^y_{R'}|\le N_R\hbox{ for every $R'\ge R$ and $y\in\R$}.$$
Since both $\lambda^y_R$ and $\varphi^y_R$ are unique, the aforementioned estimates and compactness arguments imply that, for each $R>0$, the maps $y\mapsto \lambda^y_R$ and $y\mapsto \varphi^y_R$ are continuous in $\R$ (the continuity of $y\mapsto\varphi^y_R$ is understood in the sense of the uniform topology in $[-R,R]$). Note also that since~$d$ and $f$ are periodic in $x$, it follows that $\lambda^y_R$ and $\varphi^y_R$ are periodic with respect to $y$ as well.

Similarly, for each $y\in\R$, there exist a unique principal eigenvalue $\lambda^y$ and a unique principal eigenfunction $\phi^y$ of the periodic problem
\begin{equation}
\label{3.11}
\begin{aligned}
\begin{cases}
-d(x+y)(\phi^y)''(x)-f_s(x+y,0)\phi^y(x)=\lambda^y\phi^y(x),\quad&x\in\mathbb{R}\backslash(S-y),\\
\phi^y(x^-)=\phi^y(x^+),~~~(\phi^y)'(x^-)=\sigma(\phi^y)'(x^+),~&x=nl-y,\\
\phi^y(x^-)=\phi^y(x^+),~\sigma(\phi^y)'(x^-)=(\phi^y)'(x^+),~~~&x=nl+l_2-y,\\
\phi^y~\text{is periodic},~ \phi^y>0~\text{in}~\mathbb{R},
~\Vert\phi^y\Vert_{L^\infty(\mathbb{R})}=1,
\end{cases}
\end{aligned}
\end{equation}
where $\phi^y$ is continuous in $\R$ and $\phi^y|_{\overline{\mathcal{I}}}$ is of class $C^\infty(\overline{\mathcal{I}})$ for each shifted patch $\mathcal{I}\subset\R\!\setminus\!(S-y)$. First of all, it is straightforward to observe:

\begin{lemma}\label{lem2}
The principal eigenvalue $\lambda^y$ of~\eqref{3.11} does not depend on $y$, that is, $\lambda^y=\lambda^0=\lambda_1$ for all $y\in\mathbb{R}$, where $\lambda_1$ is the principal eigenvalue of the eigenvalue problem~\eqref{ep-0}.
\end{lemma}

\begin{proof}
Setting $\phi(x):=\phi^y(x-y)$ for $x\in\R$, the function $\phi$ satisfies
\begin{equation*}
\begin{aligned}
\begin{cases}
-d(x)\phi''(x)-f_s(x,0)\phi(x)=\lambda^y\phi(x),~~~~~&x\in\mathbb{R}\backslash S,\\
\phi(x^-)=\phi(x^+),~~~\phi'(x^-)=\sigma\phi'(x^+),~~&x=nl,\\
\phi(x^-)=\phi(x^+),~\sigma\phi'(x^-)=\phi'(x^+),~~~~&x=nl+l_2,\\
\phi~\text{is periodic}, \phi>0~\text{in}~\mathbb{R},
\Vert\phi\Vert_{L^\infty(\mathbb{R})}=1.
\end{cases}
\end{aligned}
\end{equation*}
By uniqueness of the principal eigenvalue, one then has $\lambda^y=\lambda^0=\lambda_1$.
\end{proof}

The second lemma provides a comparison between $\lambda^y_R$ and $\lambda_1$.

\begin{lemma}
\label{lemma-3.4}
For all $y\in\mathbb{R}$ and $R>0$, one has $\lambda^y_R>\lambda_1$.
\end{lemma}

\begin{proof}
Fix any $y\in\R$ and $R>0$, and assume by way of contradiction that $\lambda^y_R\le \lambda_1$. Notice that the continuous function $\varphi^y_R$ satisfies
\begin{equation}
\label{3.13}
\begin{aligned}
\begin{cases}
-d(x\!+\!y)(\varphi^y_R)''\!-\!f_s(x\!+\!y,0)\varphi^y_R\!-\!\lambda_1\varphi^y_R\!=\!(\lambda^y_R\!-\!\lambda_1)\varphi^y_R, &\text{in}~(-R,R)\backslash(S-y),\\
\varphi^y_R(x^-)=\varphi^y_R(x^+),~~~(\varphi^y_R)'(x^-)=\sigma(\varphi^y_R)'(x^+), &x=nl-y\in(-R,R),\\
\varphi^y_R(x^-)=\varphi^y_R(x^+),~\sigma(\varphi^y_R)'(x^-)=(\varphi^y_R)'(x^+), &x=nl+l_2-y\in(-R,R),\\
\varphi^y_R>0~\text{in}~(-R,R),~~\varphi^y_R(\pm R)=0,~~
\Vert\varphi^y_R\Vert_{L^\infty(-R,R)}=1,
\end{cases}
\end{aligned}
\end{equation}
while the continuous solution $\phi^y=\phi(\cdot+y)$ of~\eqref{3.11} satisfies
\begin{equation*}
\begin{aligned}
\begin{cases}
-d(x+y)(\phi^y)''-f_s(x+y,0)\phi^y-\lambda_1\phi^y=0, &\text{in}~(-R,R)\backslash(S-y),\\
\phi^y(x^-)=\phi^y(x^+),~~~(\phi^y)'(x^-)=\sigma(\phi^y)'(x^+), &x=nl-y\in(-R,R),\\
\phi^y(x^-)=\phi^y(x^+),~\sigma(\phi^y)'(x^-)=(\phi^y)'(x^+), &x=nl+l_2-y\in(-R,R),\\
\phi^y>0,~\qquad\quad~&\text{in}~[-R,R],
\end{cases}
\end{aligned}
\end{equation*}
where $\phi$ is the principal eigenfunction of~\eqref{ep-0}. Therefore, $\phi^y>\kappa\varphi^y_R$ in $[-R,R]$ for all $\kappa>0$ small enough. Define
\begin{align*}
\kappa^*=\sup\big\{\kappa>0: \phi^y>\kappa\varphi^y_R~\text{in}~[-R,R]\big\}\ \in(0,+\infty).
\end{align*}
By continuity, $\phi^y\ge\kappa^*\varphi^y_R$ in $[-R,R]$ and there exists $x_0\in[-R,R]$ such that $\phi^y(x_0)=\kappa^*\varphi^y_R(x_0)$. But since $\phi^y>0$ in $[-R,R]$ and $\varphi^y_R(\pm R)=0$, one infers  that $x_0\in(-R,R)$. On the  other hand,  the function $\kappa^*\varphi^y_R$ satisfies \eqref{3.13}, hence
\begin{align*}
-d(x+y)(\kappa^*\varphi^y_R)''-f_s(x+y,0)\kappa^*\varphi^y_R-\lambda_1\kappa^*\varphi^y_R\le0~\text{in}~(-R,R)\backslash(S-y),
\end{align*}
thanks to the assumption $\lambda^y_R\le \lambda_1$. It then follows from finitely many applications of the strong elliptic maximum principle and the Hopf lemma that $\kappa^*\varphi^y_R\equiv \phi^y$ in $(-R,R)$ and then in $[-R,R]$ by continuity, a contradiction with the boundary conditions at $x=\pm R$. Consequently, $\lambda^y_R>\lambda_1$ for all~$y\in\mathbb{R}$ and $R>0$.
\end{proof}

For any two positive real numbers  $R_1<R_2$, by replacing $\lambda^y_R$ with $\lambda^y_{R_1}$ and $\lambda_1$ with $\lambda^y_{R_2}$ in the above proof, and by noticing that $\varphi^y_{R_2}>0$ in $[-R_1,R_1]$,  a similar argument as that of Lemma~\ref{lemma-3.4} implies that $\lambda^y_{R_1}>\lambda^y_{R_2}$. That is,

\begin{lemma}
\label{lemma-3.5}
For all $y\in\mathbb{R}$, the function $R\mapsto \lambda^y_R$ is decreasing in $R>0$.
\end{lemma}

The last result before the proof of Proposition~\ref{prop-3.2} is the following convergence result.

\begin{lemma}
\label{lemma-3.6}
One has $\lim_{R\to+\infty}\lambda^y_R=\lambda_1$ uniformly in $y\in\mathbb{R}$.
\end{lemma}

\begin{proof}
First of all, from Lemma~\ref{lemma-3.5} and the periodicity and continuity with respect to $y$, it is sufficient to show that $\lambda^y_R\to\lambda_1$ as $R\to+\infty$ for each $y\in\R$, from Dini's theorem. So let us fix~$y\in\R$ in the proof. For $R>0$, consider the elliptic operator $\mathcal{L}^y_Ru:=-d(x+y)u''-f_s(x+y,0)u$ with domain
$$E_R=\big\{\psi\in H^1_0(-R,R):\psi|_{\mathcal{I}}\in H^2(\mathcal{I})\hbox{ for all $\mathcal{I}$, and $\psi$ satisfies the interface conditions in}~\eqref{3.10}\big\},$$
where $\mathcal{I}$ is any patch of the type $J_s$ or $K_r$ in $(-R,R)\!\setminus\!(S-y)$. Note in particular that the principal eigenfunction $\varphi^y_R$ of~\eqref{3.10} belongs to $E_R\!\setminus\!\{0\}$. Owing to~\eqref{sigma}, the operator $\mathcal{L}^y_R$ is symmetric with respect to the inner product
\begin{align*}
\langle u,v \rangle_R=\sum\limits_s\int_{J_s} uv+\sum\limits_r\frac{1}{k}\int_{K_r} uv,\ \ u,v\in L^2(-R,R).
\end{align*}
Therefore, one has the following variational formula for $\lambda^y_R$:
\begin{align}
\label{3.15}
\lambda^y_R=\min_{\psi\in E_R\!\setminus\!\{0\}}\frac{B(\psi,\psi)-\Lambda\langle \psi,\psi \rangle_R}{\langle \psi,\psi \rangle_R}=\min_{\psi\in E_R\!\setminus\!\{0\}}\frac{\langle \mathcal{L}^y_R\psi,\psi \rangle_R}{\langle \psi,\psi \rangle_R},
\end{align}
with $\Lambda$ and $B$ as in~\eqref{defLambda} and~\eqref{defB}. It is easy to see that one can choose a family of $C^2(\R)$ functions $(\chi^y_R)_{R\ge 2}$, such that $\sup_{R\ge 2}\|\chi^y_R\|_{C^2(\R)}<+\infty$ and, for each $R\ge 2$,
\begin{align*}
\begin{cases}
\chi^y_R(x)=1 & \text{for all}~x\in[-R+1,R-1],\\
\chi^y_R(x)=0 & \text{for all}~x\in(-\infty,-R]\cup[R,+\infty),\\
(\chi^y_R)'(x)=0 & \text{for all }x\in S-y,\\
0\le \chi^y_R\le 1 & \text{in }\R.
\end{cases}
\end{align*}
Let $\phi^y$ be the solution of \eqref{3.11}. The function $\psi^y_R:=\phi^y\chi^y_R$ belongs to $E_R\!\setminus\!\{0\}$ and 
\begin{align*}
\frac{\langle \mathcal{L}^y_R\psi^y_R,\psi^y_R \rangle_R}{\langle \psi^y_R,\psi^y_R \rangle_R}=\frac{\langle \mathcal{L}^y_R\phi^y,\phi^y \rangle_{R-1} + D_R}{\langle \psi^y_R,\psi^y_R \rangle_R}=\frac{\lambda_1\langle\phi^y,\phi^y \rangle_{R-1} + D_R}{\langle \psi^y_R,\psi^y_R \rangle_R}
\end{align*}
from Lemma~\ref{lem2}, where
\begin{align*}
D_R&=\langle\mathcal{L}^y_R\psi^y_R,\psi^y_R \rangle_R - \langle \mathcal{L}^y_R\phi^y,\phi^y \rangle_{R-1}\\
&=\sum_{s}\int_{\widetilde J_s}-d(\cdot+y)(\phi^y\chi^y_R)''(\phi^y\chi^y_R)-f_s(\cdot+y,0)(\phi^y\chi^y_R)^2\\
&~~+\sum\limits_r\frac{1}{k}\int_{\widetilde K_r} -d(\cdot+y)(\phi^y\chi^y_R)''(\phi^y\chi^y_R)-f_s(\cdot+y,0)(\phi^y\chi^y_R)^2,
\end{align*}
and $\widetilde J_s$ and $\widetilde K_r$ stand for the patches of type 1 and 2 in $\big((-R,-R+1)\cup(R-1,R)\big)\setminus(S-y)$, that is, $\big((-R,-R+1)\cup(R-1,R)\big)\setminus(S-y)=(\cup_s \widetilde J_s)\cup(\cup_r \widetilde K_r)$. Since $\phi^y$ is periodic and~$\phi^y|_{\overline{\mathcal{I}}}\in C^2(\overline{\mathcal{I}})$ (and even $C^\infty(\overline{\mathcal{I}})$) for each patch $\mathcal{I}\subset\R\setminus(S-y)$, the quantities $\|\phi^y|_{\overline{\mathcal{I}}}\|_{C^2(\overline{\mathcal{I}})}$ are bounded independently of $\mathcal{I}$. Since $\sup_{R\ge 2}\|\chi^y_R\|_{C^2(\R)}<+\infty$, it follows that there exists $C>0$ such that~$|D_R|\le C$ for all $R\ge2$. Likewise, one also has
$$\sup_{R\ge2}|\langle\psi^y_R,\psi^y_R \rangle_R - \langle \phi^y,\phi^y \rangle_{R-1}|<+\infty.$$
On the other hand, since $\phi^y$ is continuous, positive and periodic, there exists $\delta>0$ such that $\phi^y\ge\delta>0$ in $\mathbb{R}$, hence $\langle \phi^y,\phi^y\rangle_{R-1}\ge2\min(1,1/k)\delta^2(R-1)$ and
\begin{align*}
\frac{\langle\psi^y_R,\psi^y_R \rangle_R }{\langle \phi^y,\phi^y \rangle_{R-1}}\to 1~~\text{as}~R\to+\infty.
\end{align*}
Using the estimates above, one gets that 
\begin{align*}
\frac{\langle \mathcal{L}^y_R\psi^y_R,\psi^y_R \rangle_R}{\langle \psi^y_R,\psi^y_R \rangle_R}\to \lambda_1~~\text{as}~R\to+\infty,
\end{align*}
and, together with \eqref{3.15} and Lemma \ref{lemma-3.4}, it follows that $\lambda^y_R\to \lambda_1$ as $R\to+\infty$. As already emphasized, this provides the desired conclusion.
\end{proof}

Now we are in a position to give the proof of Proposition \ref{prop-3.2}.

\begin{proof}[Proof of Proposition~$\ref{prop-3.2}$] Assume~\eqref{2.5} and $\lambda_1<0$. Let $p$ be a continuous nonnegative bounded solution of the stationary problem \eqref{m-pb-elliptic}, with $p|_{\bar I}\in C^2(\bar I)$ for each patch $I$ in $\mathbb{R}$. Assume that~$p\not \equiv 0$. By an immediate induction, the strong maximum principle and the Hopf lemma then imply that $p>0$ in $\mathbb{R}$. Now, from Lemma~\ref{lemma-3.6}, there is $R>0$ such that
\begin{align}
\label{3.25}
\forall\,y\in\mathbb{R},~~~\lambda^y_R<\frac{\lambda_1}{2}<0,
\end{align}
where $(\lambda_R^y,\varphi^y_R)$ denotes the eigenpair of \eqref{3.10}. From~\eqref{2.5}, one can choose $\kappa_0>0$ small enough such that, for all $0<\kappa\le \kappa_0$ and $y\in\R$,
\begin{align}
\label{3.24}
f(x+y,\kappa\varphi^y_R(x))\ge f_s(x+y,0)\kappa\varphi^y_R(x)+\frac{\lambda_1}{2}\kappa\varphi^y_R(x)~~\text{for all}~x\in[-R,R]\backslash(S-y).
\end{align}
For each $y\in\R$, the function $p^y:=p(\cdot+y)$ satisfies
\begin{equation}
\label{3.26}
\begin{aligned}
\begin{cases}
-d(x+y)(p^y)''(x)=f(x+y,p^y(x)),\qquad~~\,x\in\mathbb{R}\backslash(S-y),\\
p^y(x^-)=p^y(x^+),~(p^y)'(x^-)=\sigma (p^y)'(x^+),~~x=nl-y,\\
p^y(x^-)=p^y(x^+),~\sigma (p^y)'(x^-)=(p^y)'(x^+),~~x=nl+l_2-y,
\end{cases}
\end{aligned}
\end{equation}
while, for each $\kappa\in(0,\kappa_0]$, the continuous function $\kappa\varphi^y_R$ satisfies
\begin{equation}\label{kappaphi}
\begin{aligned}
\begin{cases}
-d(x+y)\kappa(\varphi^y_R)''(x)-f_s(x+y,0)\kappa\varphi^y_R(x)=\lambda^y_R\kappa\varphi^y_R(x),&\text{in}~(-R,R)\backslash(S-y),\\
\kappa\varphi^y_R(x^-)=\kappa\varphi^y_R(x^+),~~\kappa(\varphi^y_R)'(x^-)=\sigma\kappa(\varphi^y_R)'(x^+),&x=nl-y\in(-R,R),\\
\kappa\varphi^y_R(x^-)=\kappa\varphi^y_R(x^+),\,\sigma\kappa(\varphi^y_R)'(x^-)=\kappa(\varphi^y_R)'(x^+),&x=nl+l_2-y\in(-R,R),\\
\kappa\varphi^y_R>0~\text{in}~(-R,R),~~\kappa\varphi^y_R(\pm R)=0.
\end{cases}
\end{aligned}
\end{equation}
It then follows from \eqref{3.25}--\eqref{3.24} that, for each $\kappa\in(0,\kappa_0]$ and $y\in\R$,
\begin{align}
\label{3.28}
-d(x\!+\!y)\kappa(\varphi^y_R)''(x)\!-\!f(x\!+\!y,\kappa\varphi^y_R(x))\!\le\!\Big(\lambda^y_R\!-\!\frac{\lambda_1}{2}\Big)\kappa\varphi^y_R(x)\!<\!0~\text{for}~x\!\in\!(-R,R)\backslash(S\!-\!y).
\end{align}
 	
Let us finally consider any $y\in\R$ and prove that $p^y\ge\kappa_0\varphi^y_R$ in $[-R,R]$. Assuming not and using the continuity and positivity of $p^y$ and the continuity of $\varphi^y_R$, one can then define
\begin{align*}
\kappa^*=\sup\{\kappa\in (0, \kappa_0] : p^y\ge\kappa\varphi^y_R~\text{in}~[-R,R]\}\in(0,\kappa_0)
\end{align*}
and one has $p^y\ge\kappa^*\varphi^y_R$ in $[-R,R]$ with equality at a point $x_0\in[-R,R]$. Since $p^y>0$ in $\mathbb{R}$ and $\varphi^y_R(\pm R)=0$, there holds $x_0\in(-R,R)$. From~\eqref{3.26}--\eqref{3.28} and finitely many applications of the strong maximum principle and the Hopf lemma, one gets that $p^y\equiv\kappa^*\varphi^y_R$ in $(-R,R)$ and then in $[-R,R]$ by continuity, a contradiction with the boundary conditions at $x=\pm R$. As a consequence, $p^y\ge\kappa_0\varphi^y_R$ in $[-R,R]$.
 	
Thus, $p(y)=p^y(0)\ge\kappa_0\varphi^y_R(0)$ for all $y\in\mathbb{R}$. Since the function $y\mapsto \kappa_0\varphi^y_R(0)$ is periodic, continuous and positive, one concludes that $\inf_{\mathbb{R}}p>0$.
\end{proof}

\begin{proof}[Proof  of Theorem~$\ref{thm-2.4-uniqueness}$]
Assume that $f$ satisfies~\eqref{2.5}--\eqref{2.4} and that $\lambda_1<0$. Let $q$ and $p$ be two positive  bounded solutions of~\eqref{m-pb-elliptic} (in the sense that $q$ and $p$ are continuous in $\R$ and have restrictions in $\bar I$ of class $C^2(\bar I)$ for each patch $I\subset\R\backslash S$). Applying Proposition \ref{prop-3.2}, there exists $\varep>0$ such that $q\ge \varep$ and $p\ge \varep$ in $\mathbb{R}$. One can then define the positive real number
\begin{align*}
\gamma^*=\sup\{\gamma>0: q>\gamma p~\text{in}~\mathbb{R}\}\in(0,+\infty).
\end{align*} 
We shall prove that $\gamma^*\ge1$, which will easily yield the conclusion by interchanging the roles of~$p$ and~$q$. Assume by way of contradiction that $\gamma^*<1$, and set $z:=q-\gamma^* p\ge0$. From the definition of $\gamma^*$, there exists a sequence $(x_m)_{m\in\mathbb{N}}$ such that $z(x_m)\to 0$ as $m\to+\infty$.  Moreover, the nonnegative function $z$ is continuous in $\R$, has restrictions in $\bar I$ of class $C^2(\bar I)$ for each patch~$I\subset\R\backslash S$, and it satisfies 
\begin{equation}
\label{3.30}
\begin{aligned}
\begin{cases}
-d(x)z''(x)-f(x,q(x))+\gamma^*f(x,p(x))=0,~~&x\in\mathbb{R}\backslash S,\\
z(x^-)=z(x^+),~z'(x^-)=\sigma z'(x^+),~~&x=nl,\\
z(x^-)=z(x^+),~\sigma z'(x^-)=z'(x^+),~~&x=nl+l_2.
\end{cases}
\end{aligned}
\end{equation}
As in the proof of Theorem~\ref{thm-2.1-existence}, let us assume in~\eqref{2.4} that $s\mapsto f_1(s)/s$ is decreasing with respect to $s>0$ (the case when $s\mapsto f_2(s)/s$ is decreasing with respect to $s>0$ can be handled similarly). Since $\gamma^*\in(0,1)$ and $p$ is positive in $\R$, one then has $\gamma^*f(x,p(x))=\gamma^*f_1(p(x))<f_1(\gamma^*p(x))=f(x,\gamma^*p(x))$ for $x\in(nl-l_1,nl)$ and $n\in\Z$, while $\gamma^*f(x,p(x))=\gamma^*f_2(p(x))\le f_2(\gamma^*p(x))=f(x,\gamma^*p(x))$ for $x\in(nl,nl+l_2)$ and $n\in\Z$. Hence, \eqref{3.30} implies that
\begin{equation*}
\left\{
\begin{aligned}
-d_1z''(x)-f_1(q(x))+f_1(\gamma^*p(x))>0,~~x\in(nl-l_1,nl),\\
-d_2z''(x)-f_2(q(x))+f_2(\gamma^*p(x))\ge 0,~~x\in(nl,nl+l_2).
\end{aligned}\right.
\end{equation*}
Therefore, $z$ satisfies a problem of the type~\eqref{3.5} for some bounded function $b$ defined in~$\R\!\setminus\!S$. Hence, as in the proof of Theorem~\ref{thm-2.1-existence}, if $x_m\to\bar x\in\R$ up to extraction of a subsequence, one gets a contradiction by using the strong maximum principle and the Hopf lemma. In the general case,  let $\overline x_m\in(-l_1,l_2]$ be such that $x_m-\overline x_m\in l\mathbb{Z}$, and let $\overline x_\infty\in[-l_1,l_2]$ be such that~$\overline x_m\to\overline x_\infty$ as $m\to+\infty$, up to extraction of some subsequence. Next, set $z_m=q_m-\gamma^* p_m$, where $q_m(x):=q(x+x_m-\overline x_m)$ and $p_m(x):=p(x+x_m-\overline x_m)$. Since $d(x)$ and $f(x,u)$ are periodic with respect to $x$, it follows from \eqref{3.30} that the functions $z_m$'s satisfy
\begin{equation*}
\begin{aligned}
\begin{cases}
-d(x)z_m''(x)-f(x,q_m(x))+\gamma^*f(x,p_m(x))=0~~&x\in\R\!\setminus\!S,\\
z_m(x^-)=z_m(x^+),~z_m'(x^-)=\sigma z_m'(x^+),~~&x=nl,\\
z_m(x^-)=z_m(x^+),~\sigma z_m'(x^-)=z_m'(x^+),~~&x=nl+l_2.
\end{cases}
\end{aligned}
\end{equation*}	
The sequences of continuous functions $(q_m)_{m\in\N}$ and $(p_m)_{m\in\N}$ are bounded in $L^\infty(\R)$, and then in $C^{2,\nu}(\bar I)$ for each $\nu\in(0,1)$ and each patch $I\subset\R\backslash S$ from standard elliptic estimates. Thus, there exist three continuous functions $q_\infty\ge\varep$, $p_\infty\ge\varep$ and $z_\infty=q_\infty-\gamma^*p_\infty\ge0$ such that, up to extraction of some subsequence, $(q_m|_{\bar I},p_m|_{\bar I},z_m|_{\bar I})\to(q_\infty|_{\bar I},p_\infty|_{\bar I},z_\infty|_{\bar I})$ in $C^2(\bar I)$ as $m\to+\infty$ for each patch $I\subset\R\backslash S$. Furthermore,
\begin{equation*}
\begin{aligned}
\begin{cases}
-d(x)z_\infty''(x)-f(x,q_\infty(x))+\gamma^*f(x,p_\infty(x))=0,~~&x\in\R\!\setminus\!S,\\
z_\infty(x^-)=z_\infty(x^+),~z_\infty'(x^-)=\sigma z_\infty'(x^+),~~&x=nl,\\
z_\infty(x^-)=z_\infty(x^+),~\sigma z_\infty'(x^-)=z_\infty'(x^+),~~&x=nl+l_2,
\end{cases}
\end{aligned}
\end{equation*}	
and $z_\infty\ge 0$ in $\R$ with $z_\infty(\bar x_\infty)=0$.
Using the positivity of $p_\infty$ and the same argument as for problem \eqref{3.30} above, one reaches a contradiction. 
	 
Consequently, $\gamma^*\ge 1$, whence $q\ge p$ in $\R$. By interchanging the roles of $q$ and $p$, one also gets $p\ge q$ in~$\R$. The uniqueness is therefore obtained. Furthermore, if $p$ is a positive solution of~\eqref{m-pb-elliptic}, so is the function $x\mapsto p(x+l)$. This implies that $p$ is periodic. The proof of Theorem~\ref{thm-2.4-uniqueness} is thereby complete.
\end{proof}


\section{Long-time behavior: proof of Theorem \ref{thm-long time behavior}}
\label{Sec-thm-long time behavior}

This section is devoted to the proof of Theorem~\ref{thm-long time behavior} on the large-time behavior of the solutions of the evolution problem~\eqref{m-pb}--\eqref{d-f}.

\begin{proof}[Proof of Theorem~$\ref{thm-long time behavior}$]
Assume that $f$ satisfies \eqref{2.5}--\eqref{2.4}. Let $u$ be the unique solution, given in Theorem~\ref{thm-m-wellposedness}, of the Cauchy problem \eqref{m-pb}--\eqref{d-f} with a nonnegative, bounded and continuous  initial datum $u_0\not\equiv 0$. We know from Theorem~\ref{thm-m-wellposedness} that $u$ is continuous in $[0,+\infty)\times\R$ and positive in $(0,+\infty)\times\R$.
	
(i) Assume that $0$ is an unstable solution of~\eqref{m-pb-elliptic}, that is, $\lambda_1<0$, and let $p$ be the unique positive bounded and periodic solution of~\eqref{m-pb-elliptic} given by Theorem~\ref{thm-2.1-existence}~{\rm{(i)}} and Theorem~\ref{thm-2.4-uniqueness}. The function $p$ is continuous in $\R$ and has restriction in $\bar I$ of class $C^2(\bar I)$ for each patch $I\subset\R$. With the notations of Section~\ref{sec42} and from Lemma~\ref{lemma-3.6}, one can fix $R>0$ large enough so that~$\lambda^0_R<\lambda_1/2<0$ and~\eqref{kappaphi}--\eqref{3.28} hold with $y=0$ for all $\kappa>0$ small enough, where~$(\lambda^0_R,\varphi^0_R)$ denotes the unique eigenpair solving~\eqref{3.10} with $y=0$:
\begin{equation*}
\begin{aligned}
\begin{cases}
-d(x)(\varphi^0_R)''(x)-f_s(x,0)\varphi^0_R(x)=\lambda^0_R\varphi^0_R(x),~~~~&x\in(-R,R)\backslash S,\\
\varphi^0_R(x^-)=\varphi^0_R(x^+),~~~(\varphi^0_R)'(x^-)=\sigma(\varphi^0_R)'(x^+),~~&x=nl\in(-R,R),\\
\varphi^0_R(x^-)=\varphi^0_R(x^+),~\sigma(\varphi^0_R)'(x^-)=(\varphi^0_R)'(x^+),~~~~&x=nl+l_2\in(-R,R),\\
\varphi^0_R>0~\text{in}~(-R,R),~\varphi^0_R(\pm R)=0,~\Vert\varphi^0_R\Vert_{L^\infty(-R,R)}=1.
\end{cases}
\end{aligned}
\end{equation*}
From the continuity and positivity of $p$ and $u(1,\cdot)$ in $\R$, one can fix $\kappa>0$ small enough so that~\eqref{kappaphi}--\eqref{3.28} hold with $y=0$, together with $\kappa\varphi^0_R<p$ and $\kappa\varphi^0_R<u(1,\cdot)$ in $[-R,R]$.

Define now a function $v_0$ in $\mathbb{R}$ by
$$v_0(x)=\left\{\baa{ll}
\kappa\varphi^0_R(x) & \hbox{for $x\in[-R,R]$},\vspace{3pt}\\
0 & \hbox{for $x\in\R\!\setminus\![-R,R]$}.\eaa\right.$$
The function $v_0$ is nonnegative, continuous and bounded in $\R$, with $v_0\not\equiv0$ in~$\R$. Let  $v$ be the solution of the Cauchy problem~\eqref{m-pb}--\eqref{d-f} with initial datum $v_0$, given by Theorem~\ref{thm-m-wellposedness}. For each $n\in\N$, let $v^n$ (respectively $u^n$) be the unique bounded classical solution of~\eqref{tp-1}--\eqref{tp-4} with initial condition $v^n(0,\cdot)=\delta^nv_0|_{[-nl,nl]}$ (respectively $u^n(0,\cdot)=\delta^nu_0|_{[-nl,nl]}$), with the cut-off function~$\delta^n$ given in~\eqref{cut-off func}. From Theorem~\ref{thm-m-wellposedness}, the sequence $(v^n)_{n\in\N}$ (respectively~$(u^n)_{n\in\N})$ converges monotonically pointwise in $[0,+\infty)\times\R$ to the function $v$ (respectively $u$). For each $n\in\N$, there holds $0\le v^n(0,\cdot)\le v_0\le p$ in $[-nl,nl]$, hence $0\le v^n(t,x)\le p(x)$ for all $(t,x)\in[0,+\infty)\times[-nl,nl]$ by Proposition~\ref{bdd-cp}, and
$$0\le v(t,x)\le p(x)\ \hbox{ for all $(t,x)\in[0,+\infty)\times\R$}$$
by passing to the limit $n\to+\infty$ (actually, we also know from Theorem~\ref{thm-m-wellposedness} that $v>0$ in $(0,+\infty)\times\R$). Furthermore, for each $n\ge R/l+1$, since $v^n(t,\pm R)\ge0$ for all $t\ge0$ and since
$$v^n(0,\cdot)|_{[-R,R]}=\delta^nv_0|_{[-R,R]}=v_0|_{[-R,R]}=\kappa\varphi^0_R$$
with $\kappa\varphi^0_R$ satisfying~\eqref{kappaphi}--\eqref{3.28}, Proposition~\ref{bdd-cp} again implies that
$$v^n(t,x)\ge\kappa\varphi^0_R(x)\hbox{ for all~$(t,x)\in[0,+\infty)\times[-R,R]$}.$$
Together with the nonnegativity of $v^n$ in $[0,+\infty)\times[-nl,nl]$ and the definition of $v_0$, this yields
$$v^n(h,x)\ge v_0(x)\ge v^n(0,x)\hbox{ for all~$x\in[-nl,nl]$ and $h\ge0$},$$
hence $v^n(t+h,x)\ge v^n(t,x)$ for all~$(t,x)\in[0,+\infty)\times[-nl,nl]$ and $h\ge0$, and finally $v(t+h,x)\ge v(t,x)$ for all $(t,x)\in[0,+\infty)\times\R$ and $h\ge0$. Therefore, the function $v$ is non-decreasing in~$t$. Since it is positive in $(0,+\infty)$ and since $v(t,x)\le p(x)$ for all $(t,x)\in[0,+\infty)\times\R$, the Schauder estimates of Theorem~\ref{thm-m-wellposedness} yield the existence of a positive bounded solution $q$ of~\eqref{m-pb-elliptic} such that~$v(t,\cdot)\to q$ as $t\to+\infty$ locally uniformly in $\R$ and $v(t,\cdot)|_{\bar I}\to q|_{\bar I}$ as $t\to+\infty$ in $C^2(\bar I)$ for each patch $I\subset\R$. It follows from Theorem~\ref{thm-2.4-uniqueness} that $q\equiv p$ in $\R$. In other words,
\be\label{limitv}
v(t,\cdot)\mathop{\longrightarrow}_{t\to+\infty}p\hbox{ locally uniformly in $\R$ and }v(t,\cdot)|_{\bar I}\mathop{\longrightarrow}_{t\to+\infty}p|_{\bar I}\hbox{ in $C^2(\bar I)$ for each patch $I\!\subset\!\R$}.
\ee

On the other hand, since the continuous functions $\kappa\varphi^0_R$ and $u(1,\cdot)$ satisfy $\kappa\varphi^0_R<u(1,\cdot)$ in~$[-R,R]$ and since $u^n(1,\cdot)\to u(1,\cdot)$ locally uniformly in $\R$ as $n\to+\infty$, there is $n_0\ge R/l$ such that $\kappa\varphi^0_R(x)\le u^n(1,x)$ for all $x\in[-R,R]$ and for all $n\ge n_0$. Since $u^n\ge0$ in~$[0,+\infty)\times[-nl,nl]$, it follows that $v^n(0,x)\le v_0(x)\le u^n(1,x)$ for all $x\in[-nl,nl]$ and for all~$n\ge n_0$, hence $v^n(t,x)\le u^n(t+1,x)$ for all $(t,x)\in[0,+\infty)\times[-nl,nl]$ and for all $n\ge n_0$, by Proposition~\ref{bdd-cp}. Therefore,
\be\label{ineqvu}
v(t,x)\le u(t+1,x)\ \hbox{ for all $(t,x)\in[0,+\infty)\times\R$}.
\ee

Lastly, define $M_1:=\max(M,\|u_0\|_{L^\infty(\R)})$ with $M>0$ as in~\eqref{2.5}. As in the proof of Theorem~\ref{thm-2.1-existence}~(i), the solution $w$ (given by Theorem~\ref{thm-m-wellposedness}) of the Cauchy problem \eqref{m-pb}--\eqref{d-f} with initial datum $M_1$, is non-increasing in $t$, periodic in $x$, and converges as $t\to+\infty$ uniformly in~$\R$ to a nonnegative periodic bounded solution $\bar p$ of~\eqref{m-pb-elliptic}. Furthermore,
\be\label{inequw}
0\le u(t,x)\le w(t,x)\ \hbox{ for all $(t,x)\in[0,+\infty)\times\R$}
\ee
by Theorem~\ref{thm-m-wellposedness}. Together with~\eqref{limitv}--\eqref{ineqvu}, one infers that $\bar p\ge p\ (>0)$ in $\R$, and then $\bar p\equiv p$ by Theorem~\ref{thm-2.4-uniqueness}. Since $v(t,x)\le u(t+1,x)\le w(t+1,x)$ for all $(t,x)\in[0,+\infty)\times\R$, one then concludes that $u(t,\cdot)\to p$ as $t\to+\infty$ locally uniformly in $\R$, and together with the Schauder estimates of Theorem~\ref{thm-m-wellposedness}, that $u(t,\cdot)|_{\bar I}\to p|_{\bar I}$ as $t\to+\infty$ in $C^2(\bar I)$ for each patch $I\subset\R$. 

(ii) Let us now assume that $0$ is a stable solution of~\eqref{m-pb-elliptic}, that is, $\lambda_1\ge0$. By defining $w$ and $\bar p$ as in the previous paragraph (the definitions of $w$ and $\bar p$ did not use the stability properties of~$0$), Theorem~\ref{thm-2.1-existence}~{\rm{(ii)}} then yields $\bar p\equiv 0$ in $\R$. Together with~\eqref{inequw} and the uniform convergence $w(t,\cdot)\to\bar p=0$ in $\R$ as $t\to+\infty$, one concludes that $u(t,\cdot)\to0$ as $t\to+\infty$ uniformly in~$\R$. The proof of Theorem~\ref{thm-long time behavior} is thereby complete.
\end{proof}

An immediate corollary of Theorem~\ref{thm-long time behavior}, which will be used in the proofs of Theorems~\ref{thm-spreading result} and~\ref{thm-existence of PTW} in Section~\ref{Sec-thm-spreading properities}, is the following result.

\begin{corollary}
\label{prop-long time-periodic}
Assume that  $f$ satisfies \eqref{2.5}--\eqref{2.4} and that $0$ is an unstable solution of~\eqref{m-pb-elliptic} $($i.e., $\lambda_1<0$$)$. Let $p$ be the unique positive bounded and periodic solution of \eqref{m-pb-elliptic} given by Theorem~$\ref{thm-2.1-existence}$~{\rm{(i)}} and Theorem~$\ref{thm-2.4-uniqueness}$. Let $u$ denote the solution, given by Theorem~$\ref{thm-m-wellposedness}$, of the Cauchy problem~\eqref{m-pb}--\eqref{d-f} with a nonnegative bounded and continuous initial datum $u_0\not\equiv0$. If $u_0$ is periodic, then~$u(t,\cdot)\to p$ as $t\to+\infty$ uniformly in $\R$.
\end{corollary}

\begin{proof}
We already know from Theorem~\ref{thm-m-wellposedness} that, for each $t\ge0$, the function $x\mapsto u(t,x)$ is periodic. Since $u(t,\cdot)\to p$ as $t\to+\infty$ locally uniformly in $\R$ by Theorem~\ref{thm-long time behavior}, the conclusion follows.
\end{proof}


\section{Spreading speeds and periodic traveling waves: proofs of Theo\-rems \ref{thm-spreading result} and  \ref{thm-existence of PTW}}\label{Sec-thm-spreading properities}

This section is devoted to the study of the spatial dynamics of the problem~\eqref{m-pb}--\eqref{d-f}. We will prove the existence of an asymptotic spreading speed $c^*$, which  can be given explicitly by a variational formula using principal eigenvalues of certain linear operators. Moreover, the spreading speed coincides with the minimal  speed for pulsating traveling waves. The main approach is based on the abstract dynamical systems  theory for monostable evolution systems established in the seminal work in \cite{W2002} and further developed in  \cite{LZ2007, LZ2010}.

Hereafter we assume that the $0$ solution of \eqref{m-pb-elliptic} is unstable (i.e., $\lambda_1<0$) and that  $f$ satisfies  \eqref{2.5}--\eqref{2.4}. By Theorem \ref{thm-2.1-existence} (i) and Theorem \ref{thm-2.4-uniqueness}, there exists a unique positive bounded periodic solution $p$ of~\eqref{m-pb-elliptic}.  We point out that, with these hypotheses, populations starting with any bounded nonnegative and non-trivial initial condition always persist, by Theorem~\ref{thm-long time behavior}. Using the notations of~\cite[Section 5]{LZ2010}, we define $\mathcal{H}=\mathbb{R}$, $\widetilde{\mathcal{H}}=l\mathbb{Z}$, $X=Y=\mathbb{R}$, $\beta=p$ and $\mathcal{M}=\mathcal{C}_p$, with $\mathcal{C}_p$ given in~\eqref{defCp}. We also define a family of maps $\{Q_t\}_{t\ge 0}$ in $\mathcal{C}_p$ by 
\begin{align}\label{defQt}
Q_t(\omega)(x)=u(t,x;\omega)~~\text{for $\omega\in \mathcal{C}_p$, $t\ge0$, and $x\in\mathbb{R}$},
\end{align}
where $(t,x)\mapsto u(t,x;\omega)$ denotes the unique classical solution to the Cauchy problem \eqref{m-pb}--\eqref{d-f} with initial condition $u(0,\cdot;\omega)=\omega\in\mathcal{C}_p$, given by Theorem~\ref{thm-m-wellposedness}. In particular, $Q_0(\omega)=\omega$ for every $\omega\in\mathcal{C}_p$, and $Q_t(0)=0$ for every $t\ge0$. Furthermore, since the continuous positive function~$p$ solves~\eqref{m-pb-elliptic}, it follows from the uniqueness in Theorem~\ref{thm-m-wellposedness} that $u(t,\cdot;p)=p$ for each $t\ge0$, that is, $Q_t(p)=p$. The monotonicity in Theorem~\ref{thm-m-wellposedness} then implies that $u(t,\cdot;\omega)\in\mathcal{C}_p$ for every $\omega\in\mathcal{C}_p$ and $t\ge0$. In other words, for every $t\ge0$, $Q_t$ maps $\mathcal{C}_p$ into itself.

We recall that a family of maps $\{Q_t\}_{t\ge 0}$ from $\mathcal{C}_p$ into itself is said to be a semiflow in $\mathcal{C}_p$ if it satisfies the following properties:
\begin{enumerate}[(1)]
\item $Q_0(\omega)=\omega$ for all $\omega\in\mathcal{C}_p$;
\item $Q_{t_1}(Q_{t_2}(\omega))=Q_{t_1+t_2}(\omega)$ for all $t_1,t_2\ge 0$ and for all $\omega\in\mathcal{C}_p$;
\item the map $(t,\omega)\mapsto Q_t(\omega)$ is continuous from $[0,+\infty)\times \mathcal{C}_p$ into $\mathcal{C}_p$, with $\mathcal{C}_p$ equipped with the compact open topology, that is, $Q_{t_m}(\omega_m)\to Q_t(\omega)$ as $m\to+\infty$ locally uniformly in $\R$ if $t_m\to t$ and $\omega_m\to\omega$ locally uniformly in $\R$ as $m\to+\infty$, with $(t_m,\omega_m)\in[0,+\infty)\times\mathcal{C}_p$.
\end{enumerate}
We also say that $\{Q_t\}_{t\ge 0}$ is monotone in $\mathcal{C}_p$ if, for every $t\ge0$, $Q_t(\omega)\ge Q_t(\omega')$ in $\R$ provided $\omega\ge\omega'$ in $\R$, with $\omega,\omega'\in\mathcal{C}_p$. Lastly, $\{
Q_t\}_{t\ge0}$ is called subhomogeneous if $\gamma Q_t(\omega)\le Q_t(\gamma\omega)$ in~$\R$ for every $t\ge0$, $\gamma\in[0,1]$, and $\omega\in\mathcal{C}_p$.

The following proposition summarizes the properties of the family $\{Q_t\}_{t\ge 0}$ defined in~\eqref{defQt}.

\begin{proposition}
\label{prop-semiflow}
The family $\{Q_t\}_{t\ge 0}$ defined in~\eqref{defQt} is a monotone and subhomogeneous semiflow in $\mathcal{C}_p$. Furthermore, for every $\omega\in\mathcal{C}_p$, $a\in l\Z$, $t\ge0$ and $x\in\R$, there holds $Q_t(\omega(\cdot+a))(x)=Q_t(\omega)(x+a)$.
\end{proposition}

\begin{proof}
First of all, the property $Q_0(\omega)=\omega$ is already known by definition, and the monotonicity of $\{Q_t\}_{t\ge 0}$ follows from Theorem~\ref{thm-m-wellposedness}. Secondly, for $t_1,t_2\ge0$ and $\omega\in\mathcal{C}_p$, the function $(t,x)\mapsto u(t+t_2,x;\omega)$ is a nonnegative bounded classical solution of~\eqref{m-pb}--\eqref{d-f} in $[0,+\infty)\times\R$, with initial condition $u(t_2,\cdot;\omega)=Q_{t_2}(\omega)$.  Owing to the uniqueness in Theorem~\ref{thm-m-wellposedness}, one infers immediately that $u(t_1+t_2,\cdot;\omega)=Q_{t_1}(Q_{t_2}(\omega))$ in $\R$, that is, $Q_{t_1+t_2}(\omega)=Q_{t_1}(Q_{t_2}(\omega))$ in $\R$.

To show the continuity property, consider any $(t,\omega)\in[0,+\infty)\times\mathcal{C}_p$ and any sequence $(t_m,\omega_m)_{m\in\N}$ in $[0,+\infty)\times\mathcal{C}_p$ such that $t_m\to t$ and $\omega_m\to\omega$ locally uniformly in $\R$ as $m\to+\infty$. One has to show that $u(t_m,\cdot;\omega_m)\to u(t,\cdot;\omega)$ locally uniformly in $\R$ as $m\to+\infty$. Let $T\in(0,+\infty)$ be such that $0\le t\le T$ and $0\le t_m\le T$ for all $m\in\N$. Take any $A>0$, and let $\epsilon>0$ be arbitrary. There is a function $\bar\omega\in\mathcal{C}_p\cap C^3(\R)$ such that $\|\bar\omega\|_{C^3(\R)}<+\infty$, $\bar\omega'=0$ at all points of~$S$, and $\|\bar\omega-\omega\|_{L^\infty(\R)}\le\varep/2$. For $m\in\N$, define $\bar\omega_m:\R\to\R$ by
$$\bar\omega_m=\min\big(\max(\bar\omega,\omega_m-\varep),\omega_m+\varep\big),$$
that is, $\bar\omega_m(x)=\bar\omega(x)$ if $\omega_m(x)-\varep\le\bar\omega(x)\le\omega_m(x)+\varep$, $\bar\omega_m(x)=\omega_m(x)+\varep$ if $\bar\omega(x)>\omega_m(x)+\varep$, and~$\bar\omega_m(x)=\omega_m(x)-\varep$ if $\bar\omega(x)<\omega_m(x)-\varep$. Each function $\bar\omega_m$ belongs to~$\mathcal{C}_p$, and~$\|\bar\omega_m-\omega_m\|_{L^\infty(\R)}\le\varep$. Furthermore, since $\|\bar\omega-\omega\|_{L^\infty(\R)}\le\varep/2$ and $\omega_m\to\omega$ as $m\to+\infty$ locally uniformly in $\R$, one infers that, for each compact set $\mathcal{K}\subset\R$, one has
\be\label{omegam}
\bar\omega_m|_{\mathcal{K}}=\bar\omega|_{\mathcal{K}}\ \hbox{ for all $m$ large enough}.
\ee
As $\|\bar\omega\|_{C^3(\R)}$ is finite and $0\le u(t,x;\omega_m)\le p(x)$ for all $(t,x)\in[0,+\infty)\times\R$ and $m\in\N$, standard parabolic estimates yield the existence of a positive constant $C_1$ such that, for each patch $I=(a,b)\subset\R$, the functions $t\mapsto u(t,(a+b)/2;\bar\omega_m)$ belong to $C^{1,1/4}([0,+\infty))$ for all~$m$ large enough, and~$\|u(\cdot,(a+b)/2;\bar\omega_m)\|_{C^{1,1/4}([0,+\infty))}\le C_1$. As in the proof of Theorem~\ref{thm-m-wellposedness}, and using here the fact that $\bar\omega$ is bounded in~$C^3(\R)$ and $\bar\omega'$ vanishes on~$S$, it then follows that there are $\theta\in(0,1)$ and a positive constant $C_2$ such that, for each patch~$I\subset\R$ and from the above estimates applied at the middle points of the leftward and rightward adjacent patches, the functions $u(\cdot,\cdot;\bar\omega_m)|_{[0,+\infty)\times\bar I}$ belong to $C^{1,\theta;2,\theta}_{t;x}([0,+\infty)\times\bar I)$ and
$$\|u(\cdot,\cdot;\bar\omega_m)|_{[0,+\infty)\times\bar I}\|_{C^{1,\theta;2,\theta}_{t;x}([0,+\infty)\times\bar I)}\le C_2$$
for all $m$ large enough. Up to extraction of a subsequence, the continuous functions $u(\cdot,\cdot;\bar\omega_m)$ converge as $m\to+\infty$ locally uniformly in $[0,+\infty)\times\R$ to a nonnegative bounded classical solution~$\bar U$ of~\eqref{m-pb}--\eqref{d-f}. Notice also that $\bar U(0,\cdot)=\bar\omega$ from the above limits and~\eqref{omegam}. The uniqueness in Theorem~\ref{thm-m-wellposedness} then implies that $\bar U\equiv u(\cdot,\cdot;\bar\omega)$ in $[0,+\infty)\times\R$. Therefore, since the limit of any subsequence of  $(u(\cdot,\cdot;\bar\omega_m))_{m\in\N}$ is unique, one gets that the whole sequence~$(u(\cdot,\cdot;\bar\omega_m))_{m\in\N}$ converges locally uniformly in $[0,+\infty)\times\R$ to $u(\cdot,\cdot;\bar\omega)$. In particular, there is~$m_0\in\N$ such that
$$\max_{[0,T]\times[-A,A]}|u(\cdot,\cdot;\bar\omega_m)-u(\cdot,\cdot;\bar\omega)|\le\varep\ \hbox{ for all $m\ge m_0$}.$$
Finally, since $\|\bar\omega\!-\!\omega\|_{L^\infty(\R)}\!\le\!\varep/2$ and $\|\bar\omega_m\!-\!\omega_m\|_{L^\infty(\R)}\!\le\!\varep$ for each $m$, formula~\eqref{Lipuv} of the proof of Theorem~\ref{thm-m-wellposedness} yields $\|u(\cdot,\cdot;\bar\omega)\!-\!u(\cdot,\cdot;\omega)\|_{L^\infty([0,T]\times\R)}\!\le\!\varep e^{LT}$ and $\|u(\cdot,\cdot;\bar\omega_m)\!-\!u(\cdot,\cdot;\omega_m)\|_{L^\infty([0,T]\times\R)}\!\le\!2\varep e^{LT}$, with $L:=\max(\|f'_1\|_{L^\infty([0,\bar K])},\|f'_2\|_{L^\infty([0,\bar K])})$ and $\bar K:=\max(K_1,K_2,\|p\|_{L^\infty(\R)})$. One infers that
$$\max_{[0,T]\times[-A,A]}|u(\cdot,\cdot;\omega_m)-u(\cdot,\cdot;\omega)|\le\varep\Big(1+3e^{LT}\Big)$$
for all $m\ge m_0$. Since $\varep>0$ was arbitrary and $u(\cdot,\cdot;\omega)$ is continuous in $[0,+\infty)\times\R$, this shows that $u(t_m,\cdot;\omega_m)\to u(t,\cdot;\omega)$ uniformly in $[-A,A]$ as $m\to+\infty$, leading to the desired result.

Let us now show that the family $\{Q_t\}_{t\ge0}$ is subhomogeneous. So, let us consider any $\gamma\in[0,1]$ and $\omega\in\mathcal{C}_p$, and let us show that $\gamma Q_t(\omega)\le Q_t(\gamma\omega)$ in $\R$ for all $t\ge0$. From~\eqref{2.4}, it follows that $\gamma f(x,u(t,x;\omega))\le f(x,\gamma u(t,x;\omega))$ for all $(t,x)\in[0,+\infty)\times(\R\setminus S)$, hence $\gamma u(\cdot,\cdot;\omega)$ is a bounded subsolution of the problem satisfied by the bounded solution $u(\cdot,\cdot;\gamma\omega)$ in $[0,+\infty)\times\R$ (with the same initial condition $\gamma\omega$). Proposition~\ref{cp} then implies that $\gamma u(t,x;\omega)\le u(t,x;\gamma\omega)$ for all $(t,x)\in[0,+\infty)\times\R$, that is, $\gamma Q_t(\omega)\le Q_t(\gamma\omega)$ in $\R$ for all $t\ge0$.

Finally, consider any $\omega\in\mathcal{C}_p$ and $a\in l\Z$. Set $\omega_a=\omega(\cdot+a)$. The function $u(\cdot,\cdot+a;\omega)$ is still a nonnegative bounded classical solution of~\eqref{m-pb}--\eqref{d-f}, due to the periodicity of $d(x)$ and~$f(x,s)$ with respect to $x$. Since $u(0,\cdot+a;\omega)=\omega(\cdot+a)=\omega_a$, the uniqueness in Theorem~\ref{thm-m-wellposedness} implies that $u(t,x+a;\omega)=u(t,x;\omega_a)$ for all~$(t,x)\in[0,+\infty)\times\R$. This completes the proof of Proposition~\ref{prop-semiflow}.
\end{proof}

As a consequence of Theorem~\ref{thm-m-wellposedness} and Proposition~\ref{prop-semiflow}, we conclude that the solution maps $Q_t: \mathcal{C}_p\to \mathcal{C}_p$ satisfy the following properties:\begin{enumerate}
\item[(E1)] for each $t\ge0$, $Q_t$ is periodic, that is, $Q_t(T_a(\omega))=T_a(Q_t(\omega))$ for all $a\in l\mathbb{Z}$ and $\omega\in\mathcal{C}_p$, where $T_y$ is the translation operator defined by $T_y(\tilde\omega)=\tilde\omega(\cdot+y)$ for $\tilde\omega\in\mathcal{C}_p$ and $y\in l\Z$;
\item [(E2)] the set $\{Q_t(\mathcal{C}_p):t\ge0\}\subset \mathcal{C}_p$ is uniformly bounded and, for each $t\ge0$, $Q_t: \mathcal{C}_p\to \mathcal{C}_p$ is continuous; 
\item [(E3)] for each $t>0$, the map $Q_t: \mathcal{C}_p\to \mathcal{C}$ is compact with respect to the compact open topology (as a consequence of the regularity estimates of Theorem \ref{thm-m-wellposedness});
\item [(E4)] for each $t\ge0$, $Q_t$ is order-preserving (i.e., monotone);
\item [(E5)] for each $t>0$, $Q_t$ admits exactly the two periodic fixed points $0$ and $p$ in $\mathcal{C}_p$: indeed, on the one hand, one knows that $0$ and $p$ are two fixed points; on the other hand, for each $\omega\in\mathcal{C}_p\!\setminus\!\{0,p\}$, one has $Q_{mt}(\omega)\to p$ locally uniformly in $\R$ as $m\to+\infty$ by Theorem~\ref{thm-long time behavior} (and even uniformly in $\R$ if $\omega$ is periodic, by Corollary~\ref{prop-long time-periodic}), hence $\omega$ can not be a fixed point of $Q_t$.
\end{enumerate}

It then follows from \cite[Theorem 5.1]{LZ2010} that the time-$1$ map $Q_1:\mathcal{C}_p\to \mathcal{C}_p$ admits rightward and leftward asymptotic spreading speeds $c^*_+$ and $c^*_-$, in the sense that, 1) if $\omega\in\mathcal{C}_p$ is compactly supported, then
\be\label{spread1}\left\{\baa{ll}
\displaystyle\|Q_n(\omega)\|_{L^\infty([nc,+\infty))}\mathop{\longrightarrow}_{n\to+\infty}0 & \hbox{for each $c>c^*_+$},\vspace{3pt}\\
\displaystyle\|Q_n(\omega)\|_{L^\infty((-\infty,-nc'])}\mathop{\longrightarrow}_{n\to+\infty}0 & \hbox{for each $c'>c^*_-$},\eaa\right.
\ee
and, 2) if $c^*_++c^*_->0$, then, for any $\delta>0$, there is $r_\delta>0$ such that
\be\label{spread2}
\|Q_n(\omega)-p\|_{L^\infty([-nc',nc])}\mathop{\longrightarrow}_{n\to+\infty}0\hbox{ for each $c<c^*_+$ and $c'<c^*_-$ with $c+c'\ge0$},
\ee
and for each $\omega\in\mathcal{C}_p$ with $\omega\ge\delta$ on an interval of length $2r_\delta$. In the following, our goal is to give computational formulas for $c^*_{\pm}$ via the linear operators approach of \cite{W2002,LZ2007}, from which we eventually deduce that $c^*_+=c^*_-$. 

Thus, in order to compute $c^*_\pm$, we consider the linearized problem of \eqref{m-pb}--\eqref{d-f} at its zero solution:
\begin{align}
\label{7.4}
\begin{cases}
U_t=d(x)U_{xx}+f_s(x,0)U,~~&t>0,\ x\in\mathbb{R}\backslash S,\\
U(t,x^-)=U(t,x^+),~U_x(t,x^-)=\sigma U_x(t,x^+),~~&t> 0,\ x=nl,\\
U(t,x^-)=U(t,x^+),~\sigma U_x(t,x^-)=U_x(t,x^+),~~&t> 0,\ x=nl+l_2,\\
U(0,\cdot)=\omega\ge0,\ \omega\in \mathcal{C}.
\end{cases}
\end{align}
Let $\{\mathbb{L}_t\}_{t\ge 0}$ be the linear solution maps generated by~\eqref{7.4}, namely, $\mathbb{L}_t(\omega)=U(t,\cdot;\omega)$ where the function $(t,x)\mapsto U(t,x;\omega)$ is the solution of~\eqref{7.4} given by the same truncation and limit process as in the proof of Theorem~\ref{thm-m-wellposedness} (this solution satisfies the same properties as the solution of the nonlinear problem~\eqref{m-pb}--\eqref{d-f} given in Theorem~\ref{thm-m-wellposedness}, with the exception of the global boundedness: the solutions of~\eqref{7.4} are now bounded only locally with respect to $t\in[0,+\infty)$ in general). For any given $\mu\in\R$, substituting $U(t,x;\omega)=e^{-\mu x}v(t,x)$ in~\eqref{7.4} yields
\begin{align}
\label{7.5}
\begin{cases}
v_t=d(x)v_{xx}-2d(x)\mu v_x+(d(x)\mu^2+f_s(x,0))v,~~&t>0,\ x\in\mathbb{R}\backslash S,\\
v(t,x^-)=v(t,x^+),~~~[-\mu v+v_x](t,x^-)=\sigma[-\mu v+v_x](t,x^+),~~&t> 0,\ x=nl,\\
v(t,x^-)=v(t,x^+),~\,\sigma[-\mu v+ v_x](t,x^-)=[-\mu v+v_x](t,x^+),~~&t> 0,\ x=nl+l_2,\\
v(0,x)=\omega(x) e^{\mu x},~~~&x\in\mathbb{R}.
\end{cases}
\end{align}
Let $\{L_{\mu,t}\}_{t\ge 0}$ be the linear solution maps generated by \eqref{7.5} and obtained from the substitution $v(t,x)=e^{\mu x}U(t,x;\omega)$, that is, for any $\omega\in\mathcal{C}$ with $\omega\ge0$ in $\R$,
\begin{align}
\label{777}
\mathbb{L}_t\big(y\mapsto e^{-\mu y}\omega(y)\big)(x)=e^{-\mu x}L_{\mu,t}(\omega)(x), ~~\text{for}~t\ge 0\hbox{ and }x\in\mathbb{R}.
\end{align}
Substituting $v(t,x)=e^{-\lambda t}\psi(x)$ into \eqref{7.5}, with $\psi$ periodic and positive, leads to the following periodic eigenvalue problem:
\begin{equation}
\label{7.6}
\begin{aligned}
\begin{cases}
\mathcal{L}_\mu\psi(x) :=-d(x)\psi''(x)\!+\!2d(x)\mu\psi'(x)\!-\!(d(x)\mu^2\!+\!f_s(x,0))\psi(x)= \lambda\psi(x), &x\in\mathbb{R}\backslash S,\\
\psi(x^-)=\psi(x^+),~~~[-\mu \psi+\psi'](x^-)=\sigma[-\mu \psi+\psi'](x^+), &x=nl,\\
\psi(x^-)=\psi(x^+),~\,\sigma[-\mu \psi+ \psi'](x^-)=[-\mu \psi+\psi'](x^+), & x=nl+l_2,\\
\psi~\text{is periodic in}~\mathbb{R}, ~\psi>0,~ \Vert \psi\Vert_{L^\infty(\mathbb{R})}=1.
\end{cases}
\end{aligned}
\end{equation}

\begin{lemma}
\label{lemma-concave}
For each $\mu\in\R$, the eigenvalue  problem \eqref{7.6} has a simple principal eigenvalue $\lambda=\lambda(\mu)$ corresponding to a unique positive continuous periodic principal eigenfunction $\psi$, which is such that $\psi|_{\bar I}\in C^\infty(\bar I)$ for each patch $I$ in $\mathbb{R}$. Moreover, there is a max-inf characterization of $\lambda(\mu)$:
\begin{align}\label{5.16}
\lambda(\mu)=\max_{\psi\in E_\mu}\inf_{x\in\mathbb{R}\setminus S}\frac{\mathcal{L}_\mu\psi(x)}{\psi(x)},
\end{align}
where
$$\baa{rcl}
E_\mu & = & \big\{\psi\in\mathbb{P}: \psi|_{\bar I}\in C^2(\bar I)\hbox{ for each patch }I\subset\R,\,\psi>0\hbox{ in }\R,\vspace{3pt}\\
& & \ \ \psi~\text{satisfies the interface conditions in}~\eqref{7.6}\big\}\eaa$$
$($we recall that $\mathbb{P}$ is the set of all continuous and periodic functions from $\R$ to $\R$$)$.\footnote{In~\eqref{5.16}, even if the test functions $\psi$ are positive, continuous in $\R$, and have restrictions to $\bar I$ of class $C^2(\bar I)$ for each patch $I\subset\R$, the infimum of $\mathcal{L}_\mu\psi(x)/\psi(x)$ is taken over the open set $\R\setminus S$ and therefore is not a minimum in general. Notice that the quantity $\mathcal{L}_\mu\psi(x)/\psi(x)$ is in general not defined when $x\in S$, even if the limits at $x^\pm$ exist (but are different in general).} Lastly, the function $\mu\mapsto \lambda(\mu)$ is concave in $\R$, and $\lambda(0)=\lambda_1$, where $\lambda_1<0$ is the principal eigenvalue of the problem~\eqref{ep-0}.
\end{lemma}

\begin{proof}
We first fix $\mu\in\R$. The existence of a unique principal eigenvalue for problem~\eqref{7.6} can be shown similarly as for~\eqref{3.10}. This time, we introduce the space $H$ of periodic functions belonging to $H^1_{loc}(\R)$, with $\|u\|_H^2=\|u\|_{H^1(-l_1,0)}^2+(1/k)\|u\|_{H^1(0,l_2)}^2$, and $G$ the set of continuous periodic functions $u$ such that $u|_{[-l_1,0]}$ and $u|_{[0,l_2]}$ are of class $C^1([-l_1,0])$ and $C^1([0,l_2])$ respectively, with $\|u\|_G=\|u|_{[-l_1,0]}\|_{C^1([-l_1,0])}+\|u|_{[0,l_2]}\|_{C^1([0,l_2])}$. We also set $\Lambda:=\max\big(f'_1(0)+d_1\mu^2,f'_2(0)+d_2\mu^2\big)+1$. For $g\in G$, we consider the following problem
\begin{equation}
\label{7.7}
\begin{aligned}
\begin{cases}
-d(x)u''+2d(x)\mu u'+\big(\Lambda-d(x)\mu^2-f_s(x,0)\big) u= g,~~ &\text{in }\mathbb{R}\backslash S,\\
u(x^-)= u(x^+),~[-\mu  u+ u'](x^-)=\sigma[-\mu  u+ u'](x^+),~~&x=nl,\\
u(x^-)= u(x^+),~\sigma[-\mu  u+  u'](x^-)=[-\mu  u+ u'](x^+),~~& x=nl+l_2,\\
u~\text{is periodic}.
\end{cases}
\end{aligned}
\end{equation}
We can solve this problem first in a weak sense, that is, we look for $u\in H$ such that $B(u,z)=\langle g, z \rangle$ for all $z\in H$, where the bilinear form $B$ is defined by
$$\baa{rcl}
B(u,z) & = & \displaystyle\int^0_{-l_1} d_1 u'z'+d_1\mu u'z-d_1\mu uz'+\big(\Lambda-d_1 \mu^2-f'_1(0)\big)uz\vspace{3pt}\\
& & \displaystyle+\frac{1}{k}\int_0^{l_2} d_2 u'z'+d_2\mu u'z-d_2\mu uz'+\big(\Lambda-d_2\mu^2-f'_2(0)\big)uz,\eaa$$
and the scalar product $\langle,\rangle$ is defined by
\begin{align*}
\langle g, z\rangle=\int^0_{-l_1} gz ~+\frac{1}{k}\int_0^{l_2} gz ~.
\end{align*} 
Clearly, the map $z\mapsto\langle g,z\rangle$ is continuous in $H$, and $B$ is continuous in $H\times H$. Moreover, it is easily seen that, for any $u\in H$, $B(u,u)\ge \min\big(d_1,d_2,1\big)\Vert u \Vert^2_H$, whence $B$ is coercive. The Lax-Milgram theorem implies the existence of a unique $u\in H$ (hence,~$u$ can be identified with its unique continuous representative in $\R$) satisfying $B(u,z)=\langle g, z \rangle$ for all $z\in H$, and $\|u\|_H\le C_1\|g\|_{L^2(-l_1,l_2)}$ for a positive constant $C_1$ only depending on $d_{1,2}$ and~$k$. As for~\eqref{3.10}, one gets that $u|_{[-l_1,0]}$ and~$u|_{[0,l_2]}$ are in $H^2([-l_1,0])$ and $H^2([0,l_2])$ and then in $C^3([-l_1,0])$ and $C^3([0,l_2])$ respectively, that $\|u|_{[-l_1,0]}\|_{C^3([-l_1,0])}+\|u|_{[0,l_2]}\|_{C^3([0,l_2])}\le C_2\|g\|_G$ with a positive constant $C_2$ depending only on $d_{1,2}$, $k$, $l_{1,2}$, $f'_{1,2}(0)$ and $\mu$, and that the equations in~\eqref{7.7} are satisfied pointwise. Therefore, the linear mapping $T:g\in G\mapsto Tg:=u\in G$ is compact. Let now $\mathcal{K}$ be the cone $\mathcal{K}=\{u\in G:u\ge 0~\text{in }\R\}$. Its interior $\mathring{\mathcal{K}}$ is not empty, and $\mathcal{K}\cap(-\mathcal{K})=\{0\}$.  We claim that, if $g\in\mathcal{K}\!\setminus\!\{0\}$, then $u\in\mathring{\mathcal{K}}$. Indeed, by using the equality $B(u,z)=\langle g, z\rangle
$ with~$z:=u^-=\max(-u,0)\in H$, one has
$$-\int_{-l_1}^0\!\!\!\big(d_1((u^-)')^2+(\Lambda-d_1\mu^2-f'_1(0))(u^-)^2\big)-\int_0^{l_2}\!\!\!\big(d_2((u^-)')^2+(\Lambda-d_2\mu^2-f'_2(0))(u^-)^2\big)\!=\!\langle g,u^-\rangle\!\ge\!0,$$
hence $u^-\equiv 0$, that is, $u\ge 0$ in $\R$. From the strong elliptic maximum principle and the Hopf lemma, together with the fact that $g\ge\not\equiv0$, one concludes that~$u>0$ in~$[-l_1,l_2]$ and then in~$\R$ by periodicity. Therefore, $T(\mathcal{K}\!\setminus\!\{0\})\subset\mathring{\mathcal{K}}$. As for~\eqref{3.10}, one then infers from the Krein-Rutman theory the existence and uniqueness of a principal eigenpair $(\lambda,\psi)$ solving~\eqref{7.6}. We then call~$\lambda(\mu)$ this principal eigenvalue $\lambda$. Notice that, for each patch $I\subset\R$, the function $\psi|_{\bar I}$ is then of class $C^\infty(\bar I)$ since $d$ and $f_s(\cdot,0)$ are constant in $I$.

Let us now prove the max-inf representation~\eqref{5.16} of $\lambda(\mu)$. Since $\psi\in E_\mu$, one has
\begin{align*}
\lambda(\mu)\le \sup_{\psi\in E_\mu}\inf_{x\in\mathbb{R}\setminus S}\frac{\mathcal{L}_\mu\psi(x)}{\psi(x)}.
\end{align*}
To show the reverse inequality, assume by way of contradiction that there is $\varphi\in E_\mu$ such that
\begin{align*}
\lambda(\mu)<\inf_{x\in\mathbb{R}\setminus S}\frac{\mathcal{L}_\mu\varphi(x)}{\varphi(x)}.
\end{align*}
Then there exists $\eta>0$ such that
$$-d(x)\varphi''(x)+2d(x)\mu\varphi'(x)-(d(x)\mu^2+f_s(x,0))\varphi(x)-\lambda(\mu)\varphi(x)\ge \eta \varphi(x)>0~~\text{for all }x\in\R\setminus S.$$
 Since $\psi,\varphi\in E_\mu$, there exists $\vartheta>0$ such that $\varphi\ge \vartheta\psi$ in $\mathbb{R}$ with equality somewhere. Set $w:=\varphi-\vartheta\psi$. Then $w\in\mathbb{P}$, $w\ge0$ in $\R$, $w|_{\bar I}\in C^2(\bar I)$ for each patch $I\subset\R$, and $w$ satisfies
\begin{equation}\label{eqw}
\begin{aligned}
\begin{cases}
-d(x)w''(x)+2d(x)\mu w'(x)-(d(x)\mu^2+f_s(x,0))w(x)-\lambda(\mu) w(x)>0, &x\in\mathbb{R}\backslash S,\\
w(x^-)= w(x^+),~\ \ [-\mu  w+ w'](x^-)=\sigma[-\mu  w+ w'](x^+), &x=nl,\\
w(x^-)= w(x^+),~\,\sigma[-\mu  w+  w'](x^-)=[-\mu  w+ w'](x^+), & x=nl+l_2,\\
w~\text{is}~\text{periodic},
\end{cases}
\end{aligned}
\end{equation}
and there exists $x_0\in\mathbb{R}$ such that $w(x_0)=0$. The point $x_0$ can not belong to $\R\setminus S$ because of the strict inequality in the first line of~\eqref{eqw}. Therefore, $w>0$ in $\R\setminus S$ and $x_0\in S$. The Hopf lemma then implies that $w'(x_0^+)>0$ and $w'(x_0^-)<0$, with $w(x_0)=0$, contradicting the interface conditions in~\eqref{eqw}. One has then reached a contradiction. Hence,
\begin{align*}
\lambda(\mu)\ge \sup_{\psi\in E_\mu}\inf_{x\in\mathbb{R}\setminus S}\frac{\mathcal{L}_\mu\psi(x)}{\psi(x)}.
\end{align*}
The max-inf characterization~\eqref{5.16} of $\lambda (\mu)$ follows, and the supremum is a maximum since $\psi\in E_\mu$.
 
Next, we prove the concavity of the function $\mu\mapsto\lambda(\mu)$. With the change of functions $\psi(x)=e^{\mu x}\widetilde\psi(x)$ in~\eqref{5.16}, one has 
\begin{align*}
\frac{\mathcal{L}_\mu\psi(x)}{\psi(x)}=\frac{-d(x)\widetilde\psi''(x)}{\widetilde\psi(x)}-f_s(x,0)\ \hbox{ for all }x\in\R\setminus S,
\end{align*}
hence
\begin{align}
\label{5.17}
\lambda(\mu)=\max_{\widetilde\psi\in \widetilde E_\mu}\inf_{x\in\mathbb{R}\setminus S}\bigg(\frac{-d(x)\widetilde\psi''(x)}{\widetilde\psi(x)}-f_s(x,0)\bigg),
\end{align}
where
$$\baa{rcl}
\widetilde E_\mu & = & \big\{\widetilde\psi\in\mathcal{C}: x\mapsto e^{\mu x}\widetilde{\psi}(x)\in\mathbb{P},\ \widetilde \psi|_{\bar I}\in C^2(\bar I)\hbox{ for each patch }I\subset\R,\ \widetilde\psi>0\hbox{ in }\R,\vspace{3pt}\\
& & \ \ \widetilde\psi~\text{satisfies the interface conditions in}~\eqref{ep-0}\big\}.\eaa$$
Consider any real numbers $\mu_1$ and $\mu_2$, and any $t\in[0,1]$, and set $\mu=t\mu_1+(1-t)\mu_2$. One has to verify that $\lambda(\mu)\ge t\lambda(\mu_1)+(1-t)\lambda(\mu_2)$. Let $\widetilde\psi_1$ and $\widetilde\psi_2$ be arbitrarily chosen in $\widetilde E_{\mu_1}$ and $\widetilde E_{\mu_2}$, respectively. Define $z_1=\ln\widetilde\psi_1$, $z_2=\ln\widetilde\psi_2$, $z=tz_1+(1-t)z_2$ and $\widetilde\psi=e^z$. We claim that $\widetilde\psi\in\widetilde E_\mu$. In fact, since $\widetilde\psi=\widetilde\psi_1^t\widetilde\psi_2^{1-t}$, then $\widetilde\psi\in\mathcal{C}$ and $\widetilde\psi|_{\bar I}\in C^2(\bar I)$ for each patch $I\subset\R$. Furthermore, the function $x\mapsto e^{\mu x}\widetilde{\psi}(x)=(e^{\mu_1x}\widetilde{\psi}_1(x))^t\times(e^{\mu_2x}\widetilde{\psi}_2(x))^{1-t}$ is periodic, and the flux conditions in~\eqref{ep-0} can be easily derived from $\widetilde\psi'=(\widetilde\psi_1^t\widetilde\psi_2^{1-t})'=t\widetilde\psi_1^{t-1}\widetilde\psi_1'\widetilde\psi_2^{1-t}+(1-t)\widetilde\psi_2^{-t}\widetilde\psi'_2\widetilde\psi_1^t$ in $\R\setminus S$ and from the fact that both $\widetilde\psi_1\in \widetilde E_{\mu_1}$ and $\widetilde\psi_2\in \widetilde E_{\mu_2}$ satisfy  the interface conditions in~\eqref{ep-0}. Therefore, by \eqref{5.17} we have
\begin{align*}
\lambda(\mu)\ge \inf_{x\in\mathbb{R}\setminus S}\bigg(\frac{-d(x)\widetilde\psi''(x)}{\widetilde\psi(x)}-f_s(x,0)\bigg).
\end{align*}
Notice that, for each $x\in\R\setminus S$, one has $-d(x)\widetilde\psi''(x)/\widetilde\psi(x)=-d(x)z''(x)-d(x)(z'(x))^2$, and 
\begin{align*}
(z'(x))^2=&(tz_1'(x)+(1-t)z_2'(x))^2= t(z'_1(x))^2+(1-t)(z'_2(x))^2-t(1-t)(z'_1(x)-z'_2(x))^2\\
\le &t(z'_1(x))^2+(1-t)(z'_2(x))^2,
\end{align*}
hence
$$\baa{rcl}
\displaystyle\frac{-d(x)\widetilde\psi''(x)}{\widetilde\psi(x)}-f_s(x,0) & \ge & t\,\big(\!-d(x)z_1''(x)-d(x)(z'_1(x))^2-f_s(x,0)\big)\vspace{3pt}\\
& & +(1-t)\,\big(\!-d(x)z_2''(x)-d(x)(z'_2(x))^2-f_s(x,0)\big).\eaa$$
Eventually, we find that 
$$\baa{rcl}
\lambda(\mu) & \ge & \displaystyle\inf_{x\in\mathbb{R}\setminus S}\bigg(\frac{-d(x)\widetilde\psi''(x)}{\widetilde\psi(x)}-f_s(x,0)\bigg)\vspace{3pt}\\
& \ge & \displaystyle t\inf_{x\in\mathbb{R}\setminus S}\bigg(\frac{-d(x)\widetilde\psi_1''(x)}{\widetilde\psi_1(x)}-f_s(x,0)\bigg) +(1-t)\inf_{x\in\mathbb{R}\setminus S}\bigg(\frac{-d(x)\widetilde\psi_2''(x)}{\widetilde\psi_2(x)}-f_s(x,0)\bigg).\eaa$$
Since $\widetilde\psi_1$ and $\widetilde\psi_2$ were arbitrarily chosen in $\widetilde E_{\mu_1}$ and $\widetilde E_{\mu_2}$ respectively, one infers from \eqref{5.17} that  $\lambda(\mu)\ge t\lambda(\mu_1)+(1-t)\lambda(\mu_2)$. That is, $\mu\mapsto\lambda(\mu)$ is concave in $\R$, which also yields the continuity of this function. Lastly, we also observe that the problem~\eqref{7.6} coincides with~\eqref{ep-0} when $\mu=0$, that is, $\lambda(0)=\lambda_1$, which is here negative by assumption. This completes the proof of Lemma~\ref{lemma-concave}.
\end{proof}

Since for each $t>0$ the linear operator $L_{\mu,t}$ defined by~\eqref{7.5}--\eqref{777} is strongly positive and compact, the Krein-Rutman theorem again implies that its spectral radius $r(L_{\mu,t})$ is positive and is the principal eigenvalue of $L_{\mu,t}$, that is, $r(L_{\mu,t})=e^{-\lambda(\mu)t}$.

We are now in a position to give variational formulas for the rightward and leftward asymptotic spreading speeds $c^*_\pm$ given by~\eqref{spread1}--\eqref{spread2} via the linear operators approach.
 
\begin{theorem}\label{thm-formula of asp}
Let $c^*_+$ and $c^*_-$ be the rightward and leftward asymptotic spreading speeds of $Q_1$, given by~\eqref{spread1}--\eqref{spread2}. Then, 
\begin{align}
\label{c_+-}
c^*_+=\inf_{\mu>0}\frac{-\lambda(\mu)}{\mu},~~c^*_-=\inf_{\mu>0}\frac{-\lambda(-\mu)}{\mu}.
\end{align}
Furthermore, we have $c^*_+=c^*_->0$. 
\end{theorem}

\begin{proof}
Due to assumption \eqref{2.4}, and we have $f(x,s)\le f_s(x,0)s$ for all $x\in\mathbb{R}\setminus S$ and $s\ge0$. Then, for every $\omega\in\mathcal{C}_p$ and $n\in\N$, the solution $u^n:=u^n(\cdot,\cdot;\omega)$ of~\eqref{tp-1}--\eqref{tp-4} satisfies
\begin{equation*}
\begin{aligned}
\begin{cases}
u^n_t\le d(x)u^n_{xx}+f_s(x,0)u^n,~~&t>0,\ x\in(-nl,nl)\backslash S,\\
u^n(t,x^-)=u^n(t,x^+),~u^n_x(t,x^-)=\sigma u^n_x(t,x^+),~~&t> 0,\ x\in S_1\cap(-nl,nl),\\
u^n(t,x^-)=u^n(t,x^+),~\sigma u^n_x(t,x^-)=u^n_x(t,x^+),~~&t> 0,\ x\in S_2\cap(-nl,nl).
\end{cases}
\end{aligned}
\end{equation*}
Proposition~\ref{bdd-cp} and the construction of the solutions $U(\cdot,\cdot;\omega)$ of~\eqref{7.4} by using the same truncation and limit process as in the proof of Theorem~\ref{thm-m-wellposedness} imply that $u^n(t,x;\omega)\le U^n(t,x;\omega)$ for all $(t,x)\in[0,+\infty)\times[-nl,nl]$ and $n\in\N$, hence $u(t,x;\omega)\le U(t,x,\omega)$ for all $(t,x)\in[0,+\infty)\times\R$. In other words, $Q_t(\omega)\le\mathbb{L}_t(\omega)$ in $\R$ for all $t\ge0$ and  $\omega\in\mathcal{C}_p$. Particularly, by taking $t=1$, one has $Q_1(\omega)\le\mathbb{L}_1(\omega)$ in $\R$ for all $\omega\in\mathcal{C}_p$. For any $\mu\in\R$, one has
$$\mathbb{L}^\mu(\psi)(x):=e^{\mu x}\mathbb{L}_1(y\mapsto e^{-\mu y}\psi(y))(x)=L_{\mu,1}(\psi)(x)\ \hbox{ for every $\psi\in\mathbb{P}$ and $x\in\R$},$$
thanks to \eqref{777}. It then follows that $e^{-\lambda(\mu)}$ is the principal eigenvalue of $\mathbb{L}^\mu$. On the other hand, by Lemma \ref{lemma-concave}, the function $\mu\mapsto \ln (e^{-\lambda(\mu)})=-\lambda(\mu)$ is convex. With similar arguments as in~\cite[Theorem~2.5]{W2002} and in~\cite[Theorem~3.10~(i)]{LZ2007}, one then obtains that 
\begin{align}
\label{7.8}
c^*_+\le\inf_{\mu>0}\frac{\ln(e^{-\lambda(\mu)})}{\mu}=\inf_{\mu>0}\frac{-\lambda(\mu)}{\mu}.
\end{align}

On the other hand,  for any given $\varep>0$, there is a $\delta>0$ such that 
\begin{align*}
f(x,s)\ge (f_s(x,0)-\varep)s,~~\hbox{for all $x\in\mathbb{R}\setminus S$ and $u\in[0,\delta]$}.
\end{align*}
From the continuity of the solutions of~\eqref{m-pb}--\eqref{d-f} with respect to the initial conditions, as stated in Theorem~\ref{thm-m-wellposedness}, there is a positive real number $\eta$ such that $\eta\le p$ in $\R$ and $u(t,x;\eta)\le \delta$ for all $t\in[0,1]$ and $x\in\mathbb{R}$. Define
$$\mathcal{C}_\eta=\big\{\omega\in\mathcal{C}: 0\le\omega\le\eta\hbox{ in }\R\big\}.$$
It then follows from Theorem~\ref{thm-m-wellposedness} that 
\begin{equation*}
0\le u(t,x;\omega)\le u(t,x;\eta)\le\delta,~~\hbox{for all $\omega\in\mathcal{C}_\eta$, $t\in[0,1]$ and $x\in\mathbb{R}$}.
\end{equation*}
Thus, for any $\omega\in\mathcal{C}_\eta$, the solution $u(\cdot,\cdot;\omega)$ to \eqref{m-pb}--\eqref{d-f} satisfies
\begin{equation*}
\label{7.9}
\begin{aligned}
\begin{cases}
u_t\ge d(x)u_{xx}+(f_s(x,0)-\varep)u,~~&t\in(0,1],\ x\in\mathbb{R}\backslash S,\\
u(t,x^-)=u(t,x^+),~u_x(t,x^-)=\sigma u_x(t,x^+),~~&t\in(0,1],\ x=nl,\\
u(t,x^-)=u(t,x^+),~\sigma u_x(t,x^-)=u_x(t,x^+),~~&t\in(0,1],\ x=nl+l_2.
\end{cases}
\end{aligned}
\end{equation*}
Consider now the linear problem 
\begin{equation}
\label{7.10}
\begin{aligned}
\begin{cases}
V_t= d(x)V_{xx}+(f_s(x,0)-\varep)V,~~&t>0,\ x\in\mathbb{R}\backslash S,\\
V(t,x^-)=V(t,x^+),~V_x(t,x^-)=\sigma V_x(t,x^+),~~&t> 0,\ x=nl,\\
V(t,x^-)=V(t,x^+),~\sigma V_x(t,x^-)=V_x(t,x^+),~~&t> 0,\ x=nl+l_2.
\end{cases}
\end{aligned}
\end{equation}
Let $\{\mathbb{L}^\varep_t\}_{t\ge 0}$ be the solution maps generated by the above linear system, as for~\eqref{7.4} above ($\mathbb{L}^0_t=\mathbb{L}_t$ for all $t\ge0$). Then, Proposition~\ref{bdd-cp} and the construction of the solutions of~\eqref{7.10} as in the proof of Theorem~\ref{thm-m-wellposedness}  imply that $\mathbb{L}^\varep_t(\omega)\le Q_t(\omega)$ in $\R$ for all $t\in[0,1]$ and $\omega\in\mathcal{C}_\eta$. In particular, $\mathbb{L}^\varep_1(\omega)\le Q_1(\omega)$ in $\R$ for all $\omega\in\mathcal{C}_\eta$. 
 
Denote by $\lambda^\varep(\mu)$ the first eigenvalue of the following eigenvalue problem: 
\begin{equation*}
\label{7.6'}
\begin{aligned}
\begin{cases}
-d(x)\psi''(x)+2\mu d(x)\psi'(x)-(d(x)\mu^2+(f_s(x,0)-\varep))\psi(x)= \lambda^\varep(\mu)\,\psi(x),~~ &x\in\mathbb{R}\backslash S,\\
\psi(x^-)=\psi(x^+),~[-\mu \psi+\psi'](x^-)=\sigma[-\mu \psi+\psi'](x^+),~~&x=nl,\\
\psi(x^-)=\psi(x^+),~\sigma[-\mu \psi+ \psi'](x^-)=[-\mu \psi+\psi'](x^+),~~ &x=nl+l_2,\\
\psi~\text{is periodic in}~\mathbb{R}, ~\psi>0,~ \Vert \psi\Vert_{L^\infty(\mathbb{R})}=1.
\end{cases}
\end{aligned}
\end{equation*}
By uniqueness of the principal eigenvalue of~\eqref{7.6}, there holds $\lambda^\varep(\mu)=\lambda(\mu)+\varep$. From the convexity of the function $\mu\mapsto-\lambda^\varep(\mu)=-\lambda(\mu)-\varep$ and the arguments in \cite[Theorem~2.4]{W2002} and in \cite[Theorem~3.10~(ii)]{LZ2007}, one infers that
\begin{align}
\label{7.11}
c^*_+\ge\inf_{\mu>0}\frac{\ln(e^{-\lambda^\varep(\mu)})}{\mu}=\inf_{\mu>0}\frac{-\lambda(\mu)-\varep}{\mu},
\end{align}
and this property is valid for all $\varep>0$. Together with~\eqref{7.8},~\eqref{7.11}, and the positivity of $-\lambda(0)=-\lambda_1$, it follows that
$$c^*_+=\inf_{\mu>0}\frac{-\lambda(\mu)}{\mu}.$$

By the change of variable $v(t,x;\omega)=u(t,-l_1-x;\omega(-l_1-\cdot))$,\footnote{Notice that $p(x)=p(-l_1-x)$ for all $x\in\R$ by invariance of~\eqref{m-pb-elliptic} with respect to this change of variable and by the uniqueness result of Theorem~\ref{thm-2.4-uniqueness}, hence $x\mapsto\omega(-l_1-x)\in\mathcal{C}_p$ for every $\omega\in\mathcal{C}_p$.} one gets that $c^*_-$ is the rightward asymptotic spreading speed of the resulting problem for the solutions $v(\cdot,\cdot;\omega)$. Therefore,
\begin{align*}
c^*_-=\inf_{\mu>0}\frac{-\lambda(-\mu)}{\mu}.
\end{align*}
Consequently, \eqref{c_+-} is proved.

Next, for any $\mu\in\R$, if $\psi_\mu$ is the principal eigenfunction of the problem~\eqref{7.6}, with principal eigenvalue $\lambda(\mu)$, then the function $x\mapsto\hat\psi(x):=\psi_\mu(-l_1-x)$ satisfies~\eqref{7.6} with $-\mu$ instead of $\mu$ in the equations and the interface conditions, but with the same eigenvalue $\lambda(\mu)$. By uniqueness of the principal eigenvalue, one deduces that $\lambda(-\mu)=\lambda(\mu)$. Therefore,~\eqref{c_+-} yields $c^*_+=c^*_-$.

Lastly, consider any compactly supported $\omega\in\mathcal{C}_p$ such that $\omega\not\equiv0$ in $\R$ and $\omega(x)<p(x)$ for all $x\in\R$ (and remember that $\omega$ and $p$ are continuous, and that $p$ is periodic and positive in $\R$). From Theorem~\ref{thm-long time behavior}~(i), one knows that $u(t,\cdot;\omega)\to p$ as $t\to+\infty$ locally uniformly in $\R$. Hence, there is $T\in\N$ such that $u(T,\cdot;\omega)\ge\omega(\cdot\pm l)$ in $\R$. Theorem~\ref{thm-m-wellposedness} and Proposition~\ref{prop-semiflow} then imply in particular that $u(2T,\cdot;\omega)\ge\omega(\cdot\pm 2l)$ in $\R$, hence $u(mT,\cdot;\omega)\ge\omega(\cdot\pm ml)$ in $\R$ for all $m\in\N$ by an immediate induction. In other words, $Q_{mT}(\omega)\ge\omega(\cdot\pm ml)$ in $\R$ for all $m\in\N$, and it follows from property~\eqref{spread1} that $c^*_\pm>0$. This completes the proof of Theorem~\ref{thm-formula of asp}. 
\end{proof}

\begin{proof}[Proofs of Theorems~$\ref{thm-spreading result}$ and~$\ref{thm-existence of PTW}$]
By~\cite[Theorems~5.2 and~5.3]{LZ2010}, together with Theorems~\ref{thm-long time behavior}~(i) and~\ref{thm-formula of asp}, one directly obtains Theorem~\ref{thm-spreading result}, with spreading speed $c^*:=c^*_\pm$, as well as the existence of time-nondecreasing periodic rightward and leftward traveling waves for problem~\eqref{m-pb}--\eqref{d-f} with all and only all speeds $c\ge c^*$. To complete the proof of Theorem~\ref{thm-existence of PTW}, it is left to show that these periodic traveling waves are strictly monotone in time. For $c\ge c^*>0$, consider a periodic rightward (the case of leftward waves can be handled similarly) traveling wave solving~\eqref{m-pb}--\eqref{d-f}, written as $u(t,x)=W(x-ct,x)$ (with $u(t,x)=Q_{t-t'}(u(t',\cdot))(x)$ for all $t'\le t\in\R$ and $x\in\R$), where $W(s,x)$ is periodic in $x$, nonincreasing in $s$, and $W(-\infty,x)=p(x)$, $W(+\infty,x)=0$ for all $x\in\R$. Notice in particular that $0\le u(t,x)\le p(x)$ for all $(t,x)\in\R\times\R$, that $u(t,x)\to 0$ as $t\to-\infty$ and $u(t,x)\to p(x)$ as $t\to+\infty$ for every $x\in\R$, and that $u(t+h,x)\ge u(t,x)$ for every $h>0$ and $(t,x)\in\R^2$. From Proposition~\ref{cp}, it follows that, for every $h>0$ and~$t_0\in\R$, either $u(\cdot+h,\cdot)\equiv u$ in $(t_0,+\infty)\times\R$, or $u(\cdot+h,\cdot)>u$ in $(t_0,+\infty)\times\R$. Since $u(-\infty,x)=0<p(x)=u(+\infty,x)$ for every $x\in\R$, one easily infers that, for every $h>0$, $u(\cdot+h,\cdot)>u$ in $\R\times\R$. Therefore, $u$ is increasing in $t$ and the periodic rightward traveling wave~$W(x-ct,x)$ is decreasing in its first argument, and all properties of Definition~\ref{def4} are therefore satisfied.
\end{proof}
 

\appendix
\section{Comparison principles}
\label{Appendix-A}

In this appendix, we prove comparison results for the problem~\eqref{m-pb}--\eqref{d-f}, as well as for a class of more general non-periodic versions of~\eqref{m-pb}--\eqref{d-f}, and for the patchy model in an interval $(a,b)\subset\mathbb{R}$ composed of finitely many patches, say $I_i$ for $i=1,\ldots,n$. For the latter, which we first deal with, the landscape $(a,b)$ can be either bounded or unbounded.  Set $-\infty\le a=x_0<x_1<\cdots<x_n=b\le+\infty$ and $I_i=(x_{i-1},x_i)$ for $i=1,\ldots,n$. Since the results will be used in the present paper and in the future work~\cite{HLZ2}, we state them in more generality to cover  different applications. We consider a one-dimensional parabolic operator
\begin{align*}
\mathcal{L}u:=u_t-d(x)u_{xx}-c(t,x)u_x-F(x,u),~~\text{for}~t>0\hbox{ and }x\in(a,b)\backslash\{x_1,\ldots,x_{n-1}\}=\bigcup_{i=1}^nI_i,
\end{align*}
with interface conditions
\begin{align}
\label{app-interface}
u(t,x_i^-)=u(t,x_i^+)\hbox{ and }u_x(t,x_i^-)=\sigma_iu_x(t,x_i^+),~~\hbox{for }t>0\hbox{ and }i=1,\ldots,n-1.
\end{align}
If $a$ or $b$ is finite, we impose Dirichlet-type boundary conditions:
\begin{equation}
\label{app-boundary}
u(t,a)=\varphi^-(t)~~\text{or}~~u(t,b)=\varphi^+(t),~~\text{for}~t\ge0,
\end{equation}
where $\varphi^\pm:[0,+\infty)\to\R$ are given continuous functions. Here, the function $x\mapsto d(x)$ is assumed to be constant and positive in each patch, i.e., $d|_{I_i}=d_i>0$ for some constant $d_i$,\footnote{From the proofs below, it is easily seen that we can consider more general diffusion coefficients $d(t,x)$ such that $d|_{(0,+\infty)\times I_i}$ can be extended to a continuous and positive function in $[0,+\infty)\times\overline{I_i}$, for each $1\le i\le n$.} the function~$c$ is assumed to be continuous and bounded in $(0,T_0)\times\cup_{i=1}^nI_i$ for every $T_0\in(0,+\infty)$, the $\sigma_i$'s are given positive real numbers, and, for each $1\le i\le n$, $F(x,s)=f_i(s)$ for $(x,s)\in I_i\times\R$, with~$f_i\in C^1(\R)$.

We first give the definition of super- and subsolutions of $\mathcal{L}u=0$ associated with the interface and boundary conditions \eqref{app-interface}--\eqref{app-boundary}.
 
\begin{definition}\label{defsubsuper}
For $T\in(0,+\infty]$, we say that a continuous function $\overline{u}:[0,T)\times\overline{(a,b)}\to\R$,\footnote{The notation $\overline{(a,b)}$ covers all possible four cases when $a$ or $b$ is finite or not. If $a$ and $b$ are finite, then $\overline{(a,b)}=[a,b]$.} which is assumed to be bounded in $[0,T_0]\times\overline{(a,b)}$ for every $T_0\in(0,T)$, is a supersolution for the problem~$\mathcal{L}u=0$ with interface and boundary conditions \eqref{app-interface}--\eqref{app-boundary}, if $\overline{u}|_{(0,T)\times\overline{I_i}}\in  C^{1;2}_{t;x}((0,T)\times\overline{I_i})$ satisfies $\mathcal{L}\overline{u}|_{(0,T)\times I_i}\ge 0$ in the classical sense for each $1\le i\le n$, and if
\begin{equation*}
\overline{u}_x(t,x_i^-)\ge \sigma_i \overline{u}_x(t,x_i^+),~~\text{for}~t\in(0,T)\text{ and}~i=1,\ldots,n-1,
\end{equation*}
and
\begin{equation*}
\overline{u}(t,a)\ge\varphi^-(t)~\text{ or }~\overline{u}(t,b)\ge\varphi^+(t),~~\text{for}~t\in [0,T),
\end{equation*}
provided that $a$ or $b$ is finite. A subsolution can be defined in a similar way with all the inequality signs above reversed.
\end{definition}

The first result of the appendix is a comparison principle between super- and subsolutions when the interval $(a,b)$ is bounded.

\begin{proposition}[Comparison principle in bounded intervals]
\label{bdd-cp}
Assume that $-\infty<a<b<+\infty$. For $T\in(0,+\infty]$, let $\overline{u}$ and $\underline{u}$ be, respectively, a super- and a subsolution in~$[0,T)\times[a,b]$ of $\mathcal{L}u=0$ with~\eqref{app-interface}--\eqref{app-boundary}, and assume that $\overline{u}(0,\cdot)\ge\underline{u}(0,\cdot)$ in $[a,b]$. Then, $\overline{u}\ge\underline{u}$ in~$[0,T)\times[a,b]$ and, if $\overline{u}(0,\cdot)\not\equiv\underline{u}(0,\cdot)$, then $\overline{u}>\underline{u}$ in $(0,T)\times(a,b)$.
\end{proposition}
 
\begin{proof}
Fix any $T_0\in(0,T)$ and set
\be\label{defmu}
M:=\max\big(\|\overline{u}\|_{L^\infty([0,T_0]\times[a,b])},\|\underline{u}\|_{L^\infty([0,T_0]\times[a,b])}\big)\ \hbox{ and }\ \mu:=\max\limits_{1\le i\le  n}\|f'_i\|_{L^\infty([-M,M])}
\ee
(notice that $M$ and $\mu$ are nonnegative real numbers owing to the assumptions on $\overline{u}$, $\underline{u}$ and $f_i$). Define
$$w(t,x):=(\overline{u}(t,x)-\underline{u}(t,x))\,e^{-\mu t}\ \hbox{ for $(t,x)\in[0,T_0]\times[a,b]$}.$$
The function $w$ is continuous in $[0,T_0]\times[a,b]$, with restriction in $(0,T_0]\times\overline{I_i}$ of class $C^{1;2}_{t;x}((0,T_0]\times\overline{I_i})$ for each $1\le i\le n$, and we see from the mean value theorem that $w$ satisfies
\begin{align}\label{w1}
\mathcal{N}w:=w_t-d(x)w_{xx}-c(t,x)w_x+\big(\mu-F_s(x,\eta(t,x))\big)w\ge 0,~\text{for}~(t,x)\in(0,T_0]\!\times\!\bigcup\limits_{i=1}^n I_i,
\end{align} 
where $\eta(t,x)$ is an intermediate value between $\overline{u}(t,x)$ and $\underline{u}(t,x)$ (hence, $|\eta(t,x)|\le M$ and $\mu-F_s(x,\eta(t,x))\ge0$). Moreover, there holds
\begin{align}\label{w2}
w_x(t,x_i^-)\ge \sigma_i w_x(t,x_i^+),~~\text{for}~t\in(0,T_0]~\text{and}~i=1,\ldots,n-1,
\end{align}
together with $w(0,x)=\overline{u}(0,x)-\underline{u}(0,x)\ge 0$ for all $x\in[a,b]$, $w(t,a)\ge 0$ and $w(t,b)\ge 0$ for all $t\in[0,T_0]$. 
 	
Consider now an arbitrary $\varep>0$ and let us introduce the auxiliary function $z$ defined by
$$z(t,x):=w(t,x)+\varep(t+1)\ \hbox{ for $(t,x)\in[0,T_0]\times[a,b]$}.$$
The function $z$ has at least the same regularity as $w$, and $z>0$ in $\{0\}\times[a,b]$ and in $[0,T_0]\times\{a,b\}$. Moreover,
\be\label{Nz1}
\mathcal{N}z=\mathcal{N}w+\varep+ \big(\mu-F_s(x,\eta(t,x))\big)\varep(t+1)\ge\varep>0,~~\text{for}~(t,x)\in(0,T_0]\times\bigcup\limits_{i=1}^n I_i,
\ee
with
\be\label{z}
z_x(t,x_i^-)\ge \sigma_i z_x(t,x_i^+),~~\text{for}~t\in(0,T_0]~\text{and}~i=1,\ldots,n-1.
\ee
We claim that $z(t,x)>0$ for all $(t,x)\in[0,T_0]\times[a,b]$. Assume not. Then, by continuity, there is a point $(t_0,y_0)\in(0,T_0]\times(a,b)$ such that $z(t_0,y_0)=\min_{[0,t_0]\times[a,b]}z=0$. We first assume that $y_0\in I_i$ for some $1\le i\le n$. Since $z_t(t_0,y_0)\le 0$, $z_x(t_0,y_0)=0$ and $z_{xx}(t_0,y_0)\ge 0$, we see that
\be\label{Nz}
\mathcal{N}z(t_0,y_0)=z_t(t_0,y_0)-d_iz_{xx}(t_0,y_0)+c(t_0,y_0)z_x(t_0,y_0)+\big(\mu-f_i'(\eta(t_0,y_0))\big)z(t_0,y_0)\le 0,
\ee
which is impossible by~\eqref{Nz1}. Thus, necessarily, we can assume without loss of generality that $y_0=x_i$ for some $1\le i\le n-1$ and that $z>0$ in $[0,t_0]\times\cup_{i=1}^nI_i$. Then, the Hopf lemma yields
\begin{align*}
z_x(t_0,x_i^-)<0\ \hbox{ and }\ z_x(t_1,x_i^+)>0,
\end{align*}
which contradicts \eqref{z}. Consequently, $z>0$ in $[0,T_0]\times[a,b]$. Since $\varep>0$ was arbitrarily chosen, we obtain that $w\ge 0$ in $[0,T_0]\times[a,b]$, which immediately implies $\overline{u}\ge\underline{u}$ in $[0,T_0]\times[a,b]$, and then in $[0,T)\times[a,b]$ since $T_0\in(0,T)$ was arbitrary.

Let us now further assume that $\overline{u}(0,\cdot)\not\equiv\underline{u}(0,\cdot)$ in $[a,b]$, hence by continuity $\overline{u}(0,\cdot)>\underline{u}(0,\cdot)$ in some non-empty open subinterval of $(a,b)$ which has a non-empty intersection with $I_i$, for some $1\le i\le n$. Since we already know from the previous paragraph that $\overline{u}\ge\underline{u}$ in $[0,T)\times[a,b]$, it follows from the interior strong parabolic maximum principle that $\overline{u}>\underline{u}$ in $(0,T)\times I_i$. If the interval $(a,b)$ reduces to a single patch (that is, $n=1$), then we are done. Otherwise, either~$x_{i-1}$ or~$x_i$ belongs to the open interval $(a,b)$. Let us consider the case when $x_i\in(a,b)$ (hence,~$i\le n-1$). We now claim that $\overline{u}(t,x_i)>\underline{u}(t,x_i)$ for all $t\in(0,T)$. Indeed, otherwise, there is a time $t_0\in(0,T)$ such that $\overline{u}(t_0,x_i)=\underline{u}(t_0,x_i)$, and the Hopf lemma then implies that
$$\overline{u}_x(t_0,x_i^-)<\underline{u}_x(t_0,x_i^-).$$
But $\overline{u}_x(t_0,x_i^+)\ge\underline{u}_x(t_0,x_i^+)$ since $\overline{u}\ge\underline{u}$ in $[0,T)\times\overline{I_{i+1}}$ and $\overline{u}(t_0,x_i)=\underline{u}(t_0,x_i)$. One finally gets a contradiction with the assumptions on the spatial derivatives of the super- and subsolutions~$\overline{u}$ and~$\underline{u}$ at $x_i^\pm$. Therefore, $\overline{u}(t,x_i)>\underline{u}(t,x_i)$ for all $t\in(0,T)$. By continuity and by applying the strong interior parabolic maximum principle in $(0,T)\times I_{i+1}$, we infer that $\overline{u}>\underline{u}$ in $(0,T)\times I_{i+1}$. By an immediate induction, going from one patch to the adjacent one in the left or right directions, we get that $\overline{u}>\underline{u}$ in $(0,T)\times(a,b)$. The proof of Proposition~\ref{bdd-cp} is thereby complete.
\end{proof}
 
Then we  prove in Proposition~\ref{R-cp} the comparison principle when $(a,b)=\mathbb{R}$, still in the case of a finite number of interfaces (the case when the domain is of the form $(a,+\infty)$ with $a\in\R$, or~$(-\infty,b)$ with $b\in\R$, can be handled by a combination and a slight modification of the proofs of Propositions~\ref{bdd-cp} and~\ref{R-cp}). 
 
\begin{proposition}[Comparison principle in $\mathbb{R}$ with finitely many interfaces]
\label{R-cp}
For $T\in(0,+\infty]$, let $\overline{u}$ and $\underline{u}$ be, respectively, a super- and a subsolution in~$[0,T)\times\R$ of $\mathcal{L}u=0$ with~\eqref{app-interface}, and assume that $\overline{u}(0,\cdot)\ge\underline{u}(0,\cdot)$ in~$\R$. Then, $\overline{u}\ge\underline{u}$ in~$[0,T)\times\R$ and, if $\overline{u}(0,\cdot)\not\equiv\underline{u}(0,\cdot)$, then $\overline{u}>\underline{u}$ in $(0,T)\times\R$.
\end{proposition}

\begin{proof}
Fix any $T_0\in(0,T)$ and define the nonnegative real numbers $M$ and $\mu$ as in~\eqref{defmu} with this time $\R$ instead of $[a,b]$ in the definition of $M$. Denote $w(t,x):=(\overline{u}(t,x)-\underline{u}(t,x))e^{-\mu t}$ for~$(t,x)\in[0,T_0]\times\R$. The function $w$ is continuous and bounded in $[0,T_0]\times\R$, with restriction in $(0,T_0]\times\overline{I_i}$ of class~$C^{1;2}_{t;x}((0,T_0]\times\overline{I_i})$ for each $1\le i\le n$ (notice that, here, $I_1=(-\infty,x_1)$ and $I_n=(x_{n-1},+\infty)$ are unbounded), and $w$ still satisfies~\eqref{w1}--\eqref{w2}, together with $w(0,\cdot)=\overline{u}(0,\cdot)-\underline{u}(0,\cdot)\ge 0$ in~$\R$. Set now $R=\max_{1\le i\le n-1}|x_i|+1>0$, and let $\varrho:\R\to\R$ be a nonnegative $C^2$ function with bounded first and second order derivatives, and satisfying
$$\left\{\baa{l}
\displaystyle\varrho=0\hbox{ in }[-R,R],\ \ \lim_{x\to+\infty}\varrho(x)=+\infty,\vspace{3pt}\\
\displaystyle\Big(\max_{1\le i\le n}d_i\Big)\times\|\varrho''\|_{L^\infty(\R)}+\|c\|_{L^\infty((0,T_0]\times\cup_{i=1}^nI_i)}\times\|\varrho'\|_{L^\infty(\R)}\le\frac12.\eaa\right.$$
Let us consider an arbitrary $\varep>0$, and introduce an auxiliary function $z$ defined by
$$z(t,x):=w(t,x)+\varep(\varrho(|x|)+t+1)\ \hbox{ for $(t,x)\in[0,T_0]\times\R$}.$$
The function $z$ has at least the same regularity as $w$, while $z(0,x)\ge\varep>0$ for all $x\in\R$ and $z(t,x)\to+\infty$ as $|x|\to+\infty$ uniformly in $t\in[0,T_0]$. Moreover,
$$\mathcal{N}z\ge\mathcal{N}w+\varep-\varep d(x)\varrho''(|x|)-\varep|c(t,x)\varrho'(|x|)|+ \big(\mu-F_s(x,\eta(t,x))\big)\varep(\varrho(|x|)+t+1)\ge\frac{\varep}{2}>0$$
for $(t,x)\in(0,T_0]\times\bigcup\limits_{i=1}^n I_i$, and~\eqref{z} still holds from~\eqref{w2}, the definition of $R$ and the choice of~$\varrho$. We claim that $z(t,x)>0$ for all $(t,x)\in[0,T_0]\times\R$. Assume not. Then, by continuity and the above properties of $z$, there is a point $(t_0,y_0)\in(0,T_0]\times\R$ such that $z(t_0,y_0)=\min_{[0,t_0]\times\R}z=0$. If $y_0\in I_i$ for some $1\le i\le n$, then we see as in~\eqref{Nz} that $\mathcal{N}z(t_0,y_0)\le 0$, which is impossible. Thus, we can assume without loss of generality that $y_0=x_i$ for some $1\le i\le n-1$ and that $z>0$ in $[0,t_0]\times\cup_{i=1}^nI_i$. Then, the Hopf lemma yields $z_x(t_0,x_i^-)<0$ and $z_x(t_1,x_i^+)>0$, contradicting~\eqref{z}. Consequently, $z>0$ in $[0,T_0]\times\R$. Hence, by passing to the limit as $\varep\to0^+$, we infer that $w\ge 0$ in $[0,T_0]\times\R$, that is, $\overline{u}\ge\underline{u}$ in $[0,T_0]\times\R$, and then $\overline{u}\ge\underline{u}$ in $[0,T)\times\R$ owing to the arbitrariness of $T_0\in(0,T)$.

Lastly, if one further assumes that $\overline{u}(0,\cdot)\not\equiv\underline{u}(0,\cdot)$, then the proof of the strict inequality $\overline{u}>\underline{u}$ in $(0,T)\times\mathbb{R}$ follows similar lines as in the proof of the preceding proposition.
\end{proof}

The last statement is a comparison principle for a class, more general than~\eqref{m-pb}--\eqref{d-f}, of non-periodic problems involving countably many interfaces. Namely, we are given a countable set $S=\{x_i:i\in\Z\}\subset\R$ with
\be\label{defdelta}
\delta:=\inf_{i\in\Z}\,(x_{i+1}-x_i)>0,
\ee
and we consider the problem
\begin{equation}
\label{m-pb-bis}
\left\{\baa{rcll}
u_t-d(x)u_{xx}-c(t,x)u_x & = & F(x,u), & t>0,\ x\in\mathbb{R}\!\setminus\!S,\vspace{3pt}\\
u(t,x_i^-) & = & u(t,x_i^+), & t> 0,\ i\in\Z,\vspace{3pt}\\
u_x(t,x_i^-) & = & \sigma_i u_x(t,x_i^+), & t> 0,\ i\in\Z.\eaa\right.
\end{equation}
We assume that the function $x\mapsto d(x)$ is equal to a positive constant $d_i$ in each interval~$(x_i,x_{i+1})$, and that $\sup_{i\in\Z}d_i\!<\!+\infty$. The function $c$ is assumed to be continuous and bounded in $(0,T_0)\times(\R\!\setminus\!S)$ for every $T_0\in(0,+\infty)$, the $\sigma_i$'s are given positive real numbers, and there are~$C^1(\R)$ functions~$(f_i)_{i\in\Z}$ such that $F(x,s)=f_i(s)$ for every $(x,s)\in(x_i,x_{i+1})\times\R$ and $i\in\Z$, with $\sup_{i\in\Z}\|f'_i\|_{L^\infty([-L,L])}\!<\!+\infty$ for every $L>0$.

For $T\in(0,+\infty]$, we say that a continuous function $\overline{u}:[0,T)\times\R\to\R$, which is assumed to be bounded in $[0,T_0]\times\R$ for every $T_0\in(0,T)$, is a supersolution of~\eqref{m-pb-bis} in~$[0,T)\times\R$, if, for every $i\in\Z$, the function $\overline{u}|_{(0,T)\times[x_i,x_{i+1}]}$ is of class $C^{1;2}_{t;x}((0,T)\times[x_i,x_{i+1}])$ and satisfies $u_t(t,x)-d_iu_{xx}(t,x)-c(t,x)u_x(t,x)\ge F(x,u(t,x))$ for every $(t,x)\in(0,T)\times(x_i,x_{i+1})$, and if $\overline{u}_x(t,x_i^-)\ge \sigma_i \overline{u}_x(t,x_i^+)$ for every~$i\in\Z$ and~$t\in(0,T)$. A subsolution is defined similarly with all the inequality signs reversed.

The following result provides a comparison between sub- and supersolutions of~\eqref{m-pb-bis} with ordered initial conditions, thus yielding the uniqueness of solutions for given initial conditions.

\begin{proposition}[Comparison principle for problems of type~\eqref{m-pb-bis}]
\label{cp}
For $T\in(0,+\infty]$, let $\overline{u}$ and $\underline{u}$ be, respectively, a super- and a subsolution  of~\eqref{m-pb-bis} in~$[0,T)\times\R$ with $\overline{u}(0,\cdot)\ge\underline{u}(0,\cdot)$ in~$\R$. Then, $\overline{u}\ge\underline{u}$ in~$[0,T)\times\R$, and, if $\overline{u}(0,\cdot)\not\equiv\underline{u}(0,\cdot)$, then $\overline{u}>\underline{u}$ in~$(0,T)\times\R$.
\end{proposition}

\begin{proof}
Fix any $T_0\in(0,T)$. Define
$$M\!:=\!\max\!\big(\|\overline{u}\|_{L^\infty([0,T_0]\times\R)},\|\underline{u}\|_{L^\infty([0,T_0]\times\R)}\big)\hbox{ and }\mu\!:=\!\sup_{i\in\Z}\|f'_i\|_{L^\infty([-M,M])}\!=\!\!\sup_{x\in\R\setminus S,\,|s|\le M}\!|F_s(x,s)|,$$
which are two nonnegative real numbers. Denote $w(t,x):=(\overline{u}(t,x)-\underline{u}(t,x))e^{-\mu t}$ for~$(t,x)\in[0,T_0]\times\R$. The function $w$ is continuous and bounded in $[0,T_0]\times\R$, and it still satisfies inequalities similar to~\eqref{w1} (with, here, $\cup_{i=1}^nI_i$ replaced by $\R\setminus S$), together with $w_x(t,x_i^-)\ge\sigma_iw_x(t,x_i^+)$ for every $i\in\Z$ and $t\in(0,T_0]$. Furthermore, $w(0,\cdot)=\overline{u}(0,\cdot)-\underline{u}(0,\cdot)\ge 0$ in~$\R$.

Let now $(\rho_m)_{m\in\N}$ be a family of nonnegative $C^\infty(\R)$ mollifiers with unit mass and such that each function $\rho_m$ has a support included in $[-1/m,1/m]$. Remember that $\delta>0$ is defined in~\eqref{defdelta}. With $\mathbbm{1}_E$ denoting the characteristic function of a set $E$, and $\star$ being the convolution product, we then define
$$F:=\mathop{\bigcup}_{i\in\Z}\,\Big(x_i+\frac{3\delta}{10},x_{i+1}-\frac{3\delta}{10}\Big)$$
and
$$\phi:=\rho_m\star\big(-\mathbbm{1}_{F\cap(-\infty,-1/2)}+\mathbbm{1}_{F\cap(1/2,+\infty)}\big),$$
with a certain $m$ large enough so that the $C^\infty(\R)$ function $\phi$ satisfies $\phi\le0$ in $(-\infty,0]$, $\phi\ge0$ in~$[0,+\infty)$, $\phi=-1$ in $[x_i+2\delta/5,x_{i+1}-2\delta/5]\cap(-\infty,-1]$, $\phi=1$ in $[x_i+2\delta/5,x_{i+1}-2\delta/5]\cap[1,+\infty)$ and $\phi=0$ in $[x_i-\delta/5,x_i+\delta/5]$, for all $i\in\Z$. Notice that $-1\le\phi\le1$ in $\R$ and that $\phi'$ is bounded in $\R$. Let us then define
$$\varrho(x):=\int_0^x\phi(y)\,\mathrm{d}y$$
for $x\in\R$. The $C^\infty(\R)$ function $\varrho$ is nonnegative, it has bounded first and second order derivatives, and $\varrho(x)\to+\infty$ as $x\to\pm\infty$. There is then a positive real number $\kappa>0$ such that
$$\Big(\sup_{i\in\Z}d_i\Big)\times\|\kappa\varrho''\|_{L^\infty(\R)}+\|c\|_{L^\infty((0,T_0]\times(\R\setminus S))}\times\|\kappa\varrho'\|_{L^\infty(\R)}\le\frac12.$$

Let us then consider an arbitrary $\varep>0$, and introduce an auxiliary function $z$ defined by
$$z(t,x):=w(t,x)+\varep(\kappa\varrho(x)+t+1)\ \hbox{ for $(t,x)\in[0,T_0]\times\R$}.$$
The function $z$ is continuous in $[0,T_0]\times\R$, and it satisfies
\be\label{z3}
z(0,x)\ge\varep>0\hbox{ for all $x\in\R$, and $z(t,x)\to+\infty$ as $|x|\to+\infty$ uniformly in $t\in[0,T_0]$},
\ee
since $w$ is bounded in $[0,T_0]\times\R$ and $\varrho(\pm\infty)=+\infty$. Moreover, with the same notations as in~\eqref{w1}, one has
$$\baa{rcl}
\mathcal{N}z(t,x) & \!\!\!=\!\!\! & \displaystyle\underbrace{\mathcal{N}w(t,x)}_{\ge0}+\varep-\underbrace{(\varep d(x)\kappa\varrho''(x)\!+\!\varep c(t,x)\kappa\varrho'(x))}_{\le\varep/2}+\underbrace{(\mu\!-\!F_s(x,\eta(t,x)))}_{\ge0}\underbrace{\varep(\varrho(x)+t+1)}_{\ge0}\vspace{3pt}\\
& \!\!\!\ge\!\!\! & \displaystyle\frac{\varep}{2}>0\eaa$$
for all $(t,x)\in(0,T_0]\times(\R\!\setminus\!S)$, while $z_x(t,x_i^-)\ge\sigma_iz_x(t,x_i^+)$ for all $t\in(0,T_0]$ and $i\in\Z$ (since $w$ satisfies these inequalities and $\varrho'(x_i)=\phi(x_i)=0$ for each $i\in\Z$). 

We claim that $z>0$ in $[0,T_0]\times\R$. Assume not. Then, by continuity and~\eqref{z3}, there is~$(t_0,y_0)\in(0,T_0]\times\R$ such that $z(t_0,y_0)=\min_{[0,t_0]\times\R}z=0$. If $y_0\in(x_i,x_{i+1})$ for some~$i\in\Z$, then we see as in~\eqref{Nz} that $\mathcal{N}z(t_0,y_0)\le0$, which is impossible. Thus, one can assume without loss of generality that $y_0=x_i$ for some $i\in\Z$ and that $z>0$ in $[0,t_0]\times(\R\setminus S)$, whence $z_x(t_0,x_i^-)<0$ and $z_x(t_0,x_i^+)>0$ from the Hopf lemma, which is again impossible. As a consequence, $z>0$ in $[0,T_0]\times\R$, hence $w\ge0$ in~$[0,T_0]\times\R$ due to the arbitrariness of $\varep>0$, and finally $\overline{u}\ge\underline{u}$ in $[0,T)\times\R$ due to the arbitrariness of $T_0\in(0,T)$.

Lastly, if one further assumes that $\overline{u}(0,\cdot)\not\equiv\underline{u}(0,\cdot)$ in $\R$, then one concludes as in the proof of Proposition~\ref{bdd-cp} that $\overline{u}>\underline{u}$ in $(0,T)\times\mathbb{R}$.
\end{proof}

\section*{Acknowledgements}
This work has been carried out in the framework of the A*MIDEX project (ANR-11-IDEX-0001-02), funded by the ``Investissements d'Avenir" French Government program managed by the French National Research Agency (ANR). The research leading to these results has also received funding from ANR project RESISTE (ANR-18-CE45-0019). M.~Zhang acknowledges the China Scholarship Council for the two-year financial support  during her study at Aix-Marseille Universit\'{e}. This work was initiated while F.~Lutscher held the position of Professeur Invit\'e at Aix-Marseille Universit\'{e}. The authors are grateful for this opportunity and the financial support from Aix-Marseille Universit\'{e}. The authors would also like to thank Prof. Xing Liang and Prof. Xiao-Qiang Zhao for many helpful discussions. They are also grateful to an anonymous referee for his/her valuable comments, which enabled us to improve the original manuscript.
 

\end{document}